\documentclass[11pt,reqno]{amsart}
\usepackage{graphicx}
\usepackage{amssymb}
\usepackage{epstopdf}
\usepackage{tikz}
\usepackage{microtype}
\usepackage{hyperref,xcolor}
\usepackage[utf8x]{inputenc}
\usepackage{sseq}
\usepackage{fixme}
\usepackage[shortlabels]{enumitem}
\usepackage{microtype}
\usepackage{tikz-cd}
\usepackage{nicefrac}
\usepackage{bbm}
\usepackage{adjustbox}
\usepackage{mathrsfs}

\newcommand\scalemath[2]{\scalebox{#1}{\mbox{\ensuremath{\displaystyle #2}}}}

\calclayout

\newtheorem{theorem}{Theorem}[section]
\newtheorem{innercustomthm}{Theorem}
\newenvironment{customthm}[1]
  {\renewcommand\theinnercustomthm{#1}\innercustomthm}
  {\endinnercustomthm}
\newcommand{\Spec}{\text{Spec }}
\newtheorem{conj}[theorem]{Conjecture}
\newtheorem{innercustomconj}{Conjecture}
\newenvironment{customconj}[1]
  {\renewcommand\theinnercustomconj{#1}\innercustomconj}
  {\endinnercustomconj}
\newtheorem{corollary}[theorem]{Corollary}

\newtheorem{lemma}[theorem]{Lemma}
\newtheorem{proposition}[theorem]{Proposition}
\newtheorem{innercustomprop}{Proposition}
\newenvironment{customprop}[1]
  {\renewcommand\theinnercustomprop{#1}\innercustomprop}
  {\endinnercustomprop}

\theoremstyle{remark}
\newtheorem{remark}[theorem]{Remark}

\DeclareMathOperator{\id}{id}
\DeclareMathOperator{\pr}{pr}

\DeclareMathOperator{\Br}{Br}

\DeclareMathOperator{\tr}{tr}
\DeclareMathOperator{\Gal}{Gal}
\DeclareMathOperator{\ZM}{ZM}
\DeclareMathOperator{\ZT}{ZT}

\DeclareMathOperator{\rk}{rk}
\DeclareMathOperator{\Hom}{Hom}

\DeclareMathOperator{\colim}{colim}
\DeclareMathOperator{\Ext}{Ext}
\DeclareMathOperator{\Gen}{Gen}
\DeclareMathOperator{\Et}{\acute{E}t}

\DeclareMathOperator{\Map}{Map}
\DeclareMathOperator{\ev}{ev}

\newcommand{\OO}{{\mathcal O}}

\DeclareMathOperator{\Cl}{Cl}

\DeclareMathOperator{\im}{im}

\DeclareMathOperator{\divis}{div}
\DeclareMathOperator{\Div}{Div}

\DeclareMathOperator{\Sh}{Sh}
\DeclareMathOperator{\fl}{fl}
\newcommand{\et}{\textsf{\'et}}
\DeclareMathOperator{\DIV}{\mathscr{D}\text{\kern -1pt {\emph{iv}}}}
\DeclareMathOperator{\HOM}{\mathscr{H}\text{\kern -2.5pt {\emph{om}}}}

\newcommand{\cupp}{\smallsmile}

\fxsetup{status=draft}

\newcommand{\GG}{{\mathbb G}}
\newcommand{\FF}{{\mathbb F}}

\newcommand{\ZZ}{{\mathbb Z}}
\newcommand{\LL}{{\mathbb L}}
\newcommand{\QQ}{{\mathbb Q}}
\newcommand{\NN}{{\mathbb N}}

\newcommand{\cC}{{\mathcal C}}

\newcommand{\cE}{{\mathcal E}}

\usepackage{tikz-cd}

\newcommand{\p}{{\mathfrak{p}}}
\newcommand{\q}{{\mathfrak{q}}}

\newcommand{\mmu}{\pmb{\mu}}

\begingroup
\catcode`\&=13
\gdef\pampmatrix{%
  \begingroup
  \let&=\amsamp
  \begin{pmatrix}%
}
\gdef\endpampmatrix{\end{pmatrix}\endgroup}
\endgroup

\begin{document}

\title{Massey products in the \'etale cohomology of number fields}

\author{Eric Ahlqvist}
\address{Department of mathematics, Stockholms universitet\\
106 91 Stockholm, Sweden}
\email{eric.ahlqvist@math.su.se}

\author{Magnus Carlson}
\address{Institut für Mathematik, 
Johann Wolfgang Goethe-Universität, 
Robert-Mayer-Str. 6-8, 
D-60325 Frankfurt am Main, 
Germany}
\email{carlson@math.uni-frankfurt.de}

\begin{abstract}
    We give formulas for 3-fold Massey products in the \'etale cohomology of the ring of integers of a number field and use these to find the first known examples of imaginary quadratic fields with class group of $p$-rank two possessing an infinite $p$-class field tower, where $p$ is an odd prime. 
    Furthermore, we provide a necessary and sufficient condition, in terms of class groups of p-extensions, for the vanishing of 3-fold Massey products. As a consequence, we give an elementary and sufficient condition for the infinitude of class field towers of imaginary quadratic fields.     
    We also disprove McLeman's $(3,3)$-conjecture.  
    Lastly, we relate the vanishing of Massey products to the existence of Galois representations of $G_{\QQ,S}$ which realize an unexpectedly large class group for certain extensions of a quadratic imaginary number field. 
\end{abstract}

\thanks{The first author was supported by the Swedish Research Council 2015-05554 and the Knut and Alice Wallenberg Foundation 2021.0279. The second author was sponsored by the Knut and Alice Wallenberg Foundation 2017.0400}

\maketitle
\tableofcontents

\section{Introduction}
In this paper, we give formulas for $3$-fold Massey products in the \'etale cohomology of the ring of integers of a number field. We apply these formulas to disprove a conjecture of McLeman on class field towers of quadratic imaginary number fields, and we give the first known examples of quadratic imaginary number fields with class group of $p$-rank two that possess an infinite Hilbert $p$-class field tower. We also give an elementary sufficient condition for the infinitude of $p$-class field towers of imaginary quadratic fields, including the case when the class group is of $p$-rank 2. 

\subsection{Class field towers} \label{subseq:classfield}

Given a number field $K$, the \emph{class field tower} is the tower 
    \[
        K_0 = K \subset K_1 \subset K_2 \cdots\,,
    \] 
in which each $K_i$ is the Hilbert class field of $K_{i-1}$. In the early 20th century, Furtw\"{a}ngler posed the question of whether this tower could be infinite or if it must eventually become constant after a finite number of steps. The latter case occurs if and only if $K$ can be embedded into a a number field $L$ such that unique factorization holds in the ring of integers of $L$ (see, for example, \cite{Roquette}).  

Various approaches to Furtw\"{a}ngler's question were suggested: Artin proposed strengthening the Minkowski bounds. Meanwhile, several prominent mathematicians believed that a group-theoretical method could not give any insight \cite[p. 8]{LemmermeyerClass}. However, Shafarevich and his student Golod proved \cite{Golod--Shafarevich} \footnote{See \cite[Postscript]{KochGalois} for more details regarding the history of the Golod--Shafarevich inequality.}, using group theory, that the tower could be infinite. A first simplifying step is to pass to the \emph{Hilbert $p$-class field tower}, where $K_i$ (for $i \geq 1$) is the $p$-Hilbert class field of $K_{i-1}$. Each $K_i$ is Galois over $K$, and we define $G_K^{ur,p}$ as the Galois group of this tower; by construction, $G_K^{ur,p}$ is a pro-$p$-group.

For a general pro-$p$-group $G$, let $d = \dim_{\FF_p} H^1(G,\ZZ/p\ZZ)$ and $r = \dim_{\FF_p} H^2(G,\ZZ/p\ZZ)$. We note that $d$ is the minimum number of generators of $G$ (as a profinite group) and $r$ is the minimal number of relations in a presentation of $G$. If $G = G_K^{ur,p}$, then $d$ is the $p$-rank of the class group of $K$, i.e. equal to $\dim_{\FF_p} \Cl K/p$. Golod--Shafarevich proved that if $G$ is finite, then $r > (d-1)^2/4$ \cite{Golod--Shafarevich}. This inequality was subsequently improved by Vinberg--Gash\"utz to $r > d^2/4$ \cite{Vinberg}. Applying either of these bounds, it is easy to construct number fields with infinite $p$-class tower.

Naturally, the question arises: \emph{Given a number field $K$, is the $p$-class field tower of $K$ infinite or not?} A first step is to classify the situation when $K$ is quadratic imaginary and $p$ is odd. In this case $r=d$ and, by the above inequalities, the class field tower is infinite if $d \geq 4$. Subsequent work by Koch--Venkov \cite{Koch--Venkov}, proved that if $d \geq 3$, then the tower is infinite. On the other hand, the tower is finite if $d=1$. This leaves open the case when $d=2$. Let us note that this is not just of idle interest; the unramified Fontaine--Mazur conjecture \cite[Conjecture 5B]{FontaineMazur} predicts that $G_K^{ur,p}$ has no infinite $p$-adic analytic quotients. If $G_K^{ur,p}$ satisfies $r > d^2/4$, then the unramified Fontaine--Mazur is true for $G_K^{ur,p}$. Thus, only when $d=2$ can there potentially be counterexamples to the unramified Fontaine--Mazur conjecture (a related discussion can be found in \cite[p. 705]{NeukirchCohomology}).

In \cite{Mcleman}, McLeman made several conjectures regarding $p$-class field towers of imaginary quadratic fields $K$ in the case when $d = 2$ and $p$ is odd. To explain his conjectures, we start by pointing out that $G_K^{ur,p}$ then has a presentation (as a pro-$p$-group) with $2$ generators and $2$ relations. There is a refined version of the Golod--Shafarevich inequality which takes into account the \emph{depth} of the relations \cite[Theorem 7.20]{KochGalois}, with respect to the so called Zassenhaus filtration. By work of Koch--Venkov \cite{Koch--Venkov}, the two relations $r_1,r_2$ lie in depth greater or equal to $3$ and must lie in depth of odd degree. One can further show that for $G^{ur,p}$ to be finite, the depth must lie in bidegrees $(3,3)$, $(3,5)$ or $(3,7)$. McLeman proposed the following \emph{$(3,3)$-conjecture}, based on numerical data:

\begin{customconj}{\ref{conj:3,3}}[McLeman]
    Let $p > 3$ be a prime number and suppose that $K$ is an imaginary quadratic field with class group of $p$-rank two. Then the $p$-class field tower of $K$ is finite if and only if the relations lie in bidegree $(3,3)$ with respect to the Zassenhaus filtration. 
\end{customconj}

McLeman also emphasized that if one could show that the relations $r_1,r_2$ both lie in bidegree greater than or equal to $5$, then the $p$-class field tower of $K$ is infinite, and one would thus produce the first examples of infinite towers in the case when $K$ is imaginary quadratic and $H^1(G_K^{ur,p},\ZZ/p\ZZ)$ has two generators. McLeman also notes that the depths of the relations depend on \emph{Massey products} in $H^*(G_K^{ur,p},\ZZ/p\ZZ)$. 

\subsection{Massey products}
In the paper \cite{Hopkins--Wickelgren}, Hopkins and Wickelgren studied Massey products in the Galois cohomology of global fields. They showed that for a global field $K$ of characteristic not equal to two, and for classes $x,y,z \in H^1(G_K,\ZZ/2\ZZ)$ in the Galois cohomology of $K$, any Massey product $\langle x,y,z\rangle \subset H^2(G_K,\ZZ/2\ZZ)$, if defined, contains zero. Shortly thereafter, this result was generalized to arbitrary fields and arbitrary prime coefficients by Min\'{a}\v{c}--T\^{a}n \cite{Minac-Tan} and Efrat--Matzri \cite{Efrat-Matzri}. Moreover, Min\'{a}\v{c}--T\^{a}n also formulated the $n$-fold Massey vanishing conjecture which states the following: For any field $K$, any prime $p$, and any classes $x_1, \ldots , x_n \in H^1(G_K,\ZZ/p\ZZ)$, for $n\geq 3$, the Massey product $\langle x_1, \ldots, x_n \rangle$, whenever defined, contains $0$ (See \cite[Conjecture 1.6]{Minac-Tan} and \cite[Conjecture 1.1]{Minac-Tan-2}). In \cite{Minac-Tan} the conjecture was formulated in the case of the base field containing a primitive $p$th-root of unity. In \cite{Minac-Tan-2} the conjecture is formulated in full generality when the condition on the $p$th-root of unity in the base field is removed.
The case $n=4$ and $p=2$ has recently been resolved by Merkurjev--Scavia \cite{MerkurjevScavia}. For number fields, the $n$-fold Massey vanishing conjecture was proven by Harpaz--Wittenberg \cite{HarpazWittenberg}, using local-global principles for certain homogeneous spaces.

Let $K$ be a number field, $S$ a finite set of places of $K$, and $G_{K,S}$ the Galois group of the maximal extension of $K$ unramified outside of $S$. In this setiing, Massey products in $H^*(G_{K,S},\ZZ/p\ZZ)$ do not vanish in general. Morishita \cite{Morishita-Massey} calculated certain Massey products for $K = \mathbb{Q}$, in the case when $S$ a finite set of primes satisfying certain congruence conditions, in terms of ``arithmetic Milnor invariants''. He further proved that for $p=2$, the Massey products he computed could be interpreted as R\'edei symbols. In \cite{Vogel-Massey}, Vogel also expressed Massey products in terms of Milnor invariants and used these calculation to describe the relation structure of the maximal pro-$2$-quotient of $G_{\QQ,S}$, modulo the fourth level of the Zassenhaus filtration.

A further advancement was made in \cite{Sharifi-Massey}, where Sharifi gives formulas for Massey products of the form $\langle x,x, \dots , x,y \rangle$ for $x$ and $y$ elements of $H^1(G_{K,S},\ZZ/p\ZZ)$, assuming $K$ is a number field which contains all $p$th roots of unity and $S$ includes all primes above $p$. He further demonstrated that the vanishing of Massey products is tightly related to the structure of certain fundamental modules appearing in Iwasawa theory. 

More recently, in \cite{Sharifi-Bockstein}, Massey products were related to certain Bockstein maps. Using this interpretation, lower bounds were established for the $p$-rank of the class group of certain non-abelian extensions of $\QQ$.



\subsection{Overview of main results}
\subsubsection*{The class field tower problem}

Our first contribution is a disproof of the $(3,3)$-conjecture (Conjecture  \ref{conj:3,3}) of McLeman.

\begin{customthm}{\ref{thm:false}}
    The $(3,3)$-conjecture is false. For instance, for $p$ a prime and $D$ a discriminant, the pairs  
        \[
            (p,D) = (5, -90868), (7, -159592)  
        \] 
    are counterexamples to the $(3,3)$-conjecture: for each pair in the list above, the associated quadratic imaginary field with discriminant $D$ has Zassenhaus type $(5,3)$ or $(7,3)$, but the $p$-class field tower is finite.
\end{customthm}

Our second contribution is the first known examples of infinite $p$-class field towers in the case when $K$ is imaginary quadratic, $p$ is odd, and $d = \dim_{\FF_p} H^1(G_K^{ur},\ZZ/p\ZZ) = \dim_{\FF_p} \Cl K /p \Cl K$ is $2$. 

\begin{customthm}{\ref{thm:infinite}}
    There exist odd prime numbers $p$ and imaginary quadratic fields of discriminant $D$ with class group of $p$-rank two and infinite $p$-class field tower. For instance, for prime $p$ and discriminant $D$, the pairs
        \[
            \begin{split}
                (p,D) = &\  (3, -3826859), (3, -8187139), (3, -11394591), (3, -13014563)\,, \\
                        &\  (5, -2724783), (5, -4190583), (5, -6741407), (5, -6965663) \,
            \end{split}
        \]
    give examples of such fields: for each pair in the above list, the associated quadratic imaginary field with discriminant $D$ has an infinite $p$-class field tower.
\end{customthm}

We prove both Theorem \ref{thm:false} and Theorem \ref{thm:infinite} via computations of Massey products in Galois cohomology. As pointed out in Subsection \ref{subseq:classfield}, $G_K^{ur,p}$ has a presentation as a pro-$p$-group with $2$ generators and $2$ relations, and the relations lie in depth greater than or equal to $3$, with respect to the Zassenhaus filtration. The latter follows from the vanishing of the cup product $H^1(G_K^{ur},\ZZ/p\ZZ) \otimes H^1(G_K^{ur},\ZZ/p\ZZ) \to H^2(G_K^{ur},\ZZ/p\ZZ)$ \cite{Koch--Venkov}. McLeman noted in \cite{Mcleman} that the relations both being in degree $3$ is equivalent to the fact that if $x,y \in H^1(G_K^{ur},\ZZ/p\ZZ)$ are the two generators, then the Massey products 
\[
    \langle x,x,y \rangle, \langle y,y,x \rangle 
\] 
are linearly independent. Further, McLeman proved that if both $\langle x,x,y \rangle$ and $\langle y,y,x \rangle$ vanish (for $p=3$ one also needs $\langle x,x,x \rangle$ and $\langle y,y,y \rangle$ to be zero), then the $p$-class field tower is infinite. Thus, to prove Theorem \ref{thm:infinite}, it is enough to construct totally imaginary number field $K$ with class group of $p$-rank two, for $p$ odd, and such that certain threefold Massey products vanish. For the disproof of McLeman's $(3,3)$-conjecture, it suffices to construct such $K$ where the class field tower is finite, with $\langle x,x,y\rangle, \langle y,y,x \rangle$ non-vanishing but linearly dependent. 

Using our implementation of the Massey product formula (Theorem \ref{thm:main-simple}) in a C program, leveraging the PARI library \cite{PARI2}, we computed Massey products for thousands of imaginary quadratic fields. This enabled us to identify examples demonstrating Theorem \ref{thm:false} and Theorem \ref{thm:infinite}. The code and the data can be found in the repository \url{https://github.com/ericahlqvist/massey}. Theorem \ref{thm:false} also relies on earlier computations by Mayer \cite{Mayer-Finite}, which shows that the fields in Theorem \ref{thm:false} has finite class field tower.  

\subsubsection*{Main results on Massey products}
Our computation of Massey products proceeds in two steps: the first is group-cohomological, where we define a new secondary cohomology operation in Galois cohomology which in many situations coincides with Massey products. The second step is arithmetic, where we relate Massey products in Galois cohomology to Massey products in \'etale cohomology. Then we use Artin--Verdier duality and explicit resolutions to give formulas for Massey products in terms of arithmetic invariants. Recall that we may identify the \'etale fundamental group $\pi_1(X)$ (based at a geometric point $\bar{x}$, whose image is the generic point of $X=\Spec\OO_K$) with the Galois group of the maximal unramified extension $K^{ur}/K$, more or less by definition. 

For the first step, let $G$ be a profinite group. Suppose that $x,y,z \in H^1(G,\ZZ/p\ZZ)$ are such that $x \cupp y = y \cupp z= 0$. There is then a $G$-module $V_y$ and an exact sequence 
    \[
        0 \rightarrow \ZZ/p\ZZ \rightarrow V_y \rightarrow \ZZ/p\ZZ\to 0
    \]
such that the connecting homomorphism 
    \[ 
        y \cupp (-) :  H^i(G,\ZZ/p\ZZ) \to H^{i+1}(G,\ZZ/p\ZZ)
    \]
is given by a cup product with $y$. Since $y \cupp x = y \cupp z =0$, one can find elements $w_x, w_z \in H^1(G,V_y)$ mapping to $x$  and $z$ respectively. We can then form the sum 
    \[ 
        x \cupp w_z - w_x \cupp z \,,
    \]  
and this is sent to zero under the map $H^2(G,V_y) \to H^2(G,\ZZ/p\ZZ)$, which implies that one can find an element $\alpha$ mapping to it under the canonical map $H^2(G,\ZZ/p\ZZ) \to H^2(G,V_y)$. We define $\langle x,y,z \rangle'$ to be the subset of $H^2(G,\ZZ/p\ZZ)$ obtained by varying $w_x,w_z$ and the element $\alpha$. 

\begin{customprop}{\ref{prop:coins}}
    Assume that $y\cupp H^1(G,\ZZ/p\ZZ)\subseteq x\cupp H^1(G,\ZZ/p\ZZ)+H^1(G,\ZZ/p\ZZ)\cupp z$. Then the subset $\langle x,y,z \rangle'$ of $H^2(G,\ZZ/p\ZZ)$ coincides with $\langle x,y,z \rangle$.
\end{customprop}

This formula simplifies considerably when $x=y$ and $p$ is odd:

\begin{customprop}{\ref{prop:half-cup}}
    Let $p$ be an odd prime. Suppose that $x\cupp x = x\cupp y = 0$ and choose $w_x$ and $w_y$ as above. Then we have an equality 
        \[
            x\cupp w_y -w_x\cupp y = \frac{1}{2}x\cupp w_y    
        \]
    in $H^2(G, V_x)$. Hence $\langle x,x,y \rangle$ consists of elements of the form $M+x \cupp \alpha + \beta \cupp y$, for $\alpha, \beta \in H^1(G,\ZZ/p\ZZ)$, where $M$ maps to $\frac{1}{2} x \cupp w_y$ under $H^2(G, \ZZ/p\ZZ)\to H^2(G, V_x)$.
\end{customprop}

We have now completed the first step, and defined a secondary cohomology operation which, in many situations coincides with Massey products. The second step towards our computation is arithmetic. First, we fix notation:
\begin{itemize}
    \item $K$ is a number field (assumed to be totally imaginary to simplify the exposition), 
    \item $\Cl(K)$ is the class group of $K$,
    \item $\Div(K)$ is the group of fractional ideals of $K$,
    \item $X=\Spec \OO_K$, where $\OO_K$ is the ring of integers in $K$, 
    \item $L_x$ is the unramified degree $p$ extension of $K$ corresponding to an element $x\in H^1(X, \ZZ/p\ZZ)$, 
    \item $N_x$ denotes both the norm $\Cl(L_x)\to \Cl(K)$ and the norm $L_x^\times\to K^\times$,      
    \item $\mmu_p$ denotes the sheaf of $p$th roots of unity on the big fppf site on $X$,  
\end{itemize} 
To give formulas for Massey products, we first apply Proposition \ref{prop:coins} and \ref{prop:half-cup} to $\pi_1(X)^{(p)} = \Gal(K^{ur,p}/K),$ the maximal pro-$p$ quotient of the \'etale fundamental group of $X$. 

Second, the Massey product $\langle x,y,z \rangle \subset H^2(\pi_1(X)^{(p)},\ZZ/p\ZZ)$ is identified, under the inclusion $H^2(\pi_1(X)^{(p)},\ZZ/p\ZZ) \rightarrow H^2(X,\ZZ/p\ZZ)$, with the Massey product $\langle x,y,z \rangle$, where now $x,y,z$ are viewed as elements of $H^1(X,\ZZ/p\ZZ)$ under the identification $H^1(\pi_1(X)^{(p)}, \ZZ/p\ZZ) = H^1(X,\ZZ/p\ZZ)$. By using certain resolutions of the \'etale sheaf $P_y$, which is the pullback of the $G$-module $V_y$ to the \'etale site of $X$, we can then give a formula for the Massey product 
    \[ 
        \langle x,y,z \rangle \subset H^2(X,\ZZ/p\ZZ)
    \] 
in favorable situations, using Artin--Verdier duality.

To state our formulas, we recall that by the duality theorem of Artin--Verdier \cite[Corollary 3.2]{MilneADT}, there is, for every $i$, a perfect pairing 
    \[
        \langle -, -\rangle \colon  H^i(X, \ZZ/p\ZZ)\times H^{3-i}(X, \mmu_p) \to H^3(X, \mmu_p)\cong \frac{1}{p}\ZZ/\ZZ\subset \QQ/\ZZ\,,
    \]
where $\mmu_p$-cohomology is always taken in the big fppf site (see Remark \ref{rk:et-vs-fppf}). When $i=1$, this pairing can be described explicitly via the Artin-symbol. 

Via Artin--Verdier duality, any representative of the subset $\langle x,y,z \rangle \subset H^2(X,\ZZ/p\ZZ)$ can be identified with a functional on $H^1(X,\mmu_p)$. The group $H^1(X, \mmu_p)$ can be decribed as $Z_1/B_1$, where $Z_1 = \{ (a, J)\in K^\times\oplus\Div(K) | \divis(a)+pI=0\}$ and $B_1 = \{(b^{-p}, \divis(b)) \in K^* \oplus \Div(K)|b \in K^*\}$. We may think of the Massey product as an element 
    $
        \langle x,y,z\rangle \in H^2(X,\ZZ/p\ZZ)/(H^1(X,\ZZ/p\ZZ)\cupp z+x\cupp H^1(X,\ZZ/p\ZZ))    
    $
which, under Artin--Verdier, corresponds to a map 
    $
        \ker c_x^\sim\cap \ker c_z^\sim \to \QQ/\ZZ \,,   
    $
where $c_x^\sim$ denotes the Pontryagin dual of the map $x\cupp (-)\colon H^1(X,\ZZ/p\ZZ)\to H^2(X,\ZZ/p\ZZ)$. We can now state our main formulas:


\begin{customthm}{\ref{cor:massey}}
Let $x,y,z\in H^1(X, \ZZ/p\ZZ)$ be non-zero elements such that $x\cupp y = y\cupp z = 0$, and assume further that $y\cupp H^1(X,\ZZ/p\ZZ)\subseteq x\cupp H^1(X,\ZZ/p\ZZ)+H^1(X,\ZZ/p\ZZ)\cupp z$. Then, for any $(a',J)\in \ker c_x^\sim\cap \ker c_z^\sim$,
    \begin{equation}
        \langle \langle x,y,z\rangle , (a',J)\rangle = \langle z, J_K^x\rangle+\langle x, J_K^z\rangle\,,
    \end{equation}
where $J_K^x$ and $J_K^z$ are fractional ideals of $K$, satisfying the list of relations given in Theorem \ref{thm:formula2.0}. 
\end{customthm}

This says that $\langle \langle x,y,z\rangle , (a',J)\rangle$ can be described via Artin symbols evaluated on elements $J_K^x$ and $J_K^z$, satisfying certain compatibility relations with respects to inclusions, norms, and other classical arithmetic operators taking place in the diagram 
    \[
        \begin{tikzcd}
            & L_{xy} & & L_{yz} & \\
            L_x\ar{ur} & & L_y\ar{ul}\ar{ur} & &  L_z\ar{ul}\\
            & & K\ar{ull}\ar{u}\ar{urr} & & 
        \end{tikzcd}\,.
    \]
When $K$ is an imaginary quadratic field, the subgroup $\ker c_x^\sim\cap \ker c_z^\sim\subseteq H^1(X, \mmu_p)$ will in fact be all of $H^1(X, \mmu_p)$. 
For Massey products of the form $\langle x,x,y\rangle$, with $p$ odd, the formula
simplifies:

\begin{customthm}{\ref{thm:main-simple}} 
    Let $L_x\supset K$ be an unramified extension of degree $p$ representing an element $x\in H^1(X,\ZZ/p\ZZ)$, where $p$ is an odd prime. Let $y\in H^1(X,\ZZ/p\ZZ)$ be an element such that $x\cupp y=0$. Then, for any $(a',J)\in \ker c_x^\sim\cap \ker c_y^\sim\subseteq H^1(X, \mmu_p)$, the equality
        \[
            \langle \langle x,x,y\rangle, (a',J)\rangle =
            \begin{cases} 
                \langle y, N_x(I')+J\rangle & \mbox{ if }p = 3\,, \\
                \langle y, N_x(I')\rangle & \mbox{ if }p > 3\,,
            \end{cases} 
        \]
    holds, where $(t, I')\in L_x^\times\oplus \Div(L_x)$ is any element satisfying the following two equalities:
    \begin{itemize}
        \item $(1-\sigma_x)^2I'+\divis(t)+i_x(J)=0$ and 
        \item $N_x(t)=a'$. 
    \end{itemize}
\end{customthm}
 


When $K$ is quadratic imaginary, we relate the vanishing of Massey products to the rank of class groups of extensions of $K$.

\begin{customthm}{\ref{thm:crit-for-massey-vanish}}
    Suppose that $K$ is imaginary quadratic, $p$ is an odd prime, and that $K$ has class group of $p$-rank 2. Choose elements $x$ and $y$ forming a basis for $H^1(X, \ZZ/p\ZZ)$ with corresponding field extensions $L_x$ and $L_y$. Then the Massey products $\langle x,x,x\rangle$ and $\langle x,x,y\rangle$ both vanish if and only if $\Cl(L_x)$ has $p$-rank at least 4. In particular, if both $L_x$ and $L_y$ have class groups of $p$-rank at least 4, then the $p$-class field tower of $K$ is infinite.  
\end{customthm}

Theorem \ref{thm:crit-for-massey-vanish} is proven via method inspired by \cite{Sharifi-Bockstein}, where Massey products are described as certain connecting homomorphisms. Let $G_K^{ur}$ be the maximal unramified pro-$p$-extension of $K$ and let $G$ be the Galois group of the extension $L_x$ corresponding to $x\in H^1(X,\ZZ/p\ZZ)=H^1(G_K^{ur},\ZZ/p\ZZ)$. Let $I\subseteq \ZZ/p\ZZ[G]$ be the augmentation ideal. Then we have for every $n\geq 1$, an exact sequence 
    \[
        0\to I^n/I^{n+1} \to \ZZ/p\ZZ[G]/I^{n+1}\to \ZZ/p\ZZ[G]/I^{n}\to 0  \,,  
    \]  
and the image of the corresponding connecting homomorphism is generated by elements in the $n$-fold Massey products $\langle x,x,\dots, x,x\rangle$ and $\langle x,x,\dots, x,y\rangle$. By using induction and controlling the images of these connecting homomorphisms, we show that the vanishing (non-vanishing) of 3-fold Massey products gives lower (upper) bounds for $\rk H^1(G_K^{ur}, \ZZ/p\ZZ[G])=\rk \Cl(L_x)\otimes \FF_p$. 

We also realize the lower bound by constructing a Galois representation $G_{\QQ,S}\to GL_5(\FF_p)$  
whose restriction to $G_{L_x, S}$ takes the form 
\begin{equation}
    \begin{pmatrix}
        1   & 0 & 0 & 0  & t_{y,3} \\
            & 1          & 0 & 0 & t_{y,2}\\
            &               & 1         & 0           & t_{y,1} \\
            &               &           & 1                      & y \\
            &               &           &                           & 1 
    \end{pmatrix}    \,.
\end{equation}  
The homomorphisms $y, t_{y,1}, t_{y,2}, t_{y,3}\colon G_{L_x, S}\to \ZZ/p\ZZ$ lift to linearly independent elements in \newline $H^1(G_{L_x}^{ur}, \ZZ/p\ZZ)\cong \Hom(\Cl(L_x), \ZZ/p\ZZ)$ forcing the $p$-rank of $\Cl(L_x)$ to be at least 4.

\subsection*{Roadmap}
In Section \ref{sec:galois}, we review Massey products for profinite groups and define a new secondary cohomology operation, which we relate to 3-fold Massey products. This secondary cohomology operation will be our main tool for computing Massey products. 

Section \ref{sec:etale}, relates Massey products in the \'etale cohomology of a connected scheme $X$ to those in the group cohomology of its \'etale fundamental group $\pi_1(X)$; all the results here are known. The results of this section show that the secondary cohomology operations of Section \ref{sec:galois} can be computed in \'etale cohomology. 

Section \ref{sec:massey}, the core of the paper, is split into three subsections. 
Readers only intereted in class field towers, and not general formulas for Massey products, can safely skip Subsection \ref{subseq:comp}. 
\begin{itemize}
    \item Subsection \ref{subsec:arithmetic} provides background on arithmetic duality theory. In particular, it describes the cohomology groups $H^i(X, \ZZ/p\ZZ)$ and $H^{j}(X, \mmu_p)$, and recalls the functoriality of the arithmetic duality, which is crucial for our computations. 
    \item Then follows the technical Subsection \ref{subseq:comp}, where we derive a formula for Massey products in terms of arithmetic invariants of the number field (Theorem \ref{thm:formula2.0} and Theorem \ref{thm:main-simple}). Here we make use of the description in Section \ref{sec:galois}. In particular, Theorem \ref{thm:formula2.0} depends heavily on Proposition \ref{prop:preim} and the simpler Theorem \ref{thm:main-simple} is directly related to Proposition \ref{prop:half-cup}. The computation relies on the arithmetic duality of Subsection \ref{subsec:arithmetic} and makes use of explicit resolutions found in Appendix \ref{section:resolutions}. 
    \item Subsection \ref{subsec:necessary} establishes a sufficient condition for the vanishing of Massey product for imaginary quadratic number fields of $p$-rank two, in terms of class groups of certain unramified extensions (Theorem \ref{thm:crit-for-massey-vanish}). 
\end{itemize} 

Finally, in Section \ref{sec:classfield} we apply the material of Section \ref{sec:massey} to construct examples of imaginary quadratic number fields with $p$-rank two with infinite $p$-class field tower (Theorem \ref{thm:infinite}), as well as counterexamples to McLeman's $(3,3)$-conjecture (Theorem \ref{thm:false}). More precisely, we implement the formula given in Theorem \ref{thm:main-simple}, using PARI, to compute Massey products for thousands of imaginary quadratic fields. After examining the computations, we found the examples proving Theorem \ref{thm:infinite} and Theorem \ref{thm:false}.  

\subsection*{Acknowledgements} The authors wish to express their gratitude to the anonymous referee for carefully reading of the article, and for the many helpful suggestions which greatly improved the article. We are further grateful to Daniel Mayer for sharing a list of imaginary quadratic fields that he suspected could have infinite $p$-class field tower. Two of these were indeed shown by the authors to have infinite $p$-class field tower. The authors would also wish to thank Bill Allombert for answering several technical questions regarding PARI. Finally, the authors are also grateful to the Knut and Alice Wallenberg foundation for their support.

\section{Massey products in Galois cohomology}\label{sec:galois}
Let $G$ be a profinite group and let $C^*(G,\ZZ/p\ZZ)$ be the continuous cochain complex on $G$ with values in $\ZZ/p\ZZ$.  We further assume that $x,y,z$ are elements in the continuous cohomology group $H^1(G, \ZZ/p\ZZ)$ satisfying $x\cupp y = y\cupp z = 0$. 
The goal of this section is to give formulas for the 3-fold Massey product $\langle x,y,z\rangle\subseteq H^2(G, \ZZ/p\ZZ)$ under the hypothesis that either $x=y$ or $y\cupp (-) = 0$. More precisely, the formula will be a sum of certain cup products. 

We will define Massey products as in \cite{DwyerMassey}. Since $x\cupp y=y\cupp z=0$, there exists $k_{xy}, k_{yz}\in C^1(G, \ZZ/p\ZZ)$ such that $dk_{xy}=-x\cupp y$ and $dk_{yz}=-y\cupp z$. Now let $U_4(\ZZ/p\ZZ)$ be the group of unipotent, upper-triangular $4 \times 4$-matrices over  $\ZZ/p\ZZ$. Then the data of two such trivializing cochains $k_{xy}$ and $k_{yz}$ is equivalent to the data of a homomorphism 
    \[
        \rho\colon G\to \overline{U_4}(\ZZ/p\ZZ)\,, g\mapsto \begin{pmatrix}
            1 & x(g) & k_{xy}(g) & * \\
            0 & 1 & y(g) & k_{yz}(g) \\
            0 & 0 & 1 & z(g) \\
            0 & 0 & 0 & 1
        \end{pmatrix}    
    \]
where $\overline{U_4}(\ZZ/p\ZZ)$ is the group $U_4(\ZZ/p\ZZ)$ modulo the center, which consists of unipotent matrices of the form 
    \[
        \begin{pmatrix}
            1 & 0 & 0 & a \\
            0 & 1 & 0 & 0 \\
            0 & 0 & 1 & 0 \\
            0 & 0 & 0 & 1
        \end{pmatrix}           
    \]
where $a\in \ZZ/p\ZZ$. 
Given such a homomorphism $\rho$ we define $\langle x,y,z \rangle_\rho\in H^2(G,\ZZ/p\ZZ)$ to be the cohomology class associated to the cocycle defined by the equality 
\[
    \langle x,y,z \rangle_\rho (g,h) = x(g)k_{yz}(h)+k_{xy}(g)z(h)\,.    
\]
We now define the Massey product $\langle x,y,z\rangle$ to be the set of all $\langle x,y,z \rangle_\rho$, where $\rho$ ranges over the set of homomorphisms $G\to \overline{U_4}(\ZZ/p\ZZ)$ as above. Given $\rho$ and $\rho'$, we have 
    \[
        \begin{split}
            (\langle x,y,z \rangle_\rho-\langle x,y,z \rangle_{\rho'})(g,h) & = x(g)(k_{yz}(h)-k'_{yz}(h))+(k_{xy}(g)-k'_{xy}(g))z(h) \\
            & = (x\cupp (k_{yz}-k'_{yz})+(k_{xy}-k'_{xy})\cupp z)(g,h)\,.
        \end{split}
    \]
Since $k_{yz}-k'_{yz}$ and $k_{xy}-k'_{xy}$ are cocycles, we see that the Massey product $\langle x,y,z\rangle$ is uniquely determined up to elements in the set $x\cupp H^1(G,\ZZ/p\ZZ)+H^1(G,\ZZ/p\ZZ)\cupp z$. 

In order to prove the next lemma, we recall that the cup product 
    \[
        y\cupp(-)\colon H^n(G,\ZZ/p\ZZ) \to H^{n+1}(G,\ZZ/p\ZZ)    
    \]
is the connecting homomorphism of the short exact sequence 
    \[
        0\to \ZZ/p\ZZ\to V_y \to \ZZ/p\ZZ \to 0    
    \]
where $V_y = (\ZZ/p\ZZ)^2$, with $G$-action given by
    \[
        g\mapsto \begin{pmatrix}
            1 & y(g) \\
            0 & 1
        \end{pmatrix}\,. 
    \]
For a proof of this fact, see \cite[Example 12.37]{GuillotLocal}.
\begin{remark}
If $p=2$ in the exact sequence above, and $y$ is non-zero, then $V_y$ is isomorphic to the induced module $\mathrm{Ind}^G_{\ker y} (\mathbb{F}_2)$, which by Shapiro's lemma implies that $H^n(G,V_y)$ is canonically isomorphic to $H^n(\ker y, \mathbb{F}_2)$ for all $n$.
\end{remark}
\begin{lemma}\label{lem:null-datum}
    The datum of a trivializing cochain $k_{yz}$ of $-y\cupp z$ is equivalent to a lift $w_z\in Z^1(G, V_y)$ of $z\in H^1(G,\ZZ/p\ZZ)$. 
\end{lemma}

\begin{proof}
    If $y\cupp z=0$ in cohomology we understand, from the long exact sequence associated to $0\to \ZZ/p\ZZ\to V_y \to \ZZ/p\ZZ \to 0$, that there exists a lift $w_z\in H^1(G, V_y)$ of $z$. Writing   
        \[
            w_z=\begin{pmatrix}
                t_z \\
                z
            \end{pmatrix}\,,    
        \]
    a calculation shows that $w_z\colon G\to V_y$ is a cocycle if and only if $t_z$ is a trivializing cochain of $-y\cupp z$. 
\end{proof}

We will now use Lemma \ref{lem:null-datum} to give a formula for the Massey product in terms of lifts $w_x, w_z$ as above. Note that we have a cup product 
    \[
        H^1(G,\ZZ/p\ZZ)\times H^1(G, V_y) \to H^2(G, V_y)    
    \]
and hence a well-defined element $x\cupp w_z-w_x\cupp z\in H^2(G,V_y)$ mapping to zero in $H^2(G,\ZZ/p\ZZ)$ under the map coming from the long exact sequence above. Indeed, that this element goes to zero follows from the naturality of the cup product. Thus there is an element $H^2(G, \ZZ/p\ZZ)$ mapping to $x\cupp w_z-w_x\cupp z$ and such an element is closely related to the Massey product $\langle x,y,z\rangle$ as the following proposition shows. 

\begin{proposition}\label{prop:preim}
    The image of the Massey product $\langle x,y,z\rangle_\rho$ under the map $H^2(G,\ZZ/p\ZZ)\to H^2(G, V_y)$ is given by 
        \[
            x\cupp w_z- w_x\cupp z\,,    
        \]
    where 
        \[
            w_x = \begin{pmatrix}
                t_x \\ x
            \end{pmatrix}\,, \quad 
            w_z = \begin{pmatrix}
                t_z \\ z
            \end{pmatrix}   
        \]
    and 
        \[
            \rho(g) = \begin{pmatrix}
                1 & x(g) & -t_x(g)+x(g)y(g) & * \\
                0 & 1 & y(g) & t_{z}(g) \\
                0 & 0 & 1 & z(g) \\
                0 & 0 & 0 & 1
            \end{pmatrix} \,.    
        \]
\end{proposition}

\begin{proof}
Let $k_{yz}=t_z$ and $k_{xy}=-t_x+xy$. Then we get 
    \[
        \langle x,y,z \rangle (g,h) = x(g)t_z(h)-t_x(g)z(h)+x(g)y(g)z(h)   
    \]
and since 
    \[
        \begin{split}
            (x\cupp w_z-w_x\cupp z)(g,h) & = x(g)\begin{pmatrix}
                1 & y(g) \\
                0 & 1
            \end{pmatrix}   
            \begin{pmatrix}
                t_z(h) \\
                z(h)
            \end{pmatrix} - 
            \begin{pmatrix}
                t_x(g) \\
                x(g)
            \end{pmatrix}z(h) \\ & = \begin{pmatrix}
                x(g)t_z(h)-t_x(g)z(h)+x(g)y(g)z(h) \\
                0
            \end{pmatrix}\,,
        \end{split}
    \]
we see that $\langle x,y,z \rangle_\rho\in H^2(G,\ZZ/p\ZZ)$ maps to $x\cupp w_z-w_x\cupp z\in H^2(G, V_y)$.
\end{proof}

Let us denote by $\langle x,y,z\rangle'$ the subset of $H^2(G, \ZZ/p\ZZ)$ which is the preimage of the set of elements in $H^2(G, V_y)$ of the form $x\cupp w_z-w_x\cupp z$, where $w_x, w_z$ ranges over lifts of $x$ and $z$. Let us note that $\langle x,y,z\rangle$ is always a subset of $\langle x,y,z\rangle'$ by Proposition \ref{prop:preim}.

\begin{proposition}\label{prop:coins}
    Assume that $y\cupp H^1(G,\ZZ/p\ZZ)\subseteq x\cupp H^1(G,\ZZ/p\ZZ)+H^1(G,\ZZ/p\ZZ)\cupp z$. Then the subset $\langle x,y,z \rangle'$ of $H^2(G,\ZZ/p\ZZ)$ coincides with $\langle x,y,z \rangle$.
\end{proposition}

\begin{proof}
Recall that the indeterminacy of $\langle x,y,z\rangle$ is given by $x\cupp H^1(G, \ZZ/p\ZZ)+H^1(G, \ZZ/p\ZZ)\cupp z$. On the other hand, the indeterminacy of $\langle x,y,z\rangle'$ is given by $x\cupp H^1(G,\ZZ/p\ZZ)+H^1(G,\ZZ/p\ZZ)\cupp z+y\cupp H^1(G,\ZZ/p\ZZ)$. The proposition follows immediately from this observation. 
\end{proof}

The formula in Proposition \ref{prop:coins} simplifies considerably in the case when $x=y$.

\begin{proposition}\label{prop:half-cup}
    Let $p$ be an odd prime. Suppose that $x\cupp x = x\cupp y = 0$ and choose $w_x$ and $w_y$ as above. Then we have an equality 
        \[
            x\cupp w_y -w_x\cupp y = \frac{1}{2}x\cupp w_y    
        \]
    in $H^2(G, V_x)$. Hence $\langle x,x,y \rangle$ consists of elements of the form $M+x \cupp \alpha + \beta \cupp y$, for $\alpha, \beta \in H^1(G,\ZZ/p\ZZ)$, where $M$ maps to $\frac{1}{2} x \cupp w_y$ under $H^2(G, \ZZ/p\ZZ)\to H^2(G, V_x)$.
\end{proposition}

\begin{proof}
    Choose $w_x$ by setting $t_x=x(x-1)/2$; the element $t_x$ satisfies that $dt_x=-x\cupp x$. We have 
        \[
            \begin{split}
                (x\cupp w_y -w_x\cupp y)(g, h) & = 
                \begin{pmatrix}
                    x(g)t_y(h)+x(g)^2y(h)-t_x(g)y(h) \\
                    0
                \end{pmatrix} \\
                & =
                \frac{1}{2}\begin{pmatrix}
                    2x(g)t_y(h)+ x(g)^2y(h)+x(g)y(h)\\
                    0
                \end{pmatrix}
            \end{split}
        \]
    and 
        \[
            \frac{1}{2}x\cupp w_y(g, h) = \frac{1}{2}\begin{pmatrix}
                x(g)t_y(h)+x(g)^2y(h) \\
                x(g)y(h)
            \end{pmatrix}\,.
        \]
    Hence 
        \[
            (x\cupp w_y -w_x\cupp y-\frac{1}{2}x\cupp w_y)(g,h) = 
            \frac{1}{2}\begin{pmatrix}
                x(g)t_y(h)+x(g)y(h) \\
                -x(g)y(h)
            \end{pmatrix}\,.
        \]
    This element is the coboundary of
        \[
            \frac{1}{2}\begin{pmatrix}
                -t_y \\
                t_y
            \end{pmatrix}  \in C^1(G, V_x), 
        \]
    hence $(x\cupp w_y -w_x\cupp y-\frac{1}{2}x\cupp w_y)(g,h) = 0$ in $H^2(G,V_x)$.  
\end{proof}

In Section \ref{sec:massey} we will use Proposition \ref{prop:coins} to give a formula for the Massey product in the \'etale cohomology of number fields using number theoretical data. 
\section{Massey products in \'etale cohomology}\label{sec:etale}

Fix a connected scheme $X$. In this section, we show that we may apply the results of Section \ref{sec:galois} to give a method for computation of certain 3-fold Massey products in the \'etale cohomology of $X$ with coefficients in $\ZZ/p\ZZ$. To achieve this, we will start by comparing Massey products in the DG-algebra of cochains on a connected space $S$ with Massey products in $\pi_1(S)$ (for some choice of basepoint), as defined in Section \ref{sec:galois}. 

\subsection*{Comparison of group cohomology and \'etale cohomology}
We let $S$ be either a connected space, i.e., a simplicial set that is connected and Kan, or a connected pro-space, that is, $S = \{S_i\}_{i \in I}$ for some cofiltered category I, where each $S_i$ is a simplicial set that is connected and Kan. Given such an $S$, one can define the cochain complex $C^*(S, \ZZ/p\ZZ)$ (if $S$ is a pro-space, $C^*(S, \ZZ/p\ZZ) := \colim_i C^*(S_i, \ZZ/p\ZZ))$. As is well-known, $C^*(S, \ZZ/p\ZZ)$ has the structure of a DG-algebra and we denote its multiplication by $- \cupp -$. 

Given $x,y,z \in H^1(S, \ZZ/p\ZZ)$ such that $x \cupp y = y \cupp z =0,$ one can then define the Massey product  $\langle x,y,z \rangle$ as follows: start by choosing lifts of $x,y,z$ to $C^1(S,\ZZ/p\ZZ)$; we will by abuse of notation call these lifts $x,y$ and $z$. Now  choose $k_{xy}, k_{yz} \in C^1(S, \ZZ/p\ZZ)$ such that  $dk_{xy} = - x \cupp y, dk_{yz} = - y \cupp z$. Note that since the differential is a derivation, the element $$k_{xy} \cupp z + x \cupp k_{yz} \in C^2(S,\ZZ/p\ZZ)$$ is actually a cocycle, and thus defines a class in cohomology. One then defines the Massey product $\langle x,y,z\rangle$ as the subset $\{ k_{xy} \cupp z + x \cupp k_{yz} \}$ of $H^2(S,\ZZ/p\ZZ)$, where $k_{xy}$ and $k_{yz}$ ranges over all possible choices of trivializing 1-cochains $k_{xy}$ and $k_{yz}$ as above.
Just as when we defined Massey products for groups in Section \ref{sec:galois}, the Massey product $\langle x,y,z \rangle$ is uniquely determined by one choice of $k_{xy},k_{yz}$ as above, up to an \emph{indeterminacy}; the indeterminacy of the Massey product $\langle x,y,z \rangle$ consists of the subset of $H^2(S,\ZZ/p\ZZ)$ given by 
    \[
        H^1(S,\ZZ/p\ZZ) \cupp z + x \cupp H^1(S,\ZZ/p\ZZ)\,.
    \]
Since $H^1(S, \ZZ/p\ZZ) = H^1(\pi_1(S), \ZZ/p\ZZ)$, it is natural to ask how Massey products in $S$ and $B\pi_1(S)$, compare. Here $B\pi_1(S)$ is the (profinite) classifying space of $\pi_1(S)$. Towards this comparison, we start by noting that Massey products are weakly functorial: if we have a map $f\colon S\to T$ of spaces, and $x,y,z \in H^1(T, \ZZ/p\ZZ)$ are such that $x \cupp y = y \cupp z = 0$, then 
$$f^*(\langle x,y,z\rangle ) \subseteq \langle f^*(x),f^*(y),f^*(z) \rangle.$$ Applying this weak functoriality to the natural map $\tau\colon S \rightarrow B \pi_1(S)$, we find that  $$\tau^*(\langle x,y,z\rangle) \subseteq \langle \tau^*x,\tau^*y,\tau^*z\rangle.$$  Since the indeterminacy of $\tau^*(\langle x,y,z\rangle)$ coincides with the indeterminacy of $\langle \tau^*(x),\tau^*(y),\tau^*(z) \rangle$, one sees that $\tau^*(\langle x,y,z\rangle) =  \langle \tau^*(x),\tau^*(y),\tau^*(z)\rangle$. This shows that if $S$ is any space or pro-space and $x,y,z \in H^1(S, \ZZ/p\ZZ)$ are such that $x \cupp y = y \cupp z = 0$, then the Massey product $\langle x,y,z\rangle$ can be computed by first computing the Massey product in $B\pi_1(S)$ and then pulling back the resulting subset to $H^2(S,\ZZ/p\ZZ)$.

Let us finally note that Massey products for the profinite classifying space $B \pi_1(S)$ identifies with Massey products in the Galois cohomology of the profinite group $\pi_1(S)$ as defined in Section \ref{sec:galois}. Indeed: $B \pi_1(S)$ can be identified with the profinite simplicial set which is the (profinite) nerve of the category with one object and with the set of elements of $\pi_1(S)$ as morphisms. Then the continuous simplicial cochain complex $C^*(B \pi_1(S),\ZZ/p\ZZ)$ has in degree $n$ continuous functions $\pi_1(S)^{\times n} \rightarrow \ZZ/p\ZZ$ as elements. The resulting complex coincides with the homogeneous cochain complex computing the continuous group cohomology of $\pi_1(S)$, thus Massey products coincide. 

To a connected scheme $X$ we may associate the \emph{\'etale topological type} $\Et(X)$ \cite{FriedlanderEtale}; the \'etale topological type of a scheme $X$ is a pro-space, and by the above discussion we have a map $\Et(X)\to B\pi_1(X)$, inducing maps 
    \[
        \begin{split}
            H^1(\pi_1(X), \ZZ/p\ZZ) & \xrightarrow{\sim} H^1(X, \ZZ/p\ZZ)\,, \\
            H^2(\pi_1(X), \ZZ/p\ZZ) & \hookrightarrow H^2(X, \ZZ/p\ZZ)\,,
        \end{split}    
    \]
that identifies Massey products.  

Recall from Section \ref{sec:galois} that the cup product 
\[
    y\cupp(-)\colon H^n(\pi_1(X),\ZZ/p\ZZ) \to H^{n+1}(\pi_1(X),\ZZ/p\ZZ)    
\]
is the connecting homomorphism of the short exact sequence 
\[
    0\to \ZZ/p\ZZ\to V_y \to \ZZ/p\ZZ \to 0    
\]
where $V_y = \ZZ/p\ZZ^2$ with the $\pi_1(X)$-action given by
\[
    g\mapsto \begin{pmatrix}
        1 & y(g) \\
        0 & 1
    \end{pmatrix}  \,.  
\]
By the equivalence between $\pi_1(X)$-modules and locally constant \'etale sheaves on $X$, if we pull back the $\pi_1(X)$-module $V_y$ to $X$, we get a locally constant \'etale sheaf $P_y$ sitting in an exact sequence 
    \[
        0\to \ZZ/p\ZZ\to P_y\to \ZZ/p\ZZ \to 0
    \] 
and the corresponding connecting homomorphism 
    \[
        H^i(X, \ZZ/p\ZZ) \to H^{i+1}(X, \ZZ/p\ZZ)    
    \]
is a cup product with $y$ \cite[Lemma 3.1]{Ahlqvist--Carlson-cup}. Just as in the case of group cohomology, given classes $x,y,z\in H^1(X, \ZZ/p\ZZ)$ such that $x\cupp y=y\cupp z = 0$, we can form the set $\langle x,y,z\rangle'$ that is the preimage under $H^1(X, \ZZ/p\ZZ)\to H^1(X, P_y)$ of the set of elements $$x\cupp w_z-w_x\cupp z\in H^2(X, P_y)$$ where $w_x, w_z\in H^1(X, P_y)$ ranges over all elements mapping to $x$ and $z$ respectively. 

\begin{proposition}\label{prop:coins-et}
    Let $X$ be a connected scheme and let $x,y,z\in H^1(X,\ZZ/p\ZZ)$ be elements such that $x\cupp y=y\cupp z=0$. Assume that $x=y$ or that $y\cupp (-)\colon H^1(X, \ZZ/p\ZZ)\to H^2(X, \ZZ/p\ZZ)$ is zero. Then the Massey product $\langle x,y,z\rangle$ coincides with $\langle x,y,z\rangle'$.
\end{proposition}

\begin{proof}
    This follows from Proposition \ref{prop:coins} and the fact that the map on cohomology is induced by $\Et(X)\to B\pi_1(X)$ as in the above discussion. 
\end{proof}
\section{Massey products for number fields}\label{sec:massey}

In this section we let  $X = \Spec \OO_K$, the spectrum of the ring of integers of a number field $K$. For simplicity of exposition, we assume that $K$ is totally imaginary or that the prime $p$ is odd. The case where $K$ has real embeddings and $p$ is even and/or $X$ is replaced by an open subscheme $U\subseteq X$ may be treated using the theory of \cite{Ahlqvist--Carlson-punctured}. 

We will give an explicit number theoretical formula for the Massey product $\langle x,y,z\rangle$ together with a much simpler formula when $x=y$. To achieve this we will extensively use the theory of arithmetic duality \cite{MazurNotes, MilneADT,Demarche--Harari}. For completeness, we briefly recall the parts of the theory that we will use. 

\subsection{Background on arithmetic duality} \label{subsec:arithmetic}
Given a locally constant constructible sheaf $A$ of abelian groups on the small \'etale site of $X$, we denote by $D(A)=R\HOM(A, \GG_m)$, which is an object in the derived category $D(X_{\et})$. In particular, by definition, we have $H^i(X, D(A))\cong\Ext^i(A, \GG_m)$ for all $i\geq 0$. By Artin--Verdier duality (\cite[(2.4)]{MazurNotes},\cite[Corollary III.3.2]{MilneADT}, \cite[Theorem 1.1]{Demarche--Harari}), the map   
    \[
        R\Gamma(X, A)\otimes^{\LL} R\Gamma(X, D(A))\to R\Gamma(X, \GG_m)\to \QQ/\ZZ[-3]   
    \]
induces an isomorphism 
    \[
        R\Gamma(X, A) \cong (R\Gamma(X, D(A)))[3]^\sim    
    \]
in the derived category of abelian groups \cite[Corollary III.3.2]{MilneADT}, \cite[Theorem 1.1]{Demarche--Harari}, where $^\sim$ denotes the Pontryagin dual. As a consequence, we have a pairing of abelian groups 
    \[
        H^{3-i}(X, A)\times H^i(X, D(A))\to H^3(X, \GG_m)\cong \QQ/\ZZ    
    \]
which is perfect and makes $H^{3-i}(X, A)$ Pontryagin dual to $H^i(X, D(A))$. We will use the fact that this duality is also functorial and we refer to \cite[Section 3]{Ahlqvist--Carlson-cup} for precise statements about this functoriality. 

Let $\Div(K)$ be the group of fractional ideals of $K$ and let $\divis \colon K^\times \to \Div(K)$ be the map which takes $a\in K^\times$ to the principal ideal generated by $a$. By \cite[Section 2]{Ahlqvist--Carlson-cup} we have that
    \[
        H^i(X, D(\ZZ/p\ZZ)) =
        \begin{cases}
            \mu_p(K) & i=0\,, \\
            Z^1/B^1 & i=1\,, \\
            \Cl(K)/p\Cl(K) & i=2\,, \\
            \ZZ/p\ZZ & i=3\,, \\
            0 & i \geq 4\,,
        \end{cases}   \quad
        H^i(X, \ZZ/p\ZZ) =
        \begin{cases}
            \ZZ/p\ZZ & i=0\,, \\
            (\Cl(K)/p\Cl(K))^\sim & i=1\,, \\
            (Z^1/B^1)^\sim & i=2\,, \\
            (\mu_p(K))^\sim & i=3\,, \\
            0 & i \geq 4\,,
        \end{cases}    
    \]
where 
    \[
        \begin{split}
            Z^1 & = \ker\left(\begin{pmatrix}
                \divis & p
            \end{pmatrix}\colon K^\times \oplus \Div(K)\to \Div(K)\right)\,, \mbox{ and} \\
            B^1 & = \im\left(\begin{pmatrix}
                -p \\
                \divis
            \end{pmatrix}\colon K^\times\to K^\times \oplus \Div(K)\right)\,.
        \end{split}  
    \]
This means that $Z^1 = \{(a, I)\in K^\times\oplus \Div(K):\divis(a)+pI=0\}$,  $B^1= \{(b^{-p}, \divis(b))\}$, and $Z^1/B^1$ sits in an exact sequence 
    \[
            0\to \OO_K^\times/p\OO_K^\times \xrightarrow{b\mapsto (b, 0)} Z^1/B^1 \xrightarrow{(a, I)\mapsto I} \Cl(K)[p]\to 0\,. 
    \]
For example, when $K=\QQ(\sqrt{-5})$ and $p=2$, we have that $Z^1/B^1$ is generated by the classes $(-1, 0)$ and $(2, (2, 1+\sqrt{-5}))$, and since it has no element of order 4, we get that $Z^1/B^1\cong \ZZ/2\ZZ\oplus \ZZ/2\ZZ$. 

\begin{remark}\label{rk:et-vs-fppf}
    Let $f\colon X_{\fl} \to X_{\et}$ be the natural map of sites, where $X_{\fl}$ is the big fppf site of $X$, and denote by $f^*$ and $f_*$ the maps on topoi induced by $f$. Note that $R\Gamma(X_{\fl}, \HOM(f^*A, \GG_m)) \cong R\Gamma(X, D(A))$, naturally in $A$, where $\HOM(f^*A, \GG_m)$ is the Cartier dual. Indeed, we have the chain of canonical isomorphisms
        \[
        \begin{split}
            R\Gamma(X_{\fl}, \HOM(f^*A, \GG_m))  
            & \cong R \Hom(f^*A,\mathbb{G}_m) \\ 
            & \cong R \Hom(A,Rf_*\mathbb{G}_m) \\
            & \cong R \Hom(A,f_*\mathbb{G}_m) \\ 
            & \cong R\Gamma(X_{\et}, D(A))\,.
        \end{split}
        \]
    The first isomorphism follows from the fact that $\HOM(A,\GG_m) = R \HOM(A,\GG_m)$ by \cite[Theorem 11.6]{BrochardDuality}, and that $R \Gamma \circ R \HOM = R \Hom$. The second isomorphism follows from adjunction, and the third isomorphism follows from the fact that $Rf_* \mathbb{G}_m = f_* \mathbb{G}_m$ since $\mathbb{G}_m$ is a smooth group scheme (see \cite[Chapter III, Theorem 3.9]{MilneEtale} and its proof). The fourth equality is per definition.
\end{remark}

With the necessary prerequisites on \'etale cohomology now stated, we move on to compute formulas for Massey products, but first, we will need some preliminary number-theoretical results on bicyclic field extensions. 

\subsection{The intersection of the kernel of norms}
Let $K$ be a number field, and let $L_{x}, L_y$ be two distinct unramified extensions of degree $p$, where $p$ is a prime number, and let $L_{xy}$ be their compositum. We thus have that $L_{xy}$ is an unramified bicyclic field extension of $K$ of degree $p^2$. We have a diagram 
    \[
        \begin{tikzcd} 
            L_x \arrow[r,above,"i_y"] \arrow[d,shift left = .75ex , "N_x"]  & \arrow[l, shift left = .75ex,below,"N_y"] L_{xy} \arrow[d,shift left = .75ex , "N_x"]   \\ 
            K \arrow[r,above,"i_y"]  \arrow[u,"i_x"]  &  \arrow[l, shift left = .75ex,below,"N_y"]   \arrow[u,"i_x"]  L_y  
        \end{tikzcd}
    \] 
where all the maps are the obvious inclusions and norms. We write 
    \[
        \Gal(L_{xy}/K) \cong \Gal(L_x/K) \times \Gal(L_y/K) = \langle \sigma_x \rangle \times \langle \sigma_y \rangle\,,
    \]
where $\sigma_x$ is a generator of $\Gal(L_x/K)$ and $\sigma_y$ is a generator of $\Gal(L_y/K)$. We now wish to characterize the intersection $\ker N_x \cap \ker N_y$. In more detail: we wish to show that any element $f \in \ker N_x \cap \ker N_y \subset L_{xy}^\times$ can be written as $(1-\sigma_x)(1-\sigma_y) \gamma$ for some  $\gamma \in L_{xy}^\times$; in other words, we claim that $\ker N_x \cap \ker N_y = \im (1-\sigma_x)(1-\sigma_y)$, where $(1-\sigma_x)(1-\sigma_y)$ is viewed as an endomorphism of $L_{xy}^\times$. 

\begin{lemma} \label{lemma:condition1} 
Let $L_{xy}/K$ be the compositum of two distinct unramified field extensions $L_x/K, L_y/K$, where both are of degree $p$, for $p$ a prime number.  Suppose that $f \in L_{xy}^\times$ is such that $f = (1-\sigma_x)\alpha$ for some $\alpha \in L_{xy}^\times$. Then $f$ can be written as $ (1-\sigma_x)(1-\sigma_y) \gamma$ for some $\gamma \in L_{xy}^\times$ if and only if $N_y(\alpha)$ lies in the image of the composite 
    \[
        L_y^\times \xrightarrow{i_x} L_{xy}^\times \xrightarrow{N_y} L_x^\times\,.
    \] 
\end{lemma} 

\begin{proof}
Suppose that $f= (1-\sigma_x)\alpha = (1-\sigma_x)(1-\sigma_y) \gamma$, so that $\alpha = (1-\sigma_y)\gamma + i_x(u)$ for some $u \in L_y^\times$. Applying $N_y$ we see that $N_y(\alpha) = N_y(i_x(u))$, so the condition is necessary. To see that it is sufficient, 
let $u \in L_y^\times$ be an element satisfying that $N_y(i_x(u)) = N_y(\alpha)$. Since $(1-\sigma_x) i_x(u) = 0$, we have that 
    \[
        f = (1-\sigma_x)(\alpha-i_x(u))\,.
    \]
Now, $N_y(\alpha-i_x(u)) = 0$, so that by Hilbert's theorem 90,
$\alpha-i_x(u) = (1-\sigma_y) \gamma$, implying our result.
\end{proof}
 
\begin{lemma}\label{lemma:int-norm-ker}
Let $L_{xy}/K$ be the compositum of two distinct unramified field extensions $L_x/K, L_y/K$, where both are of degree $p$, for $p$ a prime number. Then we have the equality 
    \[
        \ker N_x \cap \ker N_y = \im (1-\sigma_x)(1-\sigma_y)\,.
    \]
\end{lemma}

\begin{proof}
The containment $\ker N_x \cap \ker N_y \supset \im (1-\sigma_x)(1-\sigma_y)$ is obvious. So suppose now that $f \in \ker N_x \cap \ker N_y$. Then $f = (1-\sigma_x)\alpha = (1-\sigma_y)\beta$ for some $\alpha, \beta \in L_{xy}^\times$. By applying $N_y$ we see that $0 = (1-\sigma_x)N_y(\alpha)$, which implies that the element $N_y(\alpha)$, which a priori is an element of $L_x$, is an element of  $K$, i.e., it is true that 
    \[
        N_y(\alpha) \in \im(K \xrightarrow{i_x} L_x)\,.
    \] 
We will now apply Lemma \ref{lemma:condition1} and show that $N_y(\alpha) = N_y(i_x(u))$ for some $u \in L_y^\times$. The following diagram is commutative
    \[
        \begin{tikzcd}
            L_y \arrow[r,"N_y"] \arrow[d,"i_x"] & K \arrow[d,"i_x"] \\ 
            L_{xy} \arrow[r,"N_y"] & L_x\,.
        \end{tikzcd}
    \] 
Thus, to see that $N_y(\alpha)$ is in the image of the composite $L_y^\times \xrightarrow{i_x} L_{xy}^\times \xrightarrow{N_y} L_x^\times$ it is enough to see that $N_y(\alpha)$, viewed as an element of $K$, is in the image of $N_y\colon L_y^\times \rightarrow K^\times$. To see this, we will use the Hasse norm theorem. We consider $\divis(N_y(\alpha)) \in \Div K$, and for a prime $\p$ of $K$, we let 
    \[
        v_\p(N_y(\alpha)) \in \ZZ 
    \]
be the valuation of $N_y(\alpha) $ at $\p$. We thus have the equality 
    \[
        \divis(N_y(\alpha)) = \prod_\p \p^{v_\p(N_y(\alpha)) }\,.
    \] 
Then, if $\p$ is split in $L_y$, clearly $\p^{v_\p(N_y(\alpha)) }$ is in the image of the associated norm map on divisors. On the other hand, if $\p$ is inert in $L_y$ but split in $L_x$, then we have that, for any prime $\q$ of $L_x$ lying over $\p$, that $\q$ is inert in $L_{xy}$. This shows that $v_\q(N_y(f)) \in p \ZZ$. Since $N_y(\alpha)$ is invariant under $\Gal(L_x /K)$, and  $\Gal(L_x/K)$ acts transitively on the primes over $\p$, the integer $v_\q(N_y(f)) \in p \ZZ$ is independent of the choice of $\q$ lying over $\p$. This implies that  $\p^{v_\p(N_y(\alpha))} \in \Div K$ is in the image of the norm map, since $v_\p(N_y(\alpha)) \in p \ZZ$.  Lastly, the case where the prime $\p$ is inert in both $L_x$ and $L_y$ is clear. Indeed, then once again $v_\p(N_y(\alpha)) \in p \ZZ$. We have thus shown that $\divis(N_y(\alpha))$ lies in the image of the norm $N_y$, that is, it is locally a norm. Hasse's norm theorem now gives the conclusion, since units are always norms in unramified extensions. 
\end{proof}
\begin{remark} \label{rmk:equal}
If we instead in the above do not assume that $x$ and $y$ are distinct, a version of the above proposition still holds. So assume that $x=y$ and denote, by abuse of notation and for the sake of uniformity, by $L_{xy}$ the total ring of fractions of $\OO_{L_x} \otimes_{\OO_K} \OO_{L_x}$. We let $\sigma_x$ a generator for the Galois action on the left factor and $\sigma_y$ a generator for the action on the right factor. Then it is still true that $\ker N_x \cap \ker N_y = \im (1-\sigma_x)(1-\sigma_y)$.  This can be proven in a manner analogous to Lemma \ref{lemma:int-norm-ker} and we leave the details to the reader.
\end{remark}
\subsection{Computations of Massey products} \label{subseq:comp}
In this subsection, we will give  formulas for 3-fold Massey products in the \'etale cohomology of $X$. Recall the definition of $P_y$ in Section \ref{sec:etale}. 
We will find these formulas by first giving, for $x, y\in H^1(X, \ZZ/p\ZZ),$ a formula for the map 
    \[
        x\cupp (-)\colon H^1(X, P_y)\to H^2(X, P_y)\,.    
    \]
By Proposition \ref{prop:coins-et} and with $x,y,z$ as in there, this is enough information to determine the Massey product $\langle x,y,z\rangle$ in favorable situations. Indeed, when either $x=y$ or when 
    \[
        y\cupp (-)\colon H^1(X,\ZZ/p\ZZ) \to H^2(X,\ZZ/p\ZZ)
    \]
is zero, the Massey product $\langle x,y,z\rangle$ is determined by $x\cupp w_z-w_x\cupp z$, so by graded commutativity, it is enough to compute $x\cupp w_z$ and $z\cupp w_x$. We will realize 
    \[
        x \cupp(-)\colon H^1(X,P_y) \to H^2(X,P_y)
    \]
as a connecting homomorphism coming from a short exact sequence 
    \begin{equation}\label{eq:main-ses}
        0 \to P_y \to P_x \otimes P_y \to P_y \to 0
    \end{equation}
of locally constant sheaves. By arithmetic duality, this connecting homomorphism is dual to a map $H^1(X,D(P_y)) \to H^2(X,D(P_y))$ \cite[Lemma 3.4]{Ahlqvist--Carlson-punctured}, \cite[Lemma 4.3]{Demarche--Harari}, and we will compute this map, and then dualize once again to get the formula for $x \cupp (-)$. After this, we will see that when we let $w_z\in H^1(X, P_y)$ be a lift of $z \in H^1(X,\ZZ/p\ZZ)$, the formula becomes more tractable. The element $x\cupp w_z\in H^2(X, P_y)$ will be identified, via arithmetic duality, with a functional on $H^1(X, D(P_y))$. We will describe the functional by the relation   
    \[
        \langle x\cupp w_z, (b,a,J,I)\rangle = \langle z, J_K(x,y, b,a,J,I)\rangle\,,    
    \]
where $\langle -,- \rangle$ denotes evaluation, and $(b,a,J,I)$ represents a class in $H^1(X, D(P_y))$. The ideal $J_K(x,y, b,a,J,I)$ represents a class in $H^2(X,D(\ZZ/p\ZZ))\cong \Cl(K)/p\Cl(K)$ which depends on $x,y,b,a,J$ and $I$. We will give an explicit formula for the ideal class $J_K(x,y,b,a,J,I)$ and hence a number-theoretical formula for $x \cupp w_z$.

After this computation has been made, we will specialize to the case where $x=y$. In this situation, the formula for the ideal class $J_K(x,y,b,a,J,I)$ becomes much more explicit, thus greatly simplifying the formula for $x \cupp w_z$.  

Let $L_y\supseteq K$ be the $\ZZ/p\ZZ$-extension corresponding to $y$, let $G_y=\Gal(L_y/K)$, and let $\sigma_y\in G_y$ be a generator. Then let us note that $P_y,$ as defined in Section \ref{sec:etale}, corresponds to the $G_y$-module $V_y=\FF_p[G_y]/I_y^2$, where $I_y$ is the augmentation ideal. Indeed, by choosing the basis $\{\sigma_y-1, 1\}$ for $V_y$, we see that $G_y$ acts via 
    \[
        \sigma_y\mapsto \begin{pmatrix}
            1 & 1 \\
            0 & 1
        \end{pmatrix} = 
        \begin{pmatrix}
            1 & y(\sigma_y) \\
            0 & 1
        \end{pmatrix}\,.   
    \]
This gives the connection to the Galois module $V_y$ described in Section \ref{sec:galois}, by showing that the Galois modules are isomorphic.

The cup product $x\cupp (-)\colon H^1(X, P_y)\to H^2(X, P_y)$ will be the connecting homomorphism of the short exact sequence (\ref{eq:main-ses}) obtained by tensoring 
    \[
        0\to \ZZ/p\ZZ\to P_x\to \ZZ/p\ZZ\to 0    \label{sequence1}
    \]
with the sheaf $P_y$. The resulting short exact sequence will correspond to a short exact sequence of $G_x\times G_y$-modules
    \begin{equation}\label{eq:ses}
        0 \to \FF_p[G_y]/I_y^2\to \FF_p[G_x\times G_y]/(I_x^2+I_y^2)\to \FF_p[G_y]/I_y^2\to 0\,.   
    \end{equation}
We may view this as a $\pi_1(X)$-module sequence where $\pi_1(X)$ acts diagonally on the middle term. By \cite[V.3.3]{BrownCohomology}, the resulting connecting homomorphism will agree with the cup product $x\cupp (-)\colon H^1(\pi_1(X), V_y)\to H^2(\pi_1(X), V_y)$.   


\subsubsection{Computation of the connecting homomorphism}
The computation of the connecting homomorphism is quite involved, with some technical details provided in Appendix \ref{sec:spectral}. First, we work in $D(X_{\et})$, the derived category of abelian sheaves on the small \'etale site of $X$. We let $C(X_{\et})$ be the category of complexes of abelian \'etale sheaves and denote by $\gamma\colon C(X_{\et}) \rightarrow D(X_{\et})$ the localization map. Our computation is divided into the following steps:
\begin{enumerate}
    \item[(I)] By Artin--Verdier duality (\cite[Lemma 2.13]{Ahlqvist--Carlson-cup}), the map $R \Gamma(X,P_y) \rightarrow R\Gamma(X,P_y[1])$ is Pontryagin dual to the map $R\Hom(P_y[1], \GG_{m}) \rightarrow R \Hom(P_y,\GG_m)$. Thus, it suffices to calculate the latter map. We start by replacing the short exact sequence (\ref{eq:ses}) with the sequence of resolutions in Diagram (\ref{eq:full-resolution}) of Appendix \ref{section:resolutions}, which we denote by $\cE\xrightarrow{\alpha}\cE'\to\cE''$. By taking the mapping cylinder and the cone of $\alpha$, we find a distinguished triangle (T): $\cE\to Cyl(\alpha)\to C(\alpha)\to \cE[1]$, of \emph{locally free} sheaves (see Appendix \ref{sec:spectral}), where the last map represents the connecting homomorphism. Note that the image under $\gamma$ of the zig-zag $\cE \xleftarrow{q} C(\alpha) \xrightarrow{\pr} \cE[1]$ represents the map $\cE \to \cE[1]$ in $D(X_{\et})$.
    \item[(II)] We resolve $\GG_m$ by the complex $\mathcal{C}$ defined by $j_* \mathbb{G}_{m,K} \rightarrow \DIV_X$. Here $j\colon \Spec K \rightarrow X$ is the generic point, and $\DIV_X = \oplus_{p \in X} (i_{p})_* \ZZ$, where $p$ ranges over all closed points of $X$ and $i_p\colon \Spec k(p) \rightarrow X$ is the inclusion. By applying $R\HOM(-, \cC)$ to (T) and using the fact that all the complexes in the first variable are degree-wise locally free, we obtain the exact triangle 
        \[
            \HOM(\cE[1],\mathcal{C}) \to \HOM(C(\alpha),\mathcal{C}) \to \HOM(Cyl(\alpha),\mathcal{C}) \to \HOM(\cE,\mathcal{C})\,.
        \]  
    Note that image under $\gamma$ of the zig-zag (Z): $\HOM(\cE,\mathcal{C}) \xrightarrow{q^*} \HOM(C(\alpha),\mathcal{C}) \xleftarrow{\pr^*} \HOM(\cE[1],\mathcal{C})$ represents the map $R\HOM(\cE,\mathbb{G}_m) \leftarrow R \HOM(\cE[1],\mathbb{G}_m)$.
    \item[(III)] 
    With $\mathbb{G}_m$ replaced by $\mathcal{C}$, the natural map $\Gamma(\HOM(\cE,\mathcal{C}),X) \rightarrow R \Gamma(\HOM(\cE,\mathcal{C}),X)$ is an isomorphism after applying $H^i$ for $i=0,1,2$. The same statement holds with $\cE$ replaced by $C(\alpha)$ (for both these facts, see Proposition \ref{prop:edge}: note that $C(\alpha)$ is a resolution of $\cE$). Since $R \Gamma$ applied to the zig-zag (Z) computes the connecting homomorphism we are after, the isomorphisms on $H^i$, in the aforementioned degrees, imply that the zig-zag
        \begin{equation}\label{ziggy}
            \Hom(\cE,\mathcal{C}) \xrightarrow{q^*} \Hom(C(\alpha),\mathcal{C}) \xleftarrow{\pr^*} \Hom(\cE[1],\mathcal{C}) 
        \end{equation} 
    of complexes of abelian groups represents the connecting homomorphism $H^{1}(X,D(P_y))\to H^{2}(X,D(P_y))$.
    \item[(IV)] In the last step, we explicitly compute what the zig-zag (\ref{ziggy}) does on cohomology. The map $q^*$ can be formally shown to be an isomorphism on cohomology in the necessary degrees. However, to give an explicit formula for the inverse of $q^*$ in cohomology, we need to use theorems from global class field theory.
\end{enumerate}
We now proceed to compute the connecting homomorphism $H^1(X,D(P_y)) \to H^2(X,D(P_y))$, which is dual to the cup product with $x$. 
To start, we realize step (I) above: We take the cone of the left part of Diagram (\ref{eq:full-resolution}) in Appendix \ref{section:resolutions}. Then, composing with the quasi-isomorphism of Diagram (\ref{eq:quasi-iso}), 
we obtain the following diagram:

    \[
        \begin{tikzcd}
                &   \ZZ[G_x\times G_y]^5\ar{r}{\id}\ar{d}{C(\alpha)^{-3}} & \ZZ[G_x\times G_y]^5\ar{d}{\delta^{-1,-2}}\ar{r} & \ZZ[G_y]\ar{d} \\
            \ZZ[G_y]\ar{d}{\delta^{1,-2}} & \ZZ[G_x\times G_y]^6\ar{l}{q_{-2}} \ar{d}{C(\alpha)^{-2}}\ar{r}{\pr} & \ZZ[G_x\times G_y]^3\ar{d}{\delta^{-1,-1}} \ar{r} & \ZZ[G_y]^2\ar{d}\\
            \ZZ[G_y]^2\ar{d}{\delta^{1,-1}} & \ZZ[G_x\times G_y]^4\ar{l}{q_{-1}}\ar{d}{C(\alpha)^{-1}}\ar{r}{\pr} & \ZZ[G_x\times G_y]\ar{r} & \begin{matrix}
                \ZZ\\
                \oplus \\
                \ZZ[G_y]
            \end{matrix} \\
            \begin{matrix}
                \ZZ\\
                \oplus \\
                \ZZ[G_y]
            \end{matrix} & 
            \begin{matrix}
                \ZZ[G_y]\\
                \oplus \\
                \ZZ[G_x] \\
                \oplus \\
                \ZZ[G_x\times G_y]
            \end{matrix}\ar{l}{q_0} & 
        \end{tikzcd}    
    \]
where the horizontal morphism in the middle represents the connecting homomorphism, and the other two horizontal morphisms of complexes are quasi-isomorphisms, which are there to simplify the description of the cohomology groups. The vertical morphisms of the cone are:
    \[
        \begin{split}
            C(\alpha)^{-3} & = \scalemath{0.8}{
            \begin{pmatrix}
                0 & -\Gamma_y & \Delta_y & 1-T_x & 0 \\
                1-T_x & 1-T_y & 0 & 0 & 0 \\
                -p & 0 & 0 & -(1-T_y)^2 & \Delta_x \\
                0 & 1-T_x & 0 & 0 & 0 \\
                1-T_x & 0 & 0 & 0 & 0 \\
                0 & 0 & 0 & -1 & 0
            \end{pmatrix}} \,, \\
            C(\alpha)^{-2} & = \scalemath{0.8}{
            \begin{pmatrix}
                -(T_y-1)^2 & -p & T_x-1 & 0 & 0 & 0 \\
                0 & T_x-1 & 0 & 1-T_y & 1-T_x & 0\\
                (1-T_y)^2 & p & 0 & 0 & \Delta_x-p & (1-T_y)^2(1-T_x) \\
                (T_y-1)(1-T_x) & 0 & 0 & \Delta_y-p & 0 & -(1-T_y)(1-T_x)^2
            \end{pmatrix}} \,, \\
            C(\alpha)^{-1} & = \scalemath{0.8}{
            \begin{pmatrix}
                \varepsilon_x & 0 & \varepsilon_x & 0 \\
                0 & 0 & 0 & \varepsilon_y \\
                0 & p & 1-T_x & 1-T_y
            \end{pmatrix}}\,,
        \end{split}
    \]
and 
    \[
        \begin{split}
            q_{-2} & = \begin{pmatrix}
                0 & 0 & 0 & -\varepsilon_x & 0 & 0
            \end{pmatrix} \\
            q_{-1} & = \begin{pmatrix}
                0 & \varepsilon_x & 0 & 0 \\
                0 & 0 & 0 & \varepsilon_x
            \end{pmatrix} \\
            q_0 & = \begin{pmatrix}
                0 & \varepsilon_x & 0 \\
                0 & 0 & \varepsilon_x
            \end{pmatrix}\,.
        \end{split}    
    \]

To complete step (II), let $\cC$ be the resolution $j_*{\GG_{m,X}}\to \DIV_X$ of $\GG_{m,X}$, concentrated in degree 0 and 1. Then apply $\Gamma\circ\HOM(-,\cC)$ to obtain the diagram

    \begin{equation}\label{eq:massive}
        \begin{tikzcd}
            \begin{matrix}
                K^\times\\
                \oplus \\
                L_y^\times
            \end{matrix}\ar{r}\ar{d}
            &
            \begin{matrix}
                L_y^\times \\
                \oplus \\
                L_x^\times \\ 
                \oplus \\
                L_{xy}^\times 
            \end{matrix}\ar{d}
            & \\
            \begin{matrix}
                (L_y^\times)^2 \\ 
                \oplus \\
                \Div(K) \\
                \oplus \\
                \Div(L_y)
            \end{matrix}\ar{r} \ar{d}
            &
            \begin{matrix}
                (L_{xy}^\times)^4 \\ 
                \oplus \\
                \Div(L_y) \\
                \oplus \\
                \Div(L_x) \\
                \oplus \\
                \Div(L_{xy})
            \end{matrix}\ar{d}{\varphi_1}
            & L_{xy}^\times\ar{l}\ar{d} & 
            \begin{matrix}
                K^\times \\
                \oplus \\
                L_y^\times
            \end{matrix}\ar{l}\ar{d}
            \\
            \begin{matrix}
                L_y^\times \\
                \oplus \\
                \Div(L_y)^2
            \end{matrix}\ar{r}{\iota_2} \ar{d}
            &
            \begin{matrix}
                (L_{xy}^\times)^6\\
                \oplus \\
                \Div(L_{xy})^4
            \end{matrix}\ar{d}{\varphi_2}
            & 
            \begin{matrix}
                (L_{xy}^\times)^3\\
                \oplus \\
                \Div(L_{xy})
            \end{matrix}\ar{l}[swap]{\iota_1}\ar{d} & 
            \begin{matrix}
                (L_y^\times)^2 \\ 
                \oplus \\
                \Div(K) \\
                \oplus \\
                \Div(L_y)
            \end{matrix}\ar{l}[swap]{\iota_0}\ar{d}{\psi}
            \\
            \begin{matrix}
                K^\times\\
                \oplus \\
                \Div(L_y)
            \end{matrix}\ar{r} \ar{d}
            &
            \begin{matrix}
                (L_{xy}^\times)^5\\
                \oplus \\
                \Div(L_{xy})^6
            \end{matrix}\ar{d}
            & 
            \begin{matrix}
                (L_{xy}^\times)^5\\
                \oplus \\
                \Div(L_{xy})^3
            \end{matrix}\ar{l}\ar{d} & 
            \begin{matrix}
                L_y^\times \\
                \oplus \\
                \Div(L_y)^2
            \end{matrix}\ar{l}\ar{d}
            \\
            \dots\ar{r} & \dots & \dots\ar{l} & 
             \begin{matrix}
                K^\times\\
                \oplus \\
                \Div(L_y)
            \end{matrix}\ar{l}
        \end{tikzcd}
    \end{equation}
which, by step (III), represents the connecting homomorphism in cohomology in degrees 0 and 1, coming from the exact sequence 
    \[
        0\to D(P_y)\to D(P_x\otimes P_y) \to D(P_y) \to 0
    \] 
and which agrees with the dual cup product 
    \[
        c_x^{\sim}\colon H^i(X,D(P_y)) \to H^{i+1}(X,D(P_y))    
    \]
for $i=0,1$. It remains to fulfil step (IV). In the above diagram, 
$L_{xy}^\times$ is the units in the total ring of fractions of $\mathcal{O}_{L_x} \otimes_{\mathcal{O}_K} \mathcal{O}_{L_y}$ and $\Div(L_{xy})$ is the free abelian group on the closed points of $\Spec \mathcal{O}_{L_x} \otimes_{\mathcal{O}_K} \mathcal{O}_{L_y}$. 
In Diagram (\ref{eq:massive}), we have
    \[
        \iota_0  = \scalemath{0.8}{
        \begin{pmatrix}
            0 & (1-\sigma_y)i_x & 0 & 0 \\
            i_x & 0 & 0 & 0 \\
            0 & 0 & 0 & 0 \\
            0 & 0 & 0 & i_x 
        \end{pmatrix}}\,, \quad 
        \iota_1  = \scalemath{0.8}{
        \begin{pmatrix}
            1 & 0 & 0 & 0 \\
            0 & 1 & 0 & 0 \\
            0 & 0 & 1 & 0 \\
            0 & 0 & 0 & 0 \\
            0 & 0 & 0 & 0 \\
            0 & 0 & 0 & 0 \\
            0 & 0 & 0 & 1 \\
            0 & 0 & 0 & 0 \\
            0 & 0 & 0 & 0 \\
            0 & 0 & 0 & 0 
        \end{pmatrix}}\,, \quad 
        \iota_2  = \scalemath{0.8}{
            \begin{pmatrix}
                0 & 0 & 0 \\
                0 & 0 & 0 \\
                0 & 0 & 0 \\
                -i_x & 0 & 0 \\
                0 & 0 & 0 \\
                0 & 0 & 0 \\
                0 & 0 & 0 \\
                0 & i_x & 0 \\
                0 & 0 & 0 \\
                0 & 0 & i_x
            \end{pmatrix}}\,,
    \]
    \[
        \psi = \scalemath{0.8}{
        \begin{pmatrix}
            \sigma_y-1 & p-N_y & 0 & 0 \\
            \divis & 0 & 0 & p \\
            0 & \divis & i_y & 1-\sigma_y
        \end{pmatrix}} \,,
    \]
    \[
        \varphi_1  = \scalemath{0.8}{
        \begin{pmatrix}
            -(\sigma_y-1)^2 & 0 & (\sigma_y-1)^2 & (\sigma_y-1)(1-\sigma_x) & 0 & 0 & 0 \\
            -p & \sigma_x-1 & p & 0 & 0 & 0 & 0 \\
            \sigma_x-1 & 0 & 0 & 0 & 0 & 0 & 0  \\
            0 & 1-\sigma_y & 0 & N_y-p & 0 & 0 & 0 \\
            0 & 1-\sigma_x & N_x-p & 0 & 0 & 0 & 0 \\
            0 & 0 & (1-\sigma_y)^2(1-\sigma_x) & -(1-\sigma_y)(1-\sigma_x)^2 & 0 & 0 & 0 \\
            \divis & 0 & 0 & 0 & i_x & 0 & 0 \\
            0 & \divis & 0 & 0 & 0 & 0 & p \\
            0 & 0 & \divis & 0 & i_x & 0 & 1-\sigma_x \\
            0 & 0 & 0 & \divis & 0 & i_y & 1-\sigma_y
        \end{pmatrix}} \,,
    \]
    \[
        \varphi_2  = \scalemath{0.8}{
            \begin{pmatrix}
                0 & 1-\sigma_x & -p & 0 & 1-\sigma_x & 0 & 0 & 0 & 0 & 0 \\
                -\Gamma_y & 1-\sigma_y & 0 & 1-\sigma_x & 0 & 0 & 0 & 0 & 0 & 0 \\
                N_y & 0 & 0 & 0 & 0 & 0 & 0 & 0 & 0 & 0 \\
                1-\sigma_x & 0 & -(1-\sigma_y)^2 & 0 & 0 & -1 & 0 & 0 & 0 & 0 \\
                0 & 0 & N_x & 0 & 0 & 0 & 0 & 0 & 0 & 0 \\
                \divis & 0 & 0 & 0 & 0 & 0 & (\sigma_y-1)^2 & 0 & -(\sigma_y-1)^2 & (1-\sigma_y)(1-\sigma_x)\\
                0 & \divis & 0 & 0 & 0 & 0 &  p & 1-\sigma_x & -p & 0 \\
                0 & 0 & \divis & 0 & 0 & 0 & 1-\sigma_x & 0 & 0 & 0 \\
                0 & 0 & 0 & \divis & 0 & 0 & 0 & \sigma_y-1 & 0 & p-N_y \\
                0 & 0 & 0 & 0 & \divis & 0 & 0 & \sigma_x-1 & p-N_x & 0 \\
                0 & 0 & 0 & 0 & 0 & \divis & 0 & 0 & -(1-\sigma_y)^2(1-\sigma_x) & (1-\sigma_y)(1-\sigma_x)^2
            \end{pmatrix}} \,,
    \]
where, for $a\in L_{xy}^\times$, we have
    \[
        \Gamma_y(a):=-\sum_{n=1}^{p-1}n\sigma_y^{n}(a)  
    \]
and $\sigma_x,\sigma_y$ are some fixed generators of $\Gal(L_x/K)$ and $\Gal(L_y/K)$ respectively. We may view these as elements of $\Gal(L_{xy}/K)$ via some fixed isomorphism $\Gal(L_{xy}/K)\cong \Gal(L_x/K)\times \Gal(L_y/K)$. We may interpret $\Gal(L_{xy}/K)$ as the Galois group of the torsor $\Spec \OO_{L_x} \otimes_{\OO_K} \OO_{L_y}$ over $\Spec \OO_K$. The maps
    \[
        \begin{split}
        	& i_x\colon L_y^\times \to L_{xy}^\times\,,  \quad i_x\colon \Div(L_y)\to \Div(L_{xy}) \\
            & i_y\colon L_x^\times \to L_{xy}^\times\,, \quad i_y\colon \Div(L_x)\to \Div(L_{xy})     
        \end{split}
    \]
are induced by the maps     
    \[
        \begin{split}
            i_y\colon \OO_{L_x} & \to \OO_{L_x}\otimes_{\OO_K}\OO_{L_y} \\
            a & \mapsto a\otimes 1
        \end{split}   
        \quad\mbox{ and }\quad
        \begin{split}
            i_x\colon \OO_{L_y} & \to \OO_{L_x}\otimes_{\OO_K}\OO_{L_y} \\
            b & \mapsto 1\otimes b
        \end{split} \,.
    \]

We are now in a position to state our main results, and complete step (IV):

\begin{lemma}\label{lem:formula}
Let $L_x, L_y\supset K$ be two unramified extensions of degree $p$ representing elements $x, y\in H^1(X,\ZZ/p\ZZ)$. 
Then, for any $w\in H^1(X,P_y)$ and any $(b,a,J,I)$ representing an element in $H^1(X,D(P_y))$, there exist elements 
    \[
        a_1, b_1\in L_{xy}^\times\,, \ a_2\in L_x^\times\,, \ 
        I_1\in \Div(L_{xy})\,, \ I_2\in \Div(L_x)\,, 
    \]
such that $x\cupp w\in H^2(X,P_y)\cong H^1(X,D(P_y))^\sim$ satisfies the equality
    \[
        \langle x\cupp w, (b,a,J,I)\rangle=\langle w, 
        \begin{pmatrix}
            \alpha & J_1 & J_2
        \end{pmatrix}\rangle 
    \]
where 
    \[
        \begin{split}
            i_x(\alpha) & = (1-\sigma_y)(\Gamma_ya_1-\Gamma_xb_1)\,, \\
            J_1 & = N_x(I_1)-\frac{p(p-1)}{2}I\,,    \\
            i_x(J_2) & = \divis(a_1)+(1-\sigma_y)I_1-i_y(I_2)\,,
        \end{split}
    \]
and 
    \[
        \begin{split}
            (1)\quad & i_x(b)+N_x(b_1) = 0 \,, \\
            (2)\quad & i_x(a)+(1-\sigma_y)b_1 = (1-\sigma_x)a_1+i_y(a_2) \,, \\
            (3)\quad & \divis(b_1)-i_x(I)+(1-\sigma_x)I_1 = 0 \,, \\
            (4)\quad & \divis(a_2)+i_x(J)+(1-\sigma_x)I_2 = 0 \,.
        \end{split}
    \]
\end{lemma}

\begin{remark}
    Note that relation (2) of Lemma \ref{lem:formula} implies that $N_x(a_2)=N_y(a)$.  
\end{remark}
\begin{remark}
Despite the authors' best efforts, we are unable to find explicit expressions for the elements $\alpha$ and $J_2$ in the above formula. Clearly, an explicit expression for these elements would greatly simplify the formula in Lemma \ref{lem:formula}.
\end{remark}
\begin{proof}

Take $(b,a,J,I)^T\in \ker \psi\subseteq 
    (L_y^\times)^2 
    \oplus
    \Div(K)
    \oplus
    \Div(L_y)$. This maps via $\iota_1\circ\iota_0$ to the elements
    \[
        \begin{pmatrix}
            (1-\sigma_y)i_x(a) &
            i_x(b) &
            0 &
            0 &
            0 &
            0 &
            i_x(I) &
            0 &
            0 &
            0
        \end{pmatrix}^T\,.
    \]
We want to modify this element by the image of $\varphi_1$ to get an element which lies in the image of $\iota_2$. 
Writing everything additively, we have the relations
    \begin{equation}\label{eq:rel1}
        \begin{split}
            (1)\quad & (1-\sigma_y)(b-\Gamma_ya) = 0 \\
            (2)\quad & \divis(b)+pI = 0 \,.
        \end{split}
    \end{equation}       

The second relation implies that $\divis(i_x(b))=N_x(-i_x(I))$ and hence Hasse's norm theorem says that $i_x(b)+N_x(b_1)=0$ for some $b_1\in L_{xy}^\times$. We have $N_x=p+(\sigma_x-1)\Gamma_x$ and hence we get 
    \begin{equation}\label{eq:missing}\scalemath{0.8}{
        \varphi_1\begin{pmatrix}
            0 \\
            \Gamma_xb_1 \\
            b_1 \\
            0 \\
            -I \\
            0 \\
            0
        \end{pmatrix}+
        \begin{pmatrix}
            (1-\sigma_y)i_x(a) \\ 
            i_x(b) \\
            0 \\ 
            0 \\
            0 \\
            0 \\
            i_x(I) \\
            0 \\
            0 \\
            0
        \end{pmatrix}= 
        \begin{pmatrix}
            (\sigma_y-1)^2b_1 \\
            N_x(b_1) \\ 
            0 \\
            (1-\sigma_y)\Gamma_xb_1 \\ 
            0 \\ 
            (1-\sigma_y)^2(1-\sigma_x)b_1 \\ 
            -i_x(I) \\ 
            \divis(\Gamma_xb_1) \\ 
            \divis(b_1)-i_x(I) \\ 
            0 
        \end{pmatrix}+
        \begin{pmatrix}
            (1-\sigma_y)i_x(a) \\ 
            i_x(b) \\
            0 \\ 
            0 \\
            0 \\
            0 \\
            i_x(I) \\
            0 \\
            0 \\
            0
        \end{pmatrix}=
        \begin{pmatrix}
            (1-\sigma_y)i_x(a)+(1-\sigma_y)^2b_1 \\ 
            0 \\
            0 \\
            (1-\sigma_y)\Gamma_xb_1 \\ 
            0 \\ 
            (1-\sigma_y)^2(1-\sigma_x)b_1 \\
            0 \\
            \divis(\Gamma_xb_1) \\ 
            \divis(b_1)-i_x(I) \\ 
            0 
        \end{pmatrix}}\,.
    \end{equation}
We have 
    \[
        \begin{split}
            N_x((1-\sigma_y)i_x(a)+(1-\sigma_y)^2b_1) & = p(1-\sigma_y)i_x(a)-(1-\sigma_y)^2i_x(b) \\
            & = (1-\sigma_y)^2i_x(\Gamma_ya-b) \\
            & = 0  
        \end{split} 
    \]
and hence Lemma \ref{lemma:int-norm-ker} when $x \neq y$ and Remark \ref{rmk:equal} implies that there is an $a_1\in L_{xy}^\times$ such that 
    \[
        (1-\sigma_y)i_x(a)+(1-\sigma_y)^2b_1 - (1-\sigma_y)(1-\sigma_x)a_1=0\,.    
    \]
Note that $(1-\sigma_y)^2(1-\sigma_x)b_1-(1-\sigma_y)(1-\sigma_x)^2a_1=(1-\sigma_y)(1-\sigma_x)i_x(a)=0$ and if we add the element 
    $
        \varphi_1
        \begin{pmatrix}
            0 & 0 & 0 & a_1 & 0 & 0
        \end{pmatrix}^T
    $
we get 
    \[
        \begin{pmatrix}
            -(1-\sigma_y)(1-\sigma_x)a_1 \\
            0 \\ 
            0 \\
            (N_y-p)a_1 \\
            0 \\
            -(1-\sigma_y)(1-\sigma_x)^2a_1 \\
            0 \\ 
            0 \\
            0 \\
            \divis(a_1) 
        \end{pmatrix}+
        \begin{pmatrix}
            (1-\sigma_y)i_x(a)+(1-\sigma_y)^2b_1 \\ 
            0 \\
            0 \\
            (1-\sigma_y)\Gamma_xb_1 \\ 
            0 \\ 
            (1-\sigma_y)^2(1-\sigma_x)b_1 \\
            0 \\
            \divis(\Gamma_xb_1) \\ 
            \divis(b_1)-i_x(I) \\ 
            0 
        \end{pmatrix}=
        \begin{pmatrix}
            0 \\ 
            0 \\
            0 \\
            (1-\sigma_y)(\Gamma_xb_1-\Gamma_ya_1) \\ 
            0 \\ 
            0\\
            0 \\
            \divis(\Gamma_xb_1) \\ 
            \divis(b_1)-i_x(I) \\ 
            \divis(a_1) 
        \end{pmatrix}\,.
    \]
Note that $N_x(\divis(b_1)-i_x(I))=-i_x(\divis(b)+pI)=0$ and hence there exists an $I_1$ such that $\divis(b_1)-i_x(I)+(1-\sigma_x)I_1=0$. Adding $\varphi_1\begin{pmatrix}
    0 & 0 & 0 & 0 & 0 & 0 & I_1
\end{pmatrix}^T$ we get  
    \[
        \begin{pmatrix}
            0 \\
            0 \\ 
            0 \\
            0 \\
            0 \\
            0 \\
            0 \\ 
            pI_1 \\
            (1-\sigma_x)I_1 \\
            (1-\sigma_y)I_1 
        \end{pmatrix}+
        \begin{pmatrix}
            0 \\ 
            0 \\
            0 \\
            (1-\sigma_y)(\Gamma_xb_1-\Gamma_ya_1) \\ 
            0 \\ 
            0\\
            0 \\
            \divis(\Gamma_xb_1) \\ 
            \divis(b_1)-i_x(I) \\ 
            \divis(a_1) 
        \end{pmatrix}=
        \begin{pmatrix}
            0 \\ 
            0 \\
            0 \\
            (1-\sigma_y)(\Gamma_xb_1-\Gamma_ya_1) \\ 
            0 \\ 
            0\\
            0 \\
            \divis(\Gamma_xb_1)+pI_1 \\ 
            0 \\ 
            \divis(a_1) +(1-\sigma_y)I_1
        \end{pmatrix}\,.
    \]
This element lies in the kernel of $\varphi_2$ and hence $(1-\sigma_x)((1-\sigma_y)(\Gamma_xb_1-\Gamma_ya_1))=0$. We also have 
    \[
        (1-\sigma_x)(\divis(\Gamma_xb_1)+pI_1)=\divis(i_x(b)+pb_1)-p(\divis(b_1))+pi_x(I)=0\,.    
    \]
Finally, 
    \[
        \begin{split}
            (1-\sigma_y)(1-\sigma_x)(\divis(a_1) +(1-\sigma_y)I_1) & = (1-\sigma_y)i_x(a+(1-\sigma_y)I) \\
            & = -(1-\sigma_y)i_y(J) \\
            & = 0\,.
        \end{split}
    \]
Hence there exist elements $I_2\in \Div(L_x)$ and $J_2\in \Div(L_y)$ such that $\divis(a_1) +(1-\sigma_y)I_1-i_y(I_2)=i_x(J_2)$. Adding $\varphi_1\begin{pmatrix}
    0 & 0 & 0 & 0 & 0 & I_2 & 0
\end{pmatrix}^T$ we get 
    \[
        \begin{pmatrix}
            0 \\ 
            0 \\
            0 \\
            (1-\sigma_y)(\Gamma_xb_1-\Gamma_ya_1) \\ 
            0 \\ 
            0\\
            0 \\
            \divis(\Gamma_xb_1)+pI_1 \\ 
            0 \\ 
            \divis(a_1)+(1-\sigma_y)I_1+i_y(I_2)
        \end{pmatrix}    
    \]
which is of the form $\iota_2\begin{pmatrix}
    \alpha & J_1 & J_2
\end{pmatrix}$.   
\end{proof}

Note that the short exact sequence 
    \[
        0\to \ZZ/p\ZZ \to P_y \to \ZZ/p\ZZ \to 0    
    \]
induces, after taking Cartier duals, a long exact sequence  
    \[
        \dots\to H^1(X,D(\ZZ/p\ZZ)) \xrightarrow{
            \begin{pmatrix}
                i_y & 0 \\
                0 & 0 \\
                0 & 0 \\
                0 & i_y
            \end{pmatrix}
        } H^1(X, D(P_y)) \xrightarrow{
            \begin{pmatrix}
                0 & N_y & 0 & 0 \\
                0 & 0 & 1 & 0 
            \end{pmatrix}
        } H^1(X,D(\ZZ/p\ZZ))\xrightarrow{c_y^\sim} \dots    
    \]
    \[
        \xrightarrow{c_y^\sim} H^2(X,D(\ZZ/p\ZZ)) \xrightarrow{
            \begin{pmatrix}
                0 \\ i_y \\ 0
            \end{pmatrix}
        } H^2(X, D(P_y)) \xrightarrow{
            \begin{pmatrix}
                0 & 0 & N_y
            \end{pmatrix}
        } H^2(X,D(\ZZ/p\ZZ))\xrightarrow{c_y^\sim} H^3(X,D(\ZZ/p\ZZ))   \,.
    \]
Splitting this long exact sequence into short exact sequences of $\FF_p$-vector spaces, we get non-canonical isomorphisms 
    \[
        \begin{split}
            H^1(X, D(P_y)) &  \cong H^1(X,D(\ZZ/p\ZZ))/c_y^\sim \oplus \ker c_y^\sim\,, \\
            H^2(X, D(P_y)) &  \cong \Cl(K)/(p,c_y^\sim) \oplus N_y(\Cl(L_y))/p\,.
        \end{split}    
    \]
Note that when $p$ is odd and $K$ is an imaginary quadratic field, since $c_y^\sim$ is then the zero map and since $H^1(X,D(\ZZ/p\ZZ))$ then identifies with $\Cl(K)[p]$, the $p$-torsion of $\Cl(K)$, we get, after choices of splittings, isomorphisms
\[
    \begin{split}
        H^1(X, D(P_y)) & \cong  \Cl(K)[p] \oplus \ker c_y^\sim\,, \\
        H^2(X, D(P_y)) & \cong \Cl(K)/p \oplus N_y(\Cl(L_y))/p\,.
    \end{split}    
\]

Now choose a section $\zeta\colon N_y(\Cl(L_y))/p \to H^2(X, D(P_y))$, inducing a splitting as above, and let $\rho\colon H^2(X, D(P_y))\to H^2(X, D(P_y))$ be the projection which is the composition 
    \[
        H^2(X, D(P_y)) \to N_y(\Cl(L_y))/p \xrightarrow{\zeta} H^2(X, D(P_y))\,.  
    \]

\begin{remark} \label{rmk:split}
    Let $z\in H^1(X,\ZZ/p\ZZ)\cong H^2(X,D(\ZZ/p\ZZ))^\sim\cong \Hom(\Cl(K)/p,\frac{1}{p}\ZZ/\ZZ)$. If $y\cupp z=0$, then, by using the above splitting, we may define a lift $$w_z\in H^1(X,P_y)\cong H^2(X, D(P_y))^\sim = \Hom(\Cl(K)/c_y^\sim \oplus N_y(\Cl(L_y))/p, \frac{1}{p}\ZZ/\ZZ)$$ of $z$ by $w_z(s,t)=z(s)$. We call this lift $w_z$ the \emph{canonical lift of }$z$. 
\end{remark}
 
The following theorem simplifies the formula of Lemma \ref{lem:formula} by using the above splitting.

\begin{theorem}\label{thm:formula2.0}
    Let $L_x, L_y\supset K$ be two unramified extensions of degree $p$ representing elements $x, y\in H^1(X,\ZZ/p\ZZ)$ and let $z\in H^1(X,\ZZ/p\ZZ)$ be such that $y\cupp z = 0$. Choose a section $\zeta\colon N_y(\Cl(L_y))/p \to H^2(X, D(P_y))$ of the canonical map $H^2(X, D(P_y)) \to N_y(\Cl(L_y))/p$ inducing a projection $\rho\colon H^2(X, D(P_y)) \to H^2(X, D(P_y))$ as above.
    Then, for any lift $w_z\in H^1(X,P_y)$ of $z$ the element and any $(b,a,J,I)$ representing an element in $H^1(X,D(P_y))$, there exist elements   
    \[
            a_1, b_1\in L_{xy}^\times\,, \ a_2, a_3\in L_x^\times\,, \ 
            I_1\in \Div(L_{xy})\,, \ I_2\in \Div(L_x)\,, 
        \]
        such that $x\cupp w_z\in H^2(X,P_y)\cong H^1(X,D(P_y))^\sim$ satisfies the equality
        \[
            \langle x\cupp w_z, (b,a,J,I)\rangle = \langle z, J_K\rangle 
        \]
   where $i_y(J_K) = J_1-\Gamma_yJ_2+i_y(N_y(J_3))-\frac{p(p-1)}{2} i_y(J_D)+\divis(\alpha')$, 
        \[
            \begin{split}
                i_x(\alpha')+\Gamma_xb_1 & = \Gamma_ya_1+i_y(a_3)\,, \\
                J_1 & = N_x(I_1)-\frac{p(p-1)}{2}I\,,    \\
                i_x(J_2) & = \divis(a_1)+(1-\sigma_y)I_1-i_y(I_2)\,, \\
            \end{split}
        \]
    and 
        \[
            \begin{split}
                (1)\quad & i_x(b)+N_x(b_1) = 0 \,, \\
                (2)\quad & i_x(a)+(1-\sigma_y)b_1 = (1-\sigma_x)a_1+i_y(a_2) \,, \\
                (3)\quad & i_x(I) = \divis(b_1)+(1-\sigma_x)I_1 \,, \\
                (4)\quad & i_x(J) = \divis(-a_2)+(1-\sigma_x)(-I_2) \,, \\
                (5)\quad & J_2-\rho(J_2)=\divis(a_4)+i_y(J_D)+(1-\sigma_y)J_3\,,
            \end{split}
        \] 
    for some element $a_4\in L_y^\times$ and some fractional ideal $J_D\in \Div(K)$.

\end{theorem}

Recall that $c_x^\sim\colon H^1(X, D(\ZZ/p\ZZ))\to H^2(X, D(\ZZ/p\ZZ))$ denotes the Pontryagin dual of the cup product $c_x\colon H^1(X,\ZZ/p\ZZ)\to H^2(X,\ZZ/p\ZZ)$ with an element $x\in H^1(X,\ZZ/p\ZZ)$. The element       
    \[
        \langle x,y,z\rangle \in H^2(X,\ZZ/p\ZZ)/(H^1(X,\ZZ/p\ZZ)\cupp z+x\cupp H^1(X,\ZZ/p\ZZ))
    \]
corresponds, under Pontryagin duality, to a functional 
    \[
        \ker c_x^\sim\cap \ker c_z^\sim \to \frac{1}{p}\ZZ/\ZZ\,.    
    \]
We will describe Massey products as such functionals. 

\begin{corollary}\label{cor:massey}
    Let $x,y,z\in H^1(X, \ZZ/p\ZZ)$ be non-zero elements such that $x\cupp y = y\cupp z = 0$, and assume further that $y\cupp H^1(G,\ZZ/p\ZZ)\subseteq x\cupp H^1(G,\ZZ/p\ZZ)+H^1(G,\ZZ/p\ZZ)\cupp z$. Then, for any $(a',J)\in \ker c_x^\sim\cap \ker c_z^\sim$,
        \begin{equation}
            \langle \langle x,y,z\rangle , (a',J)\rangle = \langle z, J_K^x\rangle+\langle x, J_K^z\rangle\,,
        \end{equation}
    where $\langle z, J_K^x\rangle$ and $\langle x, J_K^z\rangle$ are as in Theorem \ref{thm:formula2.0}. 
\end{corollary}

\begin{proof}[{Proof of Theorem \ref{thm:formula2.0}}]
    We want to see what the element $(\alpha, J_1,J_2)$ from Lemma \ref{lem:formula} looks like with respect to the splitting  of $H^2(X, D(P_y))$ given in the discussion preceding Remark \ref{rmk:split}.  
    
    We may write 
        \[
            (\alpha, J_1, J_2) = (0, J_1+\divis(\alpha')-\Gamma_y\rho(J_2), J_2-\rho(J_2)) + (0,\Gamma_y \rho(J_2), \rho(J_2))    
        \]
    where $(1-\sigma_y)i_x(\alpha')=(1-\sigma_y)(\Gamma_xb_1-\Gamma_ya_1)$. Now $J_2-\rho(J_2)$ maps to zero in $N_y(\Cl(L_y))/p$ and hence we may write 
        \begin{equation}\label{eq:last}
            J_2-\rho(J_2) = \divis(a_4)+i_y(J_D)+(1-\sigma_y)J_3 
        \end{equation}
    for some element $a_4\in L_y^\times$ and some ideals $J_D\in \Div(K)$, $J_3\in \Div(L_y)$. After reducing by the image of the map $\psi$ on page 17, we get in cohomology 
        \[
            (0, J_1+\divis(\alpha')-\Gamma_y\rho(J_2), J_2-\rho(J_2)) = (0, J_1+\divis(\alpha')-\Gamma_y\rho(J_2)-\divis(\Gamma_y a_4)-pJ_3, 0)\,.  
        \]
    Applying $-\Gamma_y$ to the equality (\ref{eq:last}) we get 
        \[
            -\Gamma_y(J_2-\rho(J_2)) = -\divis(\Gamma_y a_4)+\frac{p(p-1)}{2} i_y(J_D)-(p-i_y(N_y J_3))    
        \]
    and hence 
        \[
            -\Gamma_y\rho(J_2)-\divis(\Gamma_y a_4)-pJ_3 = -\Gamma_y J_2+i_y(N_y J_3)-\frac{p(p-1)}{2} i_y(J_D)\,.
        \]
    This completes the proof. 
\end{proof}

\subsection*{The non-connected case}
Now let $p$ be an odd prime. When $x=y$ in Theorem \ref{thm:formula2.0}, the formula for the Massey product becomes much simpler, if we apply Proposition \ref{prop:half-cup}. Before we state the result we will show what the inclusions $i_x, i_y$ and the Galois elements $\sigma_x, \sigma_y$ look like in this case. We start with the 
$\ZZ/p\ZZ$-torsor $\Spec \OO_{L_x}$, and we let $\sigma_x$ be a generator for the Galois group of $\Gal(L_x/K)$.

Now let $\OO_{L_x}\ZZ/p\ZZ$ be the Hopf algebra representing the constant group scheme $\ZZ/p\ZZ$ over $\Spec \OO_{L_x}$.  Thus, $\OO_{L_x}\ZZ/p\ZZ$ is finite free on idempotent elements $e_0,\dots, e_{p-1}$ such that $e_0+\dots+e_{p-1}=1$. 
We identify $\OO_{L_x}\otimes_{\OO_K}\OO_{L_x}$ with $\OO_{L_x}\ZZ/p\ZZ$ via the isomorphism $\varphi$ given on generators by
    \[
        \begin{split}
            \OO_{L_x}\otimes_{\OO_K}\OO_{L_x} & \to \OO_{L_x}\ZZ/p\ZZ  \\
            a\otimes b & \mapsto a\Delta(b)\,
        \end{split} 
    \]
    where $\Delta(b) = \sum_{i=0}^{p-1} \sigma_x^i(b) e_i$. 
Hence we see that the inclusions
    \[
        \begin{split}
            i_y\colon \OO_{L_x} & \to \OO_{L_x}\otimes_{\OO_K}\OO_{L_x} \\
            a & \mapsto a\otimes 1
        \end{split}   
        \quad\mbox{ and }\quad
        \begin{split}
            i_x\colon \OO_{L_x} & \to \OO_{L_x}\otimes_{\OO_K}\OO_{L_x} \\
            b & \mapsto 1\otimes b
        \end{split} 
    \]
from Subsection \ref{subseq:comp} correspond to the maps 
    \[
        i_y\colon a \mapsto a=\sum_{i=0}^{p-1}ae_i \quad\mbox{ and }\quad
        i_x\colon b \mapsto \sum_{i=0}^{p-1}\sigma_x^i(b)e_i\,.
    \]
Seen as a torsor over $\OO_K$, $\OO_{L_x}\otimes_{\OO_K}\OO_{L_x}$ is a $\ZZ/p\ZZ \times \ZZ/p\ZZ$ torsor; we let $\sigma_x$ (by abuse of notation) be a generator of the first component and $\sigma_y$be a generator of the second component. 
Then it is easy to see that one can choose $\sigma_x,\sigma_y$ such that we have, once again by abuse of notation, that $\sigma_x(a\otimes 1)=\sigma_x(a)\otimes 1$ and $\sigma_y(1\otimes b)=1\otimes \sigma_x(b)$. Via the isomorphism $\varphi\colon \OO_{L_x}\otimes_{\OO_K}\OO_{L_x}\to \OO_{L_x}\ZZ/p\ZZ$ we then get 
    \[
        \begin{split}
            \sigma_x(\sum_{i=0}^{p-1}c_ie_i)=\sum_{i=0}^{p-1}\sigma_x(c_{i-1})e_i\quad\mbox{ and }\quad \sigma_y(\sum_{i=0}^{p-1}c_ie_i)= \sum_{i=0}^{p-1}c_{i+1}e_i\,
        \end{split}    
    \]
    where all the indices are taken modulo $p$ in the obvious sense.

\begin{lemma}\label{lem:non-connected}
    Let $L_x\supset K$ be an unramified extension of degree $p$ representing an element $x\in H^1(X,\ZZ/p\ZZ)$, where $p$ is an odd prime. 
    Let $z\in H^1(X,\ZZ/p\ZZ)$ be an element such that $x\cupp z=0$ and let $w_z\in H^1(X,P_x)$ be the canonical lift of $z$. Then the element $x\cupp w_z\in H^2(X,P_x)\cong H^1(X,D(P_x))^\sim$ satisfies, for any $(b,a,J,I)\in H^1(X,D(P_x))$, the equality
        \[
            \langle x\cupp w_z, (b,a,J,I)\rangle=
            \begin{cases} 
                \langle z, 2N_x(I')+2J\rangle & \mbox{ if }p = 3\,, \\
                \langle z, 2N_x(I')\rangle & \mbox{ if }p > 3\,,
            \end{cases}
        \]
    where $I'\in \Div(L_x)$ is a fractional ideal such that $I = \divis(u)-i_x(\frac{p-1}{2}J)+(1-\sigma_x)I'$ for some element $u\in L_x^\times$ such that $i_x(N_x(u))=\Gamma_xa-b$. 
\end{lemma}

Before we prove this lemma, we want to state an immediate consequence which is one of our main results of this section. 

\begin{theorem}\label{thm:main-simple}
    Let $L_x\supset K$ be an unramified extension of degree $p$ representing an element $x\in H^1(X,\ZZ/p\ZZ)$, where $p$ is an odd prime. Let $y\in H^1(X,\ZZ/p\ZZ)$ be an element such that $x\cupp y=0$. Then, for any $(a',J)\in \ker c_x^\sim\cap \ker c_y^\sim\subseteq H^1(X, D(\ZZ/p\ZZ))$, the equality
        \[
            \langle \langle x,x,y\rangle, (a',J)\rangle =
            \begin{cases} 
                \langle y, N_x(I')+J\rangle & \mbox{ if }p = 3\,, \\
                \langle y, N_x(I')\rangle & \mbox{ if }p > 3\,,
            \end{cases} 
        \]
    holds, where $(t, I')\in L_x^\times\oplus \Div(L_x)$ is any element satisfying the following two equalities:
    \begin{itemize}
        \item $(1-\sigma_x)^2I'+\divis(t)+i_x(J)=0$ and 
        \item $N_x(t)=a'$. 
    \end{itemize}
\end{theorem}

\begin{proof}
    This is an immediate consequence of Proposition \ref{prop:half-cup} and Lemma \ref{lem:non-connected}. The only thing to note is that we may modify the classes $(b,a,J,I)\in H^1(X,D(P_x))$ in Lemma \ref{lem:non-connected} by the image of 
        \[
            \begin{pmatrix}
                0 & -p \\
                -i_x & \sigma_x-1 \\
                \divis & 0 \\
                0 & \divis
            \end{pmatrix}\colon K^\times\oplus L_x^\times \to (L_x^\times)^2\oplus \Div(K)\oplus \Div(L_x)\,.
        \]
    Applying this map to the element 
        $
            \begin{pmatrix}
                0 & -u
            \end{pmatrix}
        $
    (additive notation), with $u$ as in Lemma \ref{lem:non-connected}, we may modify $(b,a,J,I)$ to get $(b+pu, a+(1-\sigma_x)u, J, I-\divis(u))$. Putting $t=a+(1-\sigma_x)u$, it is easy to see that $N_x(t)=a'$ and $(1-\sigma_x)^2I'+\divis(t)+i_x(J)=0$. On the other hand, given $t$ and $I'$ satisfying these two relations, the element 
        \[
            \left(\Gamma_x t, t, J, (1-\sigma_x)I'-i_x\left(\frac{p-1}{2}J\right)\right)
        \]
    defines a class in $H^1(X,D(P_x))$.  
\end{proof}

\begin{proof}[Proof of Lemma \ref{lem:non-connected}]
    The idea is to proceed exactly as in the proof of Lemma \ref{lem:formula} with variables named thereafter, except that now we may choose the elements more explicitly. Since $L_x=L_y$ we write $\sigma_x=\sigma_y=\sigma$ to denote the generator of $\Gal(L_x/K)$. Similarly, we write $N=N_x=N_y$ and $\Gamma=\Gamma_x=\Gamma_y$. However, we will still distinguish between $\sigma_x$ and $\sigma_y$ as elements of $\Gal(L_{xx}/K)$, where by $L_{xx}$, we mean $p$ copies of $L_x$:
        \[
            L_{xx} = L_x\times \dots \times L_x\,.    
        \]
    Thus, whenever we use subscripts like $N_x, \Gamma_y$ etc., we mean as morphisms $L_{xx}^\times\to L_{xx}^\times$. 

    Let us set $b_1= (-b, 0, \dots, 0)^T$. We now want to find $a_1=(x_0,\dots, x_{p-1})^T$ satisfying the equation 
        \[
            (1-\sigma_x)(1-\sigma_y)a_1 = (1-\sigma_x)i_y(a)+(1-\sigma_x)^2b_1\,,     
        \]
    that is, 
        \[
            \begin{pmatrix}
                x_0-x_1+\sigma(x_0-x_{p-1}) \\
                x_1-x_2+\sigma(x_1-x_{0}) \\
                . \\
                \vdots\\
                .\\
                x_{p-1}-x_0+\sigma(x_{p-1}-x_{p-2})
            \end{pmatrix} =
            \begin{pmatrix}
                (1-\sigma)a \\
                (1-\sigma)a \\
                .\\
                \vdots\\
                .\\
                (1-\sigma)a
            \end{pmatrix}+
            \begin{pmatrix}
                -b \\
                2\sigma(b) \\
                -\sigma^2(b) \\
                0 \\
                \vdots \\
                0
            \end{pmatrix}   \,.
        \]
    Since $a_1$ is unique only up to elements in the image of $i_x$ and $i_y$, we may hope to find a solution $a_1$ under the assumption $x_0=x_{p-1}=0$. Indeed, one finds the solution  
        \[
            a_1 = \begin{pmatrix}
                0 \\
                -(1-\sigma)a +b\\
                \left(\sum_{n=3}^{p-1}(1-\sigma^n)\right)a \\
                \vdots \\
                (1-\sigma^{p-1})a+(1-\sigma^{p-2})a \\
                (1-\sigma^{p-1})a \\ 
                0
            \end{pmatrix} =
            \begin{pmatrix}
                0 \\
                -(1-\sigma)a \\
                \left(\sum_{n=3}^{p-1}(1-\sigma^n)\right)a \\
                \vdots \\
                (1-\sigma^{p-1})a+(1-\sigma^{p-2})a \\
                (1-\sigma^{p-1})a \\ 
                0
            \end{pmatrix}+
            \begin{pmatrix}
                0 \\
                b \\
                0 \\ 
                \vdots \\
                0 \\
                0 \\
                0
            \end{pmatrix}\,.
        \]
    Let us denote the first vector of the right-hand side by $a_1'$. Then 
        \[
            a_1' = (1-\sigma)\begin{pmatrix}
                0 \\ 
                -a \\
                -(p-3)\sigma^{p-1}a-\dots-2\sigma^4a-\sigma^3a \\
                \vdots \\
                -2\sigma^{p-1}a-\sigma^{p-2}a  \\
                -\sigma^{p-1}a \\
                0
            \end{pmatrix}=: (1-\sigma)a_1''\,.
        \]
    The next step is to find an $I_1$ satisfying $\divis(b_1)-i_y(I)+(1-\sigma_y)I_1=0$. This is obtained by choosing 
        \[
            I_1 = \begin{pmatrix}
                -(p-2)I & I & 0 & -I & -2I & \dots & -(p-3)I
            \end{pmatrix}^T\,.    
        \]
    Now we want to find $J_2$ and $I_2$ satisfying the relation $\divis(a_1)+(1-\sigma_x)I_1=i_y(J_2)+i_x(I_2)$. The first and last rows of this vector are given by 
        \[
            -(p-2)I+(p-3)\sigma(I) = -(p-2)I+(p-4)\sigma(I) + \sigma(I) 
        \] 
    and 
        \[
            -(p-3)I+(p-4)\sigma(I) = -(p-2)I+(p-4)\sigma(I) + I\,,      
        \]
    which is enough information to determine $I_2$ and $J_2$.
    Hence we conclude that we may choose $J_2= -(p-2)I+(p-4)\sigma(I)$ and $I_2= \sigma(I)$. 
    As in Lemma \ref{lem:formula} we may choose   
        \[
            \begin{split}
                J_1 & = N_y(I_1)-\frac{p(p-1)}{2}I \\
                    & = I - \left(\sum_{n=1}^{p-2} n\right)I -\frac{p(p-1)}{2}I \\
                    & = I - \frac{(p-1)(p-2)+p(p-1)}{2}I \\
                    & = p(2-p)I \,.
            \end{split}       
        \]
    The element $\alpha$ may be obtained, for instance, as the second row (all rows are equal) of the vector 
        \[
            \begin{split}
                i_y(\alpha) & = (1-\sigma_x)(\Gamma_x a_1-\Gamma_y b_1) \\
                            & = (p-N_x)a_1+(1-\sigma_x)(\sigma_y+2\sigma_y^2+\dots+(p-1)\sigma_y^{p-1})b_1\,.
            \end{split}    
        \]
    The $b$'s on the second row will cancel and we see that $\alpha$ is the second row of $(p-N_x)a_1'=(1-\sigma)(p-N_x)a_1''$. We want to reduce modulo the image of the matrix $\psi$ on page 17. Thus we want to find an element $\alpha'$ such that $(1-\sigma)\alpha'=\alpha$. Hence $\alpha'$ may be chosen as the second row of $(p-N_x)a_1''$, which is equal to
        \[
            (-(p-1)+(p-3)\sigma^2+2(p-4)\sigma^3+3(p-5)\sigma^4\dots+(p-3)\sigma^{p-2})a\,.    
        \]
    The coefficients of $\theta := -(p-1)+(p-3)\sigma^2+2(p-4)\sigma^3+3(p-5)\sigma^4\dots+(p-3)\sigma^{p-2}$ sums to $0$ mod $p$, if $p>3$ and to $1$ mod $p$ if $p=3$. This means that if $p>3$, then $\divis(\theta a)+(1-\sigma)\theta I=0$ modulo $p$\,.

    We have 
        \[
            0=\Gamma_x(\divis(a)+i_x(J)+(1-\sigma_x)I)=i_x(\divis(b')-\frac{p(p-1)}{2}J-N_x(I))\,,
        \] 
    where $i_x(b')=\Gamma_xa-b$, hence $\divis(b')-\frac{p(p-1)}{2}J-N_x(I)=0$. But then $b'$ is locally a norm and, by Hasse's norm theorem, there exists a $u\in L_x$ such that $b'=N_x(u)$. But then $N_x(\divis(u)-i_x(\frac{p-1}{2}J)-I)=0$ and, by Hilbert's theorem 90 for ideals, we get that there is an $I'\in \Div(L_x)$ such that $I=\divis(u)-i_x(\frac{p-1}{2}J)+(1-\sigma_x)I'$.   
    Reducing modulo the image of $\psi$ on page 17, for $p>3$, we get that 
        \[
            \begin{split}
                \begin{pmatrix}
                    \alpha \\
                    J_1 \\
                    J_2
                \end{pmatrix} 
                & =
                \begin{pmatrix}
                    (1-\sigma)\theta a  \\
                    -p(p-2)I\\
                    -(p-2)I+(p-4)\sigma(I)
                \end{pmatrix} \\
                & =
                \begin{pmatrix}
                    0 \\
                    -p(p-2)I-(1-\sigma)\theta I \\
                    -(p-2)I+(p-4)\sigma(I)
                \end{pmatrix} \\
                & =
                \begin{pmatrix}
                    0 \\
                    -p(p-2)I+p(p-4)I-(1-\sigma)\theta I \\
                    -2I
                \end{pmatrix}  \\
                & =
                \begin{pmatrix}
                    0 \\
                    -2pI-(1-\sigma)\theta I + 2\Gamma I +2N I'\\
                    0
                \end{pmatrix}\,.
            \end{split}
        \]
    We have 
        \[
            -(1-\sigma)\theta = (p-1)-(p-1)\sigma-(p-3)\sigma^2-(p-5)\sigma^3-\dots-(p-2(p-4)-3)\sigma^{p-2}+(p-3)\sigma^{p-1}    
        \]
    and hence 
        \[
            -2pI-(1-\sigma)\theta I+2\Gamma I +2N I' = -(p+1)N_x(I)+2N I'\,.    
        \]
    But $N(I)$ represents the trivial class in $\Cl(K)/p\Cl(K)$ and hence $\langle z, -(p+1)N(I)+2N I'\rangle=\langle z, 2N I'\rangle$. The case $p=3$ follows by the same argument but with an extra factor $J$ showing up. This is left to the reader. 
\end{proof}

\begin{remark} \label{rmk:vanish}
Let us note that when $x=y$ in the above formula, and $p=3$, then $\langle x,x,x \rangle$ coincides, up to indeterminacy, with the Bockstein homomorphism $H^1(X,\ZZ/p\ZZ) \to H^2(X,\ZZ/p\ZZ)$ induced from the short exact sequence $0 \to \ZZ/p\ZZ \to \ZZ/p^2\ZZ \to \ZZ/p\ZZ \to 0$. This agrees, up to sign, with \cite[Theorem 14]{Kraines-Higher}, as is expected. Note also that $\langle x,x,x \rangle$ vanishes up to indeterminacy if $p \geq 5$.
\end{remark}

\begin{remark} \label{rmk:sharifi}
Let us note that the formula in Theorem \ref{thm:main-simple} is highly reminiscent of the formula given in Theorem 5.4 of \cite{Sharifi-Massey}. 
Let us sketch the relation between Theorem \ref{thm:main-simple} and \cite[Theorem 5.4]{Sharifi-Massey} when $K$ contains all $p$th roots of unity and we have removed a set of places $S$ containing all places over the prime $p$. Indeed, with notation as in \cite[Theorem 5.4]{Sharifi-Massey}, choose $m=n=1$ and $k=2$. Then Sharifi's element $a$ corresponds to our element $x$ and Sharifi's element $b$ corresponds to our element $(a',J)$. 

We then put $y=t$ where $y$ is in Sharifi's notation. 
The requirement $D^{(k-1)}y\in \Omega^{\times p^m}$ is equivalent to the requirement that the ideal $\divis(D^{(k-1)}y)\OO_{L,S}$ is a $p$th power. To see that this is satisfied, we apply $D^1=\Gamma_x-N_x$ to the equality  
    \[
        \divis(y)+J\OO_{L,S}+(1-\sigma)^2I'=0\,.        
    \]
In Sharifi's notation, we put $\mathfrak{B}=-J\OO_{L,S}$ and $\mathfrak{N}=-I'$. 
\end{remark}


\subsection{A necessary and sufficient condition for the vanishing of Massey products} \label{subsec:necessary}

Let $p$ be an odd prime and $K$ an imaginary quadratic field with class group of $p$-rank 2. We will now show that all 3-fold Massey products vanish if and only if the $p$-rank of the class groups of $\Cl(L_x)$ and $\Cl(L_y)$ are at least 4.   Before we state and prove this result, we will set up some notation and prove a few preliminary lemmas. Note that the methods used here are inspired by those in \cite{Sharifi-Bockstein}.

Let $G$ be the Galois group of $L_x$ over $K$ and write $Y=\Spec \OO_{L_x}$. Let $\pi\colon Y\to X$ be the structure map and let $I\subseteq \pi_*\pi^*\ZZ/p\ZZ$ be the kernel of the norm map so that, by abuse of notation, $I$ corresponds to the augmentation ideal $I\subseteq \ZZ/p\ZZ[G]$ under the equivalence between locally constant sheaves split by $\pi$ and $G$-modules. 
We have, for every natural number $n$, a $G$-module $V_n=\ZZ/p\ZZ[G]/I^{n}$ and an exact sequence of $G$-modules 
    \[
        0\to I^{n}/I^{n+1}\to V_{n+1}\to V_n \to 0 \,.
    \]
For every $G_K^{ur}$-module $M$, let us use the abbreviation $H^i(M)=H^i(G_K^{ur}, M)$ in this section. 
We let $\partial_n\colon H^1(V_n)\to H^2(I^{n}/I^{n+1})\cong H^2(\ZZ/p\ZZ)$ denote the connecting homomorphism induced from the above short exact sequence.

For every $n$ such that $1\leq n \leq p$, we choose $(T-1)^{n-1}, (T-1)^{n-2}, \dots, 1$ as a basis for $V_n$ over $\FF_p$, where $T$ corresponds to the generator of $G$ given by $x$. Hence the action is via 
    \begin{equation}\label{eq:one-matrix}
        T\mapsto 
        \begin{pmatrix}
            1 & 1 & 0 & 0 & \dots & 0 \\
            0 & 1 & 1 & 0 & \dots & 0 \\
            0 & 0 & 1 & 1 & \dots & 0 \\
            \vdots & & \ddots & \ddots & & \vdots \\
            0 & \dots & & & 1 & 1 \\
            0 & \dots & & & 0 & 1
        \end{pmatrix}    \,.
    \end{equation}
The maps $I^{n}/I^{n+1}\to V_{n+1}$ and $V_{n+1}\to V_n$ then correspond to the inclusion of the first coordinate and the projection onto the last $n$ coordinates respectively. 

A more general result than the following lemma is proved in \cite{Sharifi-Bockstein}, but we give another proof for the sake of completeness. 

\begin{lemma}
    The image of $\partial_2$ is generated by the elements $\langle x,x,x\rangle$ and $\langle x,x,y\rangle$. 
\end{lemma}
\begin{proof}
    Let 
        \[
            w_z=
            \begin{pmatrix}
                t_z \\
                z
            \end{pmatrix} \colon G_K^{ur} \to V_2
        \]
    be a crossed homomorphism representing a lift of $z$, where $z$ is an element of $H^1(\ZZ/p\ZZ)$. Then we have for every $g,g'\in G_K^{ur}$
        \[
            \partial_2[w_z](g,g') = x(g)t_z(g')+t_x(g)z(g')    
        \]
    where $t_x\colon G_K^{ur}\to \ZZ/p\ZZ$ is the function obtained as index $(1,3)$ of the matrix $G_K^{ur}\to V_3$ induced by the homomorphism (\ref{eq:one-matrix}). This is indeed the Massey product $\langle x,x,z \rangle$.  If $z$ is non-trivial, clearly $\langle x,x,z \rangle$ lies in the module generated by $\langle x,x,x\rangle$ and $\langle x,x,y\rangle$. On the other hand, if $z=0$, $t_z(g)$ is a cocycle, so that $\partial_2[w_z]$ is zero since all cup products vanish by assumption.
\end{proof}

\begin{lemma}\label{lem:containment}
    For every $n\geq 2$, we have $\im \partial_{n-1} \subseteq \im \partial_{n}$. 
\end{lemma}

\begin{proof}
    For every $n\geq 2$, we have a commutative diagram of $G_K^{ur}$-modules
        \[
            \begin{tikzcd}
                0 \ar{r}&  I^{n-1}/I^n \ar{d}\ar{r} & V_{n}\ar{r}\ar{d} & V_{n-1}\ar{r}\ar{d} & 0 \\
                0 \ar{r}& I^{n}/I^{n+1} \ar{r} & V_{n+1} \ar{r} & V_n\ar{r} & 0 
            \end{tikzcd}    
        \] 
    where the vertical arrows are multiplication by $T-1$. This gives isomorphisms of (trivial) $G_K^{ur}$-modules $I^{n-1}/I^{n}\cong I^{n}/I^{n+1}\cong \ZZ/p\ZZ$ and hence we get a commutative diagram
        \[
            \begin{tikzcd}
               H^1(V_{n-1}) \ar{r}{\partial_{n-1}} \ar{d} & H^2(I^{n-1}/I^{n}) \ar[equals]{d} \\
               H^1(V_n) \ar{r}{\partial_{n}} & H^2(I^{n}/I^{n+1}) \,.
            \end{tikzcd}    
        \]  
     This proves that $\partial_{n-1}\colon H^1(V_{n-1})\to H^2(\ZZ/p\ZZ)$ factors through $\partial_n\colon H^1(V_{n})\to H^2(\ZZ/p\ZZ)$ and proves the claim. 
\end{proof}
    
    \begin{remark}\label{rk:ker-vs-indet}
        For every $n\geq 2$, the sequence 
            \[
                0 \to V_{n-1}\xrightarrow{T-1} V_n \xrightarrow{\varepsilon} \ZZ/p\ZZ\to 0    
            \] 
        is exact. In particular, we have an exact sequence 
            \[
                \dots \to H^1(V_{n-1})\to H^1(V_n)\to H^1(\ZZ/p\ZZ)\to \dots\,.    
            \]
    \end{remark}

    \begin{lemma}\label{lem:gen-for-conn}
        Let $2\leq n<p$ and let $w_x\in H^1(V_n)$ be an element mapping to the generator $x\in H^1(\ZZ/p\ZZ)$ if such a $w_x$ exists and otherwise put $w_x=0$. Similarly, let $w_y\in H^1(V_n)$ be an element mapping to the generator $y\in H^1(\ZZ/p\ZZ)$ if such $w_y$ exists and otherwise put $w_y=0$. Then $\im \partial_n$ is generated by $\partial_n(w_x)$ and $\partial_n(w_y)$ together with elements of $\im \partial_{n-1}$.  
    \end{lemma}
    
    \begin{proof}
        This follows immediately from Remark \ref{rk:ker-vs-indet} since $H^1(\ZZ/p\ZZ)$ is generated by $x$ and $y$. 
    \end{proof}

\begin{lemma}\label{lem:4-vanish}
    Let $p>3$ and assume that $3\leq n\leq p-1$ is odd. Then $\im \partial_n=\im \partial_{n-1}$.  
\end{lemma}

\begin{proof}
    The group $\Gal(K/\QQ)\cong \ZZ/2\ZZ$ acts by conjugation on $G_K^{ur}$. It also acts by conjugation on $G_x$, which induces an action on the modules $\ZZ/p\ZZ[G_x]$, $V_n$, and $I^{n}/I^{n+1}$, where the generator of $\Gal(K/\QQ)$ acts by sending $T$ to $T^{-1}$. This induces an action on cohomology groups and we get splittings into eigenspaces
        \[
            H^i(V_n) = H^i(V_n)^{+} \oplus H^i(V_n)^{-}   
        \]
    since $2$ is invertible modulo $p$. 

    The modules $I^{n}/I^{n+1}$ are isomorphic to $\ZZ/p\ZZ$ as $G_K^{ur}$-modules but are twisted versions of $\ZZ/p\ZZ$ in the sense that $1\in \ZZ/2\ZZ$ acts as $(-1)^n$ on $I^n/I^{n+1}$. Hence $H^2(I^{n}/I^{n+1})$ is fixed by the $\Gal(K/\QQ)$-action if $n$ is odd whereas the fixed points of $H^2(\ZZ/p\ZZ)$ are zero. The connecting homomorphism 
        \[
            \partial_n\colon H^1(V_n) \to H^2(I^n/I^{n+1})    
        \]        
    is $\Gal(K/\QQ)$-equivariant and any $\partial_n[f]$ is fixed by the Galois action. This means that the image of $\partial_n$ is generated by elements $\partial_n[f]$, where $[f]\in H^1(V_n)^{+}$. But these are all sent to zero under the map $H^1(V_n)\to H^1(V_1)=H^1(\ZZ/p\ZZ)^{-}$, which means $\partial_n[f]=\partial_{n-1}[f']$ for some $[f']\in H^1(V_{n-1})$, by Remark \ref{rk:ker-vs-indet}. \qedhere   
\end{proof}

\begin{lemma}\label{lem:xx-zero}
    If $2\leq n \leq p-1$, then there exists a lift $w_x\in H^1(V_n)$ of the generator $x\in H^1(\ZZ/p\ZZ)$. In particular, $\partial_n[w_x]=0$ for $2\leq n < p-1$. 
\end{lemma}

\begin{proof}
    This follows from the fact that the map $\ZZ/p\ZZ\to GL_{n+1}(\ZZ/p\ZZ)$ sending 1 to the matrix of Diagram (\ref{eq:one-matrix}) defines a group homomorphism if $n\leq p-1$. Let $M\colon G_K^{ur}\to GL_{n+1}(\ZZ/p\ZZ)$ be the induced homomorphism. Then we get functions $M_{i,j}\colon G_K^{ur}\to \ZZ/p\ZZ$ for every index $(i,j)$ and we define 
        \[
            w_x:= \begin{pmatrix}
                M_{1,2} \\
                M_{1,3} \\
                \vdots \\
                M_{1,n+1}
            \end{pmatrix} \colon G_K^{ur} \to V_{n}\,.   
        \]
    The fact that $M$ is a homomorphism implies that $w_x$ is a crossed homomorphism and since $M_{1,2}=x$, this proves the lemma. 
\end{proof}

\begin{lemma}\label{lem:non-surj}
    If $1\leq n < p-1$, then the connecting homomorphism $\partial_n$ is not surjective.
\end{lemma}

\begin{proof}
    We use induction and assume that $\partial_{n-1}$ is not surjective since this is the case for $n=2$. 

    If $\partial_{n-1}$ is zero, then Lemma \ref{lem:gen-for-conn} and Lemma \ref{lem:xx-zero} imply that the image of $\partial_n$ is generated by a single element $\partial_n[w_y]$, where $w_y\in H^1(V_n)$ is either a lift of the generator $y\in H^1(\ZZ/p\ZZ)$ or zero. Hence we get that $\partial_n$ is not surjective. 

    If $\partial_{n-1}$ has image of rank 1, then there is no lift $w_y\in H^1(V_n)$ of $y$. Hence the image of $\partial_n$ is generated by $\partial_n[w_x]$, which by Lemma \ref{lem:xx-zero}, is contained in $\im \partial_{n-1}$, which has rank 1. Thus $\im \partial_n$ has rank 1 which completes the proof. 
\end{proof}

\begin{theorem}\label{thm:crit-for-massey-vanish}
    Let $p$ be an odd prime and $K$ an imaginary quadratic field with class group of $p$-rank 2. Put $X=\Spec \OO_K$ and choose elements $x$ and $y$ forming a basis for $H^1(X, \ZZ/p\ZZ)$ with corresponding field extensions $L_x$ and $L_y$. Then the Massey products $\langle x,x,x\rangle$ and $\langle x,x,y\rangle$ both vanish if and only if $\Cl(L_x)$ has $p$-rank at least 4.  
\end{theorem}

\begin{proof}
First note that $\rk H^1(V_2)=3$ since all cup products vanish. Recall also that this means that there is no indeterminacy for 3-fold Massey products. 

Suppose that not all relevant 3-fold Massey products vanish. Then we must have $\rk(\ker \partial_2)\leq 2$. We will use induction to prove that $\rk(H^1(V_n))\leq 3$ for every $n\geq 2$. Now let $2\leq n\leq p-1$ and assume inductively that $\rk(\ker \partial_n)\leq 2$ and that $\rk(\im \partial_n) \geq 1$, the base being when $n=2$. We have an exact sequence 
    \[
        0 \to H^0(V_{n+1})\cong \FF_p\to H^1(I^n/I^{n+1})\to H^1(V_{n+1})\to \ker \partial_n \to 0    
    \]
and hence $\rk H^1(V_{n+1})\leq 3$. By Lemma \ref{lem:containment} we have that $\im \partial_n\subseteq \im \partial_{n+1}$ and hence $\rk(\im \partial_{n+1})\geq 1$. This implies that $\rk(\ker \partial_{n+1})\leq 2$, which completes the induction step and shows that 
    \[
        \rk H^1(V_p)=\rk \Cl(L_x)/p \leq 3\,.  
    \]

Conversely, suppose that all 3-fold Massey products vanish, which implies that $\rk H^1(V_3)=4$. So we are done if $p=3$. 

Now assume that $p>3$. Then Lemma \ref{lem:4-vanish} implies that $\partial_3=0$ and hence $\rk H^1(V_4) = 5$. 
We again use induction, this time to prove that $\rk(H^1(V_n))\geq 5$ for every $4\leq n<p$ and then that $\rk(H^1(V_p))\geq 4$. Let $4\leq n < p$ be an integer such that $\rk H^1(V_n)\geq 5$. 
First assume that $n<p-1$. By Lemma \ref{lem:non-surj}, the map $\partial_n$ is not surjective and hence $\rk H^1(V_{n+1})\geq 5$. 
If $n=p-1$, then $\rk H^1(V_{n+1})\geq 4$ since $\rk H^2(I^{p-1}/I^p)=2$. This shows that $\rk_{\FF_p} \Cl(L_x)=\rk H^1(V_p)\geq 4$, which completes the proof.  
\end{proof}

\begin{remark}
    Here is another way of looking at the fact that the vanishing of $\langle x,x,x\rangle$ and $\langle x,x,y\rangle$ implies that the $p$-rank of $\Cl(L_x)$ is at least 4. For simplicity we will assume that $p>3$. If $\langle x,x,y\rangle$ vanishes then, by Lemma \ref{lem:4-vanish}, we have $\im \partial_3=0$ and there exists a representation $G_K^{ur}\to GL_5(\FF_p)$ of the form
    \begin{equation}\label{eq:rep-1}
        \begin{pmatrix}
            1 & x & t_{x,1} & t_{x,2}  & t_{y,3} \\
             & 1 & x & t_{x,1} & t_{y,2}\\
             &  & 1 & x & t_{y,1} \\
             &  &   & 1 & y \\
             &  &   &   & 1 
        \end{pmatrix}    \,.
    \end{equation}
    Now let $G_{\QQ,S}$ be the Galois group of the maximal extension of $\QQ$ which is unramified outside of the set of places $S$ which ramify in $K$. Then we have a split exact sequence 
        \[
            0\to G_{K,S}\to G_{\QQ, S}\to G_{K/\QQ}\to 0    
        \] 
    and $G_{K/\QQ}$ acts on $G_{K,S}$ by conjugation. This action induces an action on the $\FF_p$-vector space $H^1(G_{K,S}, \ZZ/p\ZZ)$ via a character $\chi\colon G_{K/\QQ}\to (\ZZ/p\ZZ)^\times$ and by abuse of notation we also write $\chi$ for the induced character $G_{\QQ,S}\to (\ZZ/p\ZZ)^\times$. We define the $G_{\QQ, S}$-module $\ZZ/p\ZZ(\chi)$ to be $\ZZ/p\ZZ$ with $G_{\QQ, S}$ acting via $\chi$ and we may extend any homomorphism $z\colon G_{K,S}\to \ZZ/p\ZZ$ to a crossed homomorphism $\tilde{z}\colon G_{S,\QQ}\cong G_{K, S}\rtimes G_{K/\QQ}\to \ZZ/p\ZZ(\chi)$ by putting $\tilde{z}(a,\sigma)=z(a)$. By abuse of notation we also write $x,y, t_{y,1}$, etc., for the restrictions to $G_{K,S}$. 

    Since $p$ is odd, we get for every $n\in \NN$ that $H^n(G_{K/\QQ}, \ZZ/p\ZZ(\chi^n))=0$ and we have an isomorphism 
        \[
            H^n(G_{\QQ,S}, \ZZ/p\ZZ(\chi^n))\cong H^n(G_{K,S}, \ZZ/p\ZZ)\,.    
        \]  
    Hence the vanishing of $\langle x,x,y\rangle \in H^1(G_K^{ur}, \ZZ/p\ZZ)$ implies the vanishing of 
        \[
            \langle \tilde{x},\tilde{x},\tilde{y}\rangle \in H^1(G_{\QQ,S}, \ZZ/p\ZZ(\chi^3))\,.
        \]
    This means that the vanishing of $\langle x,x,y\rangle$ enables us to choose a representation $G_{\QQ,S}\to GL_5(\FF_p)$ of the form
    \begin{equation}\label{eq:rep-2}
        \begin{pmatrix}
            1   & \chi\tilde{x} & t_{\tilde{x},1} & \chi t_{\tilde{x},2}  & t_{\tilde{y},3} \\
                & \chi          & \tilde{x} & \chi t_{\tilde{x},1} & t_{\tilde{y},2}\\
                &               & 1         & \chi \tilde{x}            & t_{\tilde{y},1} \\
                &               &           & \chi                      & \tilde{y} \\
                &               &           &                           & 1 
        \end{pmatrix}    \,,
    \end{equation}  
    where $t_{\tilde{z}, i}(a,\sigma)=t_{z, i}(a)$ for $z=x$ or $z=y$ and $1\leq i\leq 3$. Note here that $\chi^2=1$ and that $x, t_{x,1}$, and $t_{x,2}$ are all polynomials in $x$ since they are obtained by iterating the matrix in (\ref{eq:one-matrix}). When restricting the representation to $G_{L_x, S}$, it takes the form 
    \begin{equation}\label{eq:rep-3}
        \begin{pmatrix}
            1   & 0 & 0 & 0  & t_{y,3} \\
                & 1          & 0 & 0 & t_{y,2}\\
                &               & 1         & 0           & t_{y,1} \\
                &               &           & 1                      & y \\
                &               &           &                           & 1 
        \end{pmatrix}    \,.
    \end{equation}  
    If we restrict $y, t_{y,1}, t_{y,2}, t_{y,3}$ to $G_{L_x, S}$, then these become homomorphisms and hence represent elements in $H^1(G_{L_x,S}, \ZZ/p\ZZ)$. We claim that $y, t_{y,1}, t_{y,2}, t_{y,3}$ are linearly independent as elements of $H^1(G_{L_x, S}, \ZZ/p\ZZ)$ which, in particular, implies that the representation restricts to a surjection $G_{L_x, S}\to (\ZZ/p\ZZ)^4$. Since the classes $y, t_{y,1}, t_{y,2}, t_{y,3}$ lift to $G_{L_x}^{ur}$, this also says that $\Cl(L_x)$ has $p$-rank at least 4. 
    
    The classes in $H^1(G_K^{ur}, V_p)$ given by 
        \[
            e_1=\begin{pmatrix}
                y \\
                0 \\
                0 \\
                0 \\
                \vdots \\
                0
            \end{pmatrix}, 
            e_2=\begin{pmatrix}
                t_{y,1} \\
                y \\
                0 \\
                0 \\
                \vdots \\
                0
            \end{pmatrix},  
            e_3=\begin{pmatrix}
                t_{y,2} \\
                t_{y,1} \\
                y \\
                0 \\
                \vdots \\
                0
            \end{pmatrix},   
            e_4=\begin{pmatrix}
                t_{y,3} \\
                t_{y,2} \\
                t_{y,1} \\
                y \\
                \vdots \\
                0
            \end{pmatrix} \in H^1(V_p)
        \]
    are linearly independent. This can be seen as follows: $e_1$ is the restriction of $y\in H^1(\ZZ/p\ZZ)$, which is not zero since this would imply that $y$ and $x$ are not linearly independent. This follows from the fact that, for every $n\geq 2$, we have an exact sequence 
        \[
            \dots \to H^0(V_{n-1})\to H^1(I^{n-1}/I^n)\to H^1(V_n)\to H^1(V_{n-1})\to \dots    
        \]
    and the quotient $H^1(I^{n-1}/I^n)/H^0(V_{n-1})$ is 1-dimensional and generated by the image of $y$.  
    The vector $e_1$ maps to zero in $H^1(V_{p-1})$ so it is enough to show that the images of $e_2, e_3, e_4$ are linearly independent in $H^1(V_{p-1})$. But $e_2$ maps to zero in $H^1(V_{p-2})$ and is non-zero in $H^1(V_{p-1})$ since there, it is the restriction of $y$, so it is enough to show that the images of $e_3$ and $e_4$ are linearly independent in $H^1(V_{p-3})$. Repeating the above argument, we get that $e_3$ and $e_4$ are linearly independent and hence $e_1, e_2, e_3$, and $e_4$ are linearly independent. The canonical isomorphism $H^1(G_K^{ur}, V_p)\cong H^1(G_{L_x}, \ZZ/p\ZZ)$, induced by evaluation at 1: 
        \[
            V_p\cong \Map(G,\ZZ/p\ZZ)\xrightarrow{\ev_1} \ZZ/p\ZZ\,,  
        \] 
    sends $e_1, e_2, e_3$, and $e_4$ to $y, t_{y,1}-y, t_{y,2}-t_{y,1}+y$, and $t_{y,3}-t_{y,2}+t_{y,1}-y$ respectively, and hence we see that $y, t_{y,1}, t_{y,2}$, and $t_{y,3}$ are linearly independent.  
\end{remark}

\section{Class field towers} \label{sec:classfield}
In this section we use the formulas computed in Section \ref{sec:massey} and apply them to the question of the infinitude of class field towers of imaginary quadratic fields. For an odd prime $p$ we will give the first known examples of imaginary quadratic fields $K$ with infinite $p$-class field tower and $p$-rank two, and we will also provide counterexamples to McLeman's $(3,3)$-conjecture. 

Given a number field $K$ there is, by class field theory, a maximal unramified abelian $p$-extension $H_K\supseteq K$, which is called the Hilbert $p$-class field of $K$; the Hilbert $p$-class field is always of finite degree over $K$. By iterating this construction, we obtain a tower of number fields 
    \[
        K = H_K^0\subseteq H_K^1 := H_K\subseteq H_K^2 := H_{H_K} \subseteq   \dots    
    \]
We call this tower the Hilbert $p$-class field tower, and we define $l_p(K)$, the length of the Hilbert $p$-class field tower of $K$, to be the minimum $i$ for which $H_K^i = H_K^{i+1}$ if such an $i$ exists, and let $l_p(K) = \infty$ if there is no such $i$. Then the $p$-Hilbert class field tower problem asks whether this tower stabilizes or not, i.e., whether $l_p(K) < \infty$. Denote by $\Gal(K^{ur,p}/K)$ the maximal pro-$p$-quotient of $\Gal(K^{ur}/K),$ where the latter group is the Galois group of the maximal unramified extension of $K$. Then since any pro-$p$-group is solvable, the $p$-Hilbert class field tower stabilizes if and only if $\Gal(K^{ur,p}/K)$ is a finite group. In 1925 Furtwängler asked whether this tower stabilizes for every number field $K$. 
This was an open problem for almost 40 years until Golod--Shafarevich found a counterexample by using an argument via pro-$p$-group theory \cite{Golod--Shafarevich}. We present a modern version of the group theoretic theorem they used to answer Furtwängler's question. 

\begin{theorem}[{\cite[Theorem 7.20]{KochGalois}}] \label{thm:golod}
    Let $G$ be a pro-$p$-group and let 
        \[
            1 \to R \to F \to G \to 1    
        \]
    be a pro-$p$-presentation of $G$. Suppose that $F$ has $d$ generators as a pro-p-group and choose a generating set $\Gen(R)\subseteq F$ of $R$ as a normal subgroup of $F$, and let $r_k$ be the number of elements of $\Gen(R)$ that have depth $k$ with respect to the Zassenhaus filtration on $F$. Then, if $G$ is finite, we have the inequality 
        \[
            1-dt+\sum_{i=1}^{\infty} r_kt^k > 0   
        \] 
    for all $t\in (0,1)$. 
\end{theorem}

Here the \emph{Zassenhaus filtration}
    \[
        F=F_1\supseteq F_2\supseteq F_3\supseteq \dots    
    \]
is defined by
    \[
        F_n=\{f \in F: 1-f\in I^n\}    
    \]
where $I\subseteq \FF_p[F]$ is the augmentation ideal and the \emph{depth} of a generator $\rho_i\in F$ of $R\subseteq F$ is the greatest integer $n$ such that $\rho_i\in F_n$. By choosing a minimal presentation of $G$, one can assume that the sum above starts at 2 since \cite[Proposition 2.3]{Mcleman} says that a minimal presentation yields $r_1=0$. Then one can derive that if $G$ is finite, then $r > \frac{d^2}{4}$, which is the classical Golod--Shafarevich inequality. With this inequality one can show that, for example,  $\mathbb{Q}(\sqrt{-3 \cdot 5 \cdot 7 \cdot 11 \cdot 13 \cdot 17})$ has an infinite $2$-class field tower, thus answering Furtwängler's question in the negative.

We want to turn the power series $1-dt+\sum_{i=2}^{\infty} r_kt^k$ into an invariant of the pro-$p$-group $G$ by making the power series as small as possible, viewed as a function $(0,1)\to (0,\infty)$. This is done as follows:  
let 
    \[
        0\to R\to F\to G \to 0    
    \]
be a minimal presentation of the pro-$p$-group $G$. Define recursively $R_1=\emptyset$ and 
    \[
        R_n=R_{n-1}\cup \{\rho_{n,1},\dots\rho_{n, r_n}\}    
    \]
where $\{\rho_{n,1},\dots\rho_{n, r_n}\}$ is a minimal generating set for $RF_{n}/F_{n}$ as a normal subgroup of $F/F_{n}$. Then define the \emph{Zassenhaus polynomial} to be
    \[
        Z_G(t) = 1-dt+\sum_{i=2}^{\infty} r_kt^k\,.    
    \]

Note that almost all $r_k$ will vanish, and thus $Z_G(t)$ is a polynomial. Given a minimal presentation $0\to R\to F\to G \to 0 $ we have that $F$ and $R$ can be generated by 
    \[
            d(G) = \rk_{\FF_p} H^1(G, \ZZ/p\ZZ) \mbox{ and } 
            r(G) = \rk_{\FF_p} H^2(G,\ZZ/p\ZZ)
    \]
elements respectively \cite[4.2-4.3]{SerreGalois}. From now on, $G$ will always denote $\Gal(K^{ur,p}/K)$.  Shafarevich proved in 1963 that if $p$ is odd and $K$ is an imaginary quadratic field, then we have the equality
    \[
        r(G) = d(G)\,.     
    \]
To us, this result of Shafarevich will be important since it implies that we have isomorphisms
    \[
        \begin{split}
            H^1(G, \ZZ/p\ZZ) & \cong H^1(\Spec \OO_K, \ZZ/p\ZZ)\cong H^2(\Spec \OO_K, \mmu_p)^\sim\cong (\Cl(K)/p\Cl(K))^\sim \,, \\
            H^2(G, \ZZ/p\ZZ) & \cong H^2(\Spec \OO_K, \ZZ/p\ZZ)\cong H^1(\Spec \OO_K, \mmu_p)^\sim\cong (Z^1/B^1)^\sim\,,
        \end{split}
    \]
where 
    \[
        \begin{split}
            Z^1 & = \{(a,I)\in K^\times\oplus \Div(K): \divis(a)+pI=0\} \,, \\
            B^1 & = \{(b^{-p},\divis(b))\in K^\times\oplus \Div(K): b\in K^\times\}\,.
        \end{split}    
    \]
Indeed, we always have that $H^1(G,\ZZ/p\ZZ) \cong H^1(\Spec \OO_K,\ZZ/p\ZZ)$, and to prove that $H^2(G,\ZZ/p\ZZ) \cong H^2(\Spec \OO_K,\ZZ/p\ZZ)$, it suffices to show that 
    \[ 
        \rk_{\FF_p} H^1(\Spec \OO_K, \ZZ/p\ZZ) =  \rk_{\FF_p} H^2(\Spec \OO_K, \ZZ/p\ZZ).
    \]
To see this, note that there is a map $Z^1/B^1\to \Cl(K)[p]$ sending $(a,I)$ to the class represented by $I$, and this map is an isomorphism. Indeed, the kernel consists of elements of the form $(a, \divis(b))$, where $a=ub^{-p}$ for a unit $u$. But the units of $K$ are only $\pm 1$ and $-1=(-1)^p$ since $p$ is odd. Hence $(a, \divis(b))=((ub)^{-p}, \divis(ub))=0$ in $Z^1/B^1$ which proves that $Z^1/B^1\to \Cl(K)[p]$ is an isomorphism. Since $\rk_{\FF_p} \Cl K/p \Cl K = \rk_{\FF_p} \Cl K[p]$, the claimed equality of ranks follows.

If we now restrict the $p$-class field tower problem to $p$ odd and $K$ quadratic imaginary, we then have the following result which may be found in \cite{Mcleman}:
    
\begin{proposition}
    Let $p$ be an odd prime and $K$ an imaginary quadratic field. Denote by $l_p(K)$ the length of the $p$-class field tower of $K$. Then 
        \[
            l_p(K)=
            \begin{cases}
            0 & \mbox{ if }\ \rk_{\FF_p}(\Cl(K)/p\Cl(K))=0\,,\\
            1 & \mbox{ if }\ \rk_{\FF_p}(\Cl(K)/p\Cl(K))=1\,,\\
            ? & \mbox{ if }\ \rk_{\FF_p}(\Cl(K)/p\Cl(K))=2\,,\\
            \infty & \mbox{ if }\ \rk_{\FF_p}(\Cl(K)/p\Cl(K))\geq 3\,.
            \end{cases}
        \]
\end{proposition}

\begin{proof}
    The case $\rk_{\FF_p}(\Cl(K)/p\Cl(K))=0$ is obvious and the case $\rk_{\FF_p}(\Cl(K)/p\Cl(K))\geq 3$ follows from a theorem of Koch--Venkov \cite{Koch--Venkov} and Theorem \ref{thm:golod}. If $\rk_{\FF_p}(\Cl(K)/p\Cl(K))=1$ then $G$ can be generated by one element and hence $G$ is abelian. But this can only happen if $G$ is the Galois group of $H_K$ over $K$, i.e., if $l_p(K)=1$. 
\end{proof}

Hence we see that the only unresolved case is when $\rk_{\FF_p}(\Cl(K)/p\Cl(K))=2$, in which case the Zassenhaus polynomial takes the form  
    \[
        Z_G(t) = t^i+t^j-2t+1\,     
    \]
    for some integers $i,j >1$. In fact, by Koch--Venkov \cite{Koch--Venkov}, we have that $i,j \geq 3$.
McLeman refers to the pair $(i,j)$ as the \emph{Zassenhaus type} of $G$. If $\Gal(K^{ur,p}/K)$ should have any possibility to be finite, then, as explained at the end of \cite[Section 2]{Mcleman}, the Zassenhaus type can only be $(3,3), (3,5)$, or $(3,7)$, since otherwise, the Zassenhaus polynomial will have a zero in the interval $(0,1)$ and then $G$ would be infinite by the theorem of Golod--Shafarevich. McLeman conjectured in \cite{Mcleman} that the Zassenhaus type should govern whether the $p$-class tower is infinite or not. More precisely, he made the following conjecture, which we have found counterexamples to.

\begin{conj}[{\cite[Conjecture 2.9]{Mcleman}}]\label{conj:3,3}
Let $p > 3$ be a prime number and suppose that $K$ is an imaginary quadratic field with class group of $p$-rank two. Then the
$p$-class field tower of $K$ is finite if and only if it is of Zassenhaus type $(3,3)$. 
\end{conj}

We will disprove the following equivalent version of Conjecture \ref{conj:3,3}. For a proof that these two conjectures are equivalent, see \cite[Proposition 3.2]{Mcleman}.

\begin{conj}\label{prop:inv} 
    Let $K$ be an imaginary quadratic field with class group of $p$-rank two, where $p>3$ is a prime. Let $G = \Gal(K^{ur,p}/K)$ of $K$ and choose a basis $x,y$ for $H^1(G,\ZZ/p\ZZ)$. Then $G$ is finite if and only if the matrix 
        \[
            \begin{pmatrix}
                \tr_{\rho_1}\langle x,x,y\rangle & \tr_{\rho_1}\langle y,y,x\rangle \\
                \tr_{\rho_2}\langle x,x,y\rangle & \tr_{\rho_2}\langle y,y,x\rangle
            \end{pmatrix}
        \] 
    is invertible.     
\end{conj}
    
\begin{remark}
    The maps $\tr_{\rho_i}\colon H^2(G, \ZZ/p\ZZ)\to \ZZ/p\ZZ$ form a basis for $H^2(G, \ZZ/p\ZZ)^\sim$ and since $H^2(G, \ZZ/p\ZZ)\cong H^1(\Spec \OO_K,\mmu_p)^\sim$, we see that $\tr_{\rho_1}$ and $\tr_{\rho_2}$ are just given by evaluation on some basis $e_1,e_2\in H^1(\Spec \OO_K, \mmu_p)\cong \Cl(K)[p]$. Hence we may write 
    \[
        \begin{pmatrix}
            \tr_{\rho_1}\langle x,x,y\rangle & \tr_{\rho_1}\langle y,y,x\rangle \\
            \tr_{\rho_2}\langle x,x,y\rangle & \tr_{\rho_2}\langle y,y,x\rangle
        \end{pmatrix} = 
        \begin{pmatrix}
            \langle\langle x,x,y\rangle, e_1\rangle & \langle\langle y,y,x\rangle, e_1\rangle \\
            \langle\langle x,x,y\rangle, e_2\rangle & \langle\langle y,y,x\rangle, e_2\rangle
        \end{pmatrix}.
    \] 
\end{remark}

Furthermore, we will give examples of imaginary quadratic fields $K$ with infinite $p$-class field tower when the $p$-rank of $\Cl K$ is two, and $p$ is an odd prime, using the following theorem of McLeman.

\begin{theorem}[{\cite[Theorem 3.1]{Mcleman}}]\label{thm:vanish}
    Let $K$ be an imaginary quadratic field with class group of $p$-rank two, and let $G = \Gal(K^{ur,p}/K)$. Choose a basis $x,y$ for $H^1(G, \ZZ/p\ZZ)$ and suppose that $p > 3$. Then $K$ has infinite $p$-class field tower if the Massey products $\langle x,x,y\rangle$ and $\langle y,y,x\rangle$ both vanish. For $p=3$, we need in addition that the Massey products $\langle x,x,x\rangle$ and $\langle y,y,y\rangle$ both vanish. 
\end{theorem}

Let us refer to the matrix 
    \[
        \begin{pmatrix}
            \langle\langle x,x,y\rangle, e_1\rangle & \langle\langle y,y,x\rangle, e_1\rangle \\
            \langle\langle x,x,y\rangle, e_2\rangle & \langle\langle y,y,x\rangle, e_2\rangle
        \end{pmatrix}   
    \]
as the \emph{Zassenhaus matrix} $ZM(G)$ and denote the Zassenhaus type by $ZT(G)$. Then to summarize, we may conclude that
    \[
        \begin{split}
            \rk \ZM(G) = 2 & \Leftrightarrow \ZT(G) = (3,3) \,, \\
            \rk \ZM(G) = 1 & \Rightarrow \ZT(G) = (3,5), (3,7), \mbox{ or } l_p(K)=\infty \,, \\
            \rk \ZM(G) = 0 & \Rightarrow l_p(K)=\infty \,.
        \end{split}    
    \]
We have used the C library PARI \cite{PARI2} to compute 3-fold Massey products in order to analyze the Zassenhaus types of imaginary quadratic fields $K$ with $p$-rank two. For the primes $3,5,7$, we have computed Massey products for imaginary quadratic fields with class group of $p$-rank two, except for $p=3$ where we add the stronger assumption that $9$ divides both factors of the $3$-class group to make sure that the Massey products $\langle x,x,x\rangle$ and $\langle y,y,y\rangle$ both vanish (see Remark \ref{rmk:vanish}). For $p=3$ there are $5163$ such fields and for $p=5$ there are $11070$ such fields.  
The program we use, together with computational results, can be found in the GitHub repository \url{https://github.com/ericahlqvist/massey}. 

\begin{theorem}\label{thm:false}
 The $(3,3)$-conjecture is false. For instance, for $p$ a prime and $D$ a discriminant, the pairs  
    \[
        (p,D) = (5, -90868), (7, -159592)  
    \] 
are counterexamples to the $(3,3)$-conjecture: for each pair in the above list, the associated quadratic imaginary field with discriminant $D$ has Zassenhaus type $(5,3)$ or $(7,3)$, but the $p$-class field tower is finite.
\end{theorem}

\begin{proof}
This is shown by using the program that can be found in the above GitHub repository, together with results from \cite{Mayer-Finite}. The computer program shows that one column of the Zassenhaus matrix vanishes, and thus, the Zassenhaus type is not of type $(3,3)$.
\end{proof}

\begin{theorem}\label{thm:infinite}
There exist odd prime numbers $p$ and imaginary quadratic fields of discriminant $D$ with class group of $p$-rank two and infinite $p$-class field tower. For instance, for prime $p$ and discriminant $D$, the pairs
    \[
        \begin{split}
            (p,D) = &\  (3, -3826859), (3, -8187139), (3, -11394591), (3, -13014563)\,, \\
                    &\  (5, -2724783), (5, -4190583), (5, -6741407), (5, -6965663) \,
        \end{split}
    \]
give examples of such fields: for each pair in the above list, the associated quadratic imaginary field with discriminant $D$ has an infinite $p$-class field tower.
\end{theorem}

\begin{proof}
This is shown by running the computer program that can be found in the above GitHub repository: the computation shows that the relevant Massey products vanish.
\end{proof}

\appendix
\section{Resolutions}\label{section:resolutions}
Let $G_x$ and $G_y$ be two groups abstractly isomorphic to $\ZZ/p\ZZ$. By choosing a generator of $G_x$ we identify $\FF_p[G_x]$, the group ring on $G_x$, with the ring $\FF_p[T_x]/(T_x^p-1)$ where $T_x$ is an indeterminate, and analogously $\FF_p[G_y]$ can be identified with $\FF_p[T_y]/(T_y^p-1)$. Obviously $\ZZ[G_x\times G_y]$ is then identified with $\FF_p[T_x,T_y]/(T_x^p-1,T_y^p-1)$. We let $I_x$ be the augmentation ideal $\FF_p[G_x]$ and $I_y$ the augmentation ideal of $\FF_p[G_y]$. We let 
$\varepsilon_x\colon \FF_p[G_x] \to \FF_p$ be the augmentation, and we will at times also use this symbol to denote the augmentation $\varepsilon_x\colon \ZZ[G_x] \to \ZZ$, and $\varepsilon_y$ will be used in the same manner, but with $G_y$ replacing $G_x$.

In this section, we will compute a resolution of the short exact sequence 
    \[
        0\to \FF_p[G_y]/I_y^2\xrightarrow{T_x-1}\FF_p[G_x\times G_y]/(I_y^2+I_x^2)\xrightarrow{\varepsilon_x}\FF_p[G_y]/I_y^2\to 0\,  
    \]
    into sufficiently free $\ZZ[G_x\times G_y]$-modules. All morphisms constructed will be morphisms of $\ZZ[G_x\times G_y]$-modules. 

\subsection*{A small non-free resolution of $\FF_p[G_y]/I_y^2$}
We have a surjection 
    $
        \begin{pmatrix}
            T_y-1 & 1
        \end{pmatrix}\colon 
        \ZZ\oplus \ZZ[G_y]\to \FF_p[G_y]/I_y^2
    $
and if $(T_y-1)a+b=0$ in $\FF_p[G_y]/I_y^2$, then $(T_y-1)a+b=(T_y-1)^2c+pd$ for some $c,d\in \ZZ[G_y]$. Thus $b=(1-T_y)(a+(1-T_y)c)+pd$ and we may write 
    \[
        \begin{pmatrix}
            a \\
            b
        \end{pmatrix}=
        \begin{pmatrix}
            0 & \varepsilon_y \\
            p & 1-T_y
        \end{pmatrix} 
        \begin{pmatrix}
            d \\
            a+(1-T_y)c
        \end{pmatrix}\,.   
    \]
Hence we have an exact sequence 
    \[
        \begin{matrix}
            \ZZ[G_y]^2
        \end{matrix}
        \xrightarrow{
            \begin{pmatrix}
                0 & \varepsilon_y \\
                p & 1-T_y
            \end{pmatrix}
        }        
        \begin{matrix}
            \ZZ \\ 
            \oplus \\
            \ZZ[G_y]
        \end{matrix}
        \xrightarrow{
            \begin{pmatrix}
                T_y-1 & 1
            \end{pmatrix}
        }
        \FF_p[G_y]/I_y^2\,.
    \]
If 
    \[
        \begin{pmatrix}
            0 & \varepsilon_y \\
            p & 1-T_y
        \end{pmatrix}
        \begin{pmatrix}
            a \\
            b
        \end{pmatrix}=
        \begin{pmatrix}
            \varepsilon_y(b)\\
            pa+(1-T_y)b
        \end{pmatrix}=
        \begin{pmatrix}
            0 \\ 
            0
        \end{pmatrix}
    \]
then $b= (1-T_y)b_1$ for some $b_1\in \ZZ[G_y]$ and $pa+(1-T_y)^2b_1=0$. But then $b_1$ lies in the kernel of $(1-T_y)^2$ modulo $p$ which is generated by the elements 
    \[
        \Delta_y=\sum_{n=0}^{p-1}T_y^n\,, \quad \Gamma_y=-\sum_{n=0}^{p-1}nT_y^n\,.
    \]
On the other hand, $\Gamma_y$ satisfies the relation    
    \[
        (1-T_y)\Gamma_y=p-\Delta_y    
    \]
and hence $-(1-T_y)\Gamma_y=\Delta_y$ modulo $p$ and hence $\ker (1-T_y)^2$ modulo $p$ is generated by $\Gamma_y$.
Hence $b_1=\Gamma_yb_2+pb_4$ and we get that $pa+p(1-T_y)b_2+p(1-T_y)^2b_4=0$. Since $p$ is integral, we conclude that 
    \[
        \begin{split}
            a & =-(1-T_y)(b_2+(1-T_y)b_4)\,,  \\
            b & = (p-\Delta_y)(b_2+(1-T_y)b_4)\,.
        \end{split}
    \]    
Hence we have a partial resolution 
    \[
        \ZZ[G_y]
        \xrightarrow{
            \begin{pmatrix}
                T_y-1 \\
                p-\Delta_y
            \end{pmatrix}
        }    
        \begin{matrix}
            \ZZ[G_y]^2
        \end{matrix}
        \xrightarrow{
            \begin{pmatrix}
                0 & \varepsilon_y \\
                p & 1-T_y
            \end{pmatrix}
        }        
        \begin{matrix}
            \ZZ \\ 
            \oplus \\
            \ZZ[G_y]
        \end{matrix}\,.
    \]

\subsection*{A non-free resolution of $\FF_p[G_x\times G_y]/(I_y^2+I_x^2)$}
We have a surjection 
    \[
        \begin{pmatrix}
            T_x-1 & T_y-1 & 1
        \end{pmatrix}\colon 
        \ZZ[G_y]\oplus \ZZ[G_x]\oplus \ZZ[G_x\times G_y]\to \FF_p[G_x\times G_y]/(I_y^2+I_x^2)
    \]
and if $(T_x-1)a+(T_y-1)b+c=0$ in $\FF_p[G_x\times G_y]/(I_y^2+I_x^2)$, then $(T_x-1)a+(T_y-1)b+c=(T_x-1)^2d+(T_y-1)^2e+pf$ for some $d,e,f\in \ZZ[G_x\times G_y]$. Hence 
    \[
        c = (1-T_x)(a+(1-T_x)d) + (1-T_y)(b+(1-T_y)e) + pf    
    \]
and we have an exact sequence 
    \[
        \begin{matrix}
            \ZZ[G_x\times G_y]^3
        \end{matrix}
        \xrightarrow{
            \begin{pmatrix}
                0 & \varepsilon_x & 0 \\
                0 & 0 & \varepsilon_y \\
                p & 1-T_x & 1-T_y
            \end{pmatrix}
        }
        \begin{matrix}
            \ZZ[G_y]\\
            \oplus \\
            \ZZ[G_x] \\ 
            \oplus \\
            \ZZ[G_x\times G_y]
        \end{matrix}    
        \xrightarrow{
            \begin{pmatrix}
                T_x-1 & T_y-1 & 1
            \end{pmatrix}
        }
        \ZZ[G_x\times G_y]/(I_y^2+I_x^2)\,.
    \]
If 
    \[
        \begin{pmatrix}
            0 & \varepsilon_x & 0 \\
            0 & 0 & \varepsilon_y \\
            p & 1-T_x & 1-T_y
        \end{pmatrix}
        \begin{pmatrix}
            a\\
            b\\
            c
        \end{pmatrix}
        =
        \begin{pmatrix}
            \varepsilon_x(b)\\
            \varepsilon_y(c)\\
            pa+(1-T_x)b+(1-T_y)c
        \end{pmatrix}
        =
        \begin{pmatrix}
            0\\
            0\\
            0
        \end{pmatrix}
    \]
then we get that $b=(1-T_x)b_1$, $c=(1-T_y)c_1$, and $pa+(1-T_x)^2b_1+(1-T_y)^2c_1=0$. Applying $\varepsilon_y \varepsilon_x$ we see that $p\varepsilon_y \varepsilon_xa=0$ and hence $\varepsilon_y \varepsilon_xa=0$. This means that 
    \[
        a=(1-T_y)a_1+(1-T_x)a_2    
    \]
for some $a_1,a_2\in \ZZ[G_x\times G_y]$ and we get 
    \begin{equation}\label{eq:a1}
        (1-T_x)(pa_2+(1-T_x)b_1)+(1-T_y)(pa_1+(1-T_y)c_1)=0\,.
    \end{equation}
We have that 
    \[
        (1-T_y)\Gamma_y=p-\Delta_y    
    \]
and hence $(1-T_y)\Gamma_y=p$ modulo $\Delta_y$.
If we apply $\varepsilon_x$ to Equation (\ref{eq:a1}) we get that $p\varepsilon_x(a_1)+(1-T_y)\varepsilon_x(c_1)=0$ modulo $\Delta_y$. Hence $\varepsilon_x(c_1)=-\Gamma_y\varepsilon_x(a_1)+\Delta_y\varepsilon(c')$ for some $c'\in \ZZ[G_x\times G_y]$ and 
    \[
        \begin{split}
            0   & = p\varepsilon_x(a_1)+(1-T_y)\varepsilon_x(c_1) \\
                & = p\varepsilon_x(a_1)-(p-\Delta_y)\varepsilon_x(a_1)+\Delta_y\varepsilon_x(c')\\
                & = \Delta_y\varepsilon_x(a_1+c')\,.
        \end{split} 
    \]
Hence $\varepsilon_x(c')=-\varepsilon_x(a_1)+(1-T_y)\varepsilon(a_3)$ for some $a_3\in \ZZ[G_x\times G_y]$ and we get 
    \[
        c_1 = -\Gamma_ya_1-\Delta_ya_1+(1-T_x)c_2\,.    
    \]
By symmetry we get 
    \[
        b_1 = -\Gamma_xa_2-\Delta_xa_2+(1-T_y)b_2    
    \]
for some $b_2\in \ZZ[G_x\times G_y]$. If we put the new expressions for $b_1$ and $c_1$ into Equation (\ref{eq:a1}) we get 
    \[
        (1-T_y)(1-T_x)((1-T_x)b_2+(1-T_y)c_2)=0\,.    
    \]
The kernel of $(1-T_y)(1-T_x)$ is generated by $\Delta_y$ and $\Delta_x$ so we get that 
    \begin{equation}\label{eq:a2}
        (1-T_x)b_2+(1-T_y)c_1=\Delta_yb_3+\Delta_x c_3     
    \end{equation}
for some $b_3,c_3\in \ZZ[G_x\times G_y]$. Applying $\varepsilon_y\varepsilon_x$ we get that 
    \[
        c_3=-b_3+(1-T_x)b_4+(1-T_y)c_4    
    \]
for some $b_4,c_4\in \ZZ[G_x\times G_y]$. Putting this into Equation (\ref{eq:a2}) we get 
    \begin{equation}\label{eq:a3}
        (1-T_x)b_2+(1-T_y)c_2=(\Delta_y-\Delta_x)b_3+\Delta_x(1-T_y) c_4\,.     
    \end{equation}
Apply $\varepsilon_y$ to get that $(1-T_x)\varepsilon_y(b_2)=p\varepsilon_y(b_3)$ modulo $\Delta_x$ and arguing as before, we get that 
    \[
        b_2=\Gamma_xb_3+\Delta_xb_5+(1-T_y)c_5    
    \]
for some $b_5,c_5\in \ZZ[G_x\times G_y]$. Similarly, we may apply $\varepsilon_x$ to get that 
    \[
        c_2=-\Gamma_yb_3+\Delta_xc_4+\Delta_yc_6+(1-T_x)b_6    
    \]
for some $b_6,c_6\in \ZZ[G_x\times G_y]$. Put these new expressions for $b_2$ and $c_2$ into Equation (\ref{eq:a3}) to get that 
    \[
        (1-T_y)(1-T_x)(b_6+c_5)=0   
    \]
and hence $b_6=-c_5+\Delta_yb_7+\Delta_xc_7$ for some $b_7,c_7\in \ZZ[G_x\times G_y]$. Hence 
    $
        c_2 = -\Gamma_y b_3+\Delta_xc_4+\Delta_yc_6-(1-T_x)c_5+\Delta_y(1-T_x)b_7   
    $
and hence  
    \[
        \begin{split}
            b_1 & = -(\Gamma_x+\Delta_x)a_2+(1-T_y)\Gamma_xb_3+(1-T_y)\Delta_xb_5+(1-T_y)^c_5\,, \\
            c_1 & = -(\Gamma_y+\Delta_y)a_1-(1-T_x)\Gamma_yb_3-(1-T_x)^2c_5\,.
        \end{split}    
    \]
Finally, we conclude that 
    \[
        \begin{split}
            a & = (1-T_y)a_1+(1-T_x)a_2\,, \\
            b & = (\Delta_x-p)(a_2-(1-T_y)b_3)+(1-T_y)^2(1-T_x)c_5\,, \\
            c & = (\Delta_y-p)(a_1-(1-T_x)b_3)-(1-T_y)(1-T_x)^2c_5,, 
        \end{split}    
    \] 
which implies that we have a partial resolution 
    \[\resizebox{.9\hsize}{!}{$
        \ZZ[G_x\times G_y]^3
        \xrightarrow{
            \begin{pmatrix}
                1-T_y & 1-T_x & 0  \\
                0 & \Delta_x-p & (1-T_y)^2(1-T_x) \\
                \Delta_y-p & 0 & (1-T_y)(1-T_x)^2
            \end{pmatrix}
        } 
        \ZZ[G_x\times G_y]^3
        \xrightarrow{
            \begin{pmatrix}
                0 & \varepsilon_x & 0 \\
                0 & 0 & \varepsilon_y \\
                p & 1-T_x & 1-T_y
            \end{pmatrix}
        }
        \begin{matrix}
            \ZZ[G_y]\\
            \oplus \\
            \ZZ[G_x] \\ 
            \oplus \\
            \ZZ[G_x\times G_y]
        \end{matrix}\,.
    $}
    \]

\subsection*{A free resolution of $\FF_p[G_y]/I_y^2$}

We have a surjection 
    \[
        \varepsilon_x
        \colon 
        \ZZ[G_x\times G_y]\to \FF_p[G_y]/(I_y^2)
    \]
and if $\varepsilon_x(a)=0$ in $\FF_p[G_y]/I_y^2$, then $a=(1-T_y)^2b+pc+(1-T_x)d$ for some $b,c,d\in \ZZ[G_x\times G_y]$. Hence we have an exact sequence 
    \[
        \ZZ[G_x\times G_y]^3
        \xrightarrow{
            \begin{pmatrix}
                (1-T_y)^2 & p & 1-T_x
            \end{pmatrix}
        }    
        \ZZ[G_x\times G_y]
        \xrightarrow{
            \varepsilon_x
        }
        \FF_p[G_y]/(I_y^2)\,.
    \]
If $(1-T_y)^2b + pc + (1-T_x)d=0$ then $c=(1-T_y)c_1+(1-T_x)c_2$ for some $c_1,c_2\in \ZZ[G_x\times G_y]$ and we get 
    \[
        (1-T_y)((1-T_y)b+pc_1)+(1-T_x)(d+pc_2)=0\,.    
    \]
Applying $\varepsilon_y$ we see that $d=-pc_2+\Delta_xd_1+(1-T_y)d_2$ for some $d_1,d_2\in \ZZ[G_x\times G_y]$ and applying $\varepsilon_x$ we see that $b=-\Gamma_yc_1+\Delta_yb_1+(1-T_x)b_2$ for some $b_1,b_2\in \ZZ[G_x\times G_y]$. Inserting $b$ and $d$ we get 
    \[
        (1-T_y)(1-T_x)((1-T_y)b_2+d_2)=0    
    \]
and we may write 
    \[
        d_2=-(1-T_y)b_2+\Delta_yb_3+\Delta_xb4    
    \]
for some $b_3,b_4\in \ZZ[G_x\times G_y]$. Hence we get 
    \[
        \begin{split}
            b & = -\Gamma_yc_1+\Delta_yb_1+(1-T_x)b_2\,, \\
            c & = (1-T_y)c_1+(1-T_x)c_2\,, \\
            d & = -pc_2+\Delta_xd_1-(1-T_y)^2b_2+\Delta_x(1-T_y)b_3\,,
        \end{split}    
    \]
and we have a partial resolution 
    \[
        \adjustbox{scale=0.8}{
            $\ZZ[G_x\times G_y]^5
            \xrightarrow{
                \begin{pmatrix}
                    0 & -\Gamma_y & \Delta_y & 1-T_x & 0 \\
                    1-T_x & 1-T_y & 0 & 0 & 0 \\
                    -p & 0 & 0 & -(1-T_y)^2 & \Delta_x
                \end{pmatrix}
            }   
            \ZZ[G_x\times G_y]^3
            \xrightarrow{
                \begin{pmatrix}
                    (1-T_y)^2 & p & 1-T_x
                \end{pmatrix}
            }    
            \ZZ[G_x\times G_y]$\,.   
        } 
    \]

\subsection*{Resolution of the short exact sequence}

We may now put our three resolutions together to form a resolution of the short exact sequence 
\[
    0\to \FF_p[G_y]/I_y^2\xrightarrow{T_x-1}\FF_p[G_x\times G_y]/(I_y^2+I_x^2)\xrightarrow{\varepsilon_x}\FF_p[G_y]/I_y^2\to 0\,.    
\]
This looks as follows:

\begin{equation}\label{eq:full-resolution}
    \begin{tikzcd}
        \ZZ[G_x\times G_y]^5\ar{r}\ar{d}{\delta^{-1,-2}} & \ZZ[G_x\times G_y]^3\ar{r}\ar{d}{\delta^{0,-2}} & \ZZ[G_y]\ar{d}{\delta^{1,-2}} \\
        \ZZ[G_x\times G_y]^3 \ar{r}{\alpha_{-1}}\ar{d}{\delta^{-1,-1}} & \ZZ[G_x\times G_y]^3\ar{r}{\beta_{-1}}\ar{d}{\delta^{0,-1}} & \ZZ[G_y]^2\ar{d}{\delta^{1,-1}} \\
        \ZZ[G_x\times G_y]\ar{r}{\alpha_{0}}\ar{d}{\delta^{-1,0}} & 
        \begin{matrix}
            \ZZ[G_y] \\
            \oplus \\
            \ZZ[G_x] \\
            \oplus \\
            \ZZ[G_x\times G_y]
        \end{matrix}
        \ar{r}{\beta_0}\ar{d}{\delta^{0,0}} & 
        \begin{matrix}
            \ZZ\\
            \oplus \\
            \ZZ[G_y]
        \end{matrix}
        \ar{d}{\delta^{1,0}} \\
        \FF_p[G_y]/I_y^2\ar{r}{T_x-1} & \FF_p[G_x\times G_y]/(I_y^2+I_x^2)\ar{r}{\varepsilon_x} & \FF_p[G_y]/I_y^2
    \end{tikzcd}  
\end{equation}
where 
    \[
        \adjustbox{scale=0.75}{
            $\delta^{-1,-2}=\begin{pmatrix}
                0 & -\Gamma_y & \Delta_y & 1-T_x & 0 \\
                1-T_x & 1-T_y & 0 & 0 & 0 \\
                -p & 0 & 0 & -(1-T_y)^2 & \Delta_x
            \end{pmatrix}\,, \quad 
            \delta^{0,-2}=\begin{pmatrix}
                1-T_y & 1-T_x & 0  \\
                0 & \Delta_x-p & (1-T_y)^2(1-T_x) \\
                \Delta_y-p & 0 & -(1-T_y)(1-T_x)^2
            \end{pmatrix}\,, \quad 
            \delta^{1,-2}=\begin{pmatrix}
                T_y-1\\
                p-\Delta_y
            \end{pmatrix}\,,$
        }    
    \]
    \[
        \adjustbox{scale=0.75}{
            $\alpha_{-1}=\begin{pmatrix}
                0&T_x-1&0\\
                (1-T_y)^2&p&0\\
                (T_y-1)(1-T_x)&0&0
            \end{pmatrix}\,,
            \quad 
            \beta_{-1}=\begin{pmatrix}
                \varepsilon_x&0&0\\
                0&0&\varepsilon_x
            \end{pmatrix}\,,$
        }
    \]
    \[
        \adjustbox{scale=0.75}{
            $\delta^{-1,-1}=\begin{pmatrix}
                (1-T_y)^2 & p & 1-T_x
            \end{pmatrix}\,, \quad 
            \delta^{0,-1}=\begin{pmatrix}
                0 & \varepsilon_x & 0 \\
                0 & 0 & \varepsilon_y \\
                p & 1-T_x & 1-T_y
            \end{pmatrix}\,, \quad 
            \delta^{1,-1}=\begin{pmatrix}
                0&\varepsilon_y\\
                p&1-T_y
            \end{pmatrix}\,,$
        }    
    \]
    \[
        \adjustbox{scale=0.75}{
            $\alpha_0=\begin{pmatrix}
                \varepsilon_x\\
                0\\
                0
            \end{pmatrix}\,, 
            \quad
            \beta_0=\begin{pmatrix}
                0&\varepsilon_x&0\\
                0&0&\varepsilon_x
            \end{pmatrix}\,.$ 
        }  
    \]
    \[
        \adjustbox{scale=0.75}{
            $\delta^{-1,0}=\varepsilon_x\,, \quad 
            \delta^{0,0}=\begin{pmatrix}
                T_x-1&T_y-1&1
            \end{pmatrix}\,, \quad 
            \delta^{1,0}=\begin{pmatrix}
                T_y-1&1
            \end{pmatrix}\,.$  
        }
    \]
Note also that we have a quasi-isomorphism 
    \begin{equation}\label{eq:quasi-iso}
        \begin{tikzcd}[column sep=1.2in, row sep=0.6in, ampersand replacement=\&]
            \ZZ[G_x\times G_y]^5\ar{r}{
                \begin{pmatrix}
                    0 & -\varepsilon_x & 0 & 0 & 0
                \end{pmatrix}
            }\ar{d}[swap]{
                \begin{pmatrix}
                    0 & -\Gamma_y & \Delta_y & 1-T_x & 0 \\
                    1-T_x & 1-T_y & 0 & 0 & 0 \\
                    -p & 0 & 0 & -(1-T_y)^2 & \Delta_x
                \end{pmatrix}
            } \& \ZZ[G_y]\ar{d}{
                \begin{pmatrix}
                    T_y-1 \\
                    p-\Delta_y
                \end{pmatrix}
            } \\
            \ZZ[G_x\times G_y]^3 \ar{r}{
                \begin{pmatrix}
                    0 & \varepsilon_x & 0 \\
                    (1-T_y)\varepsilon_x & 0 & 0
                \end{pmatrix}
            }\ar{d}[swap]{
                \begin{pmatrix}
                    (1-T_y)^2 & p & 1-T_x
                \end{pmatrix}
            } \& \ZZ[G_y]^2\ar{d}{
                \begin{pmatrix}
                    0 & \varepsilon_y \\
                    p & 1-T_y
                \end{pmatrix}
            } \\
            \ZZ[G_x\times G_y]\ar{r}{
                \begin{pampmatrix}
                    0 \\ \varepsilon_x
                \end{pampmatrix}
            } \& 
            \begin{matrix}
                \ZZ \\
                \oplus \\
                \ZZ[G_y]
            \end{matrix}
        \end{tikzcd}
    \end{equation}

\subsection*{Free resolution}
As an alternative, we may resolve the sequence by free resolutions:

\begin{equation}\label{eq:full-free-resolution}
    \begin{tikzcd}
        \ZZ[G_x\times G_y]^5\ar{r}\ar{d}{\delta^{-1,-2}} & \ZZ[G_x\times G_y]^5\ar{r}\ar{d}{\delta^{0,-2}} & \ZZ[G_x\times G_y]^5\ar{d}{\delta^{1,-2}} \\
        \ZZ[G_x\times G_y]^3 \ar{r}{\alpha_{-1}}\ar{d}{\delta^{-1,-1}} & \ZZ[G_x\times G_y]^3\ar{r}{\beta_{-1}}\ar{d}{\delta^{0,-1}} & \ZZ[G_x\times G_y]^3\ar{d}{\delta^{1,-1}} \\
        \ZZ[G_x\times G_y]\ar{r}{\alpha_{0}}\ar{d}{\delta^{-1,0}} & 
        \ZZ[G_x\times G_y]
        \ar{r}{\beta_0}\ar{d}{\delta^{0,0}} & 
        \ZZ[G_x\times G_y]
        \ar{d}{\delta^{1,0}} \\
        \FF_p[G_y]/I_y^2\ar{r}{T_x-1} & \FF_p[G_x\times G_y]/(I_y^2+I_x^2)\ar{r}{\varepsilon_x} & \FF_p[G_y]/I_y^2
    \end{tikzcd}  
\end{equation}
where 
    \[
        \scalemath{0.75}{
            \delta^{1,-2}=\delta^{-1,-2}=\begin{pmatrix}
                1-T_y & 1-T_x & 0 & 0 & 0 \\
                0 & -p & 0 & -(1-T_y)^2 & \Delta_x \\
                -\Gamma_y & 0 & \Delta_y & 1-T_x & 0
            \end{pmatrix}\,, \quad 
            \delta^{0,-2}=\begin{pmatrix}
                1-T_y & 1-T_x & 0 & 0 & 0 \\
                0 & -\Gamma_x & \Delta_x & 0 & (1-T_y)^2 \\
                -\Gamma_y & 0 & 0 & \Delta_y & -(1-T_x)^2
            \end{pmatrix}\,, 
        }    
    \]
    \[
        \scalemath{0.75}{
            \alpha_{-1}=\begin{pmatrix}
                T_x-1&0&0\\
                0&-1&0\\
                0&0&T_x-1
            \end{pmatrix}\,,
            \quad 
            \beta_{-1}=\begin{pmatrix}
                1&0&0\\
                0&1-T_x&0\\
                0&0&1
            \end{pmatrix}\,,
        }
    \]
    \[
        \scalemath{0.75}{
            \delta^{1,-1}=\delta^{-1,-1}=\begin{pmatrix}
                p & 1-T_x & (T_y-1)^2
            \end{pmatrix}\,, \quad 
            \delta^{0,-1}=\begin{pmatrix}
                p & (T_x-1)^2 & (T_y-1)^2
            \end{pmatrix}\,, 
        }    
    \]
    \[
        \scalemath{0.75}{
            \alpha_0=T_x-1\,, 
            \quad
            \beta_0=1\,.
        }  
    \]
    \[
        \scalemath{0.75}{
            \delta^{1,0}=\delta^{-1,0}=\varepsilon_x\,, \quad 
            \delta^{0,0}=1\,.  
        }
    \]

We need to show that the middle column is indeed exact. We have a surjection 
\[
    \ZZ[G_x\times G_y]\to \FF_p[G_x\times G_y]/(I_y^2+I_x^2)
\]
and if $a=0$ in $\FF_p[G_y]/I_y^2$, then $a=(1-T_y)^2b+(1-T_x)^2c+pd$ for some $b,c,d\in \ZZ[G_x\times G_y]$. Hence we have an exact sequence 
\[
    \ZZ[G_x\times G_y]^3
    \xrightarrow{
        \begin{pmatrix}
            p & (1-T_x)^2 & (1-T_y)^2
        \end{pmatrix}
    }    
    \ZZ[G_x\times G_y]
    \to
    \FF_p[G_x\times G_y]/(I_y^2+I_x^2)\,.
\]
Now suppose that $pa+(1-T_x)^2b+(1-T_y)^2c=0$ for some $a,b,c\in \ZZ[G_x\times G_y]$. Then $a=(1-T_y)a_1+(1-T_x)a_2$ for some $a_1,a_2\in \ZZ[G_x\times G_y]$ and we get 
    \begin{equation}\label{eq:a5}
        (1-T_x)(pa_2+(1-T_x)b)+(1-T_y)(pa_1+(1-T_y)c)=0\,.    
    \end{equation}
Applying $\varepsilon_x$ we get $p\varepsilon_x(a_1)+(1-T_y)\varepsilon_x(c)=0$ modulo $\Delta_y$ and hence 
    \[
        c=-\Gamma_y a_1+\Delta_yc_1+(1-T_x)c_2
    \]
for some $c_1,c_2\in \ZZ[G_x\times G_y]$. By symmetry
    \[
        b=-\Gamma_x a_2+\Delta_xb_1+(1-T_y)b_2
    \]
for some $b_1,b_2\in \ZZ[G_x\times G_y]$. Putting this into Equation (\ref{eq:a5}) we get 
    \[
        (1-T_y)(1-T_x)((1-T_x)b_2+(1-T_y)c_2)=0
    \]
and hence $(1-T_x)b_2+(1-T_y)c_2=\Delta_yb_3+\Delta_xc_3$ for some $b_3,c_3\in \ZZ[G_x\times G_y]$. Apply $\varepsilon_y\varepsilon_x$ to see that $c_3=-b_3+(1-T_x)b_4+(1-T_y)c_4$ for some $b_4,c_4\in \ZZ[G_x\times G_y]$. Hence we have 
    \begin{equation}\label{eq:a6}
        (1-T_x)b_2+(1-T_y)c_2=(\Delta_y-\Delta_x)b_3+(1-T_y)\Delta_xc_4\,.    
    \end{equation}
Apply $\varepsilon_y$ and $\varepsilon_x$ respectively to see that 
    \[
        \begin{split}
            b_2 & = \Gamma_xb_3+\Delta_xb_5+(1-T_y)c_5\,, \\
            c_2 & = -\Gamma_yb_3+\Delta_xc_4+\Delta_yc_6+(1-T_x)b_6
        \end{split}    
    \]
for some $b_5,b_6,c_5,c_6\in \ZZ[G_x\times G_y]$. Put this into Equation (\ref{eq:a6}) and we get $(1-T_y)(1-T_x)(b_6+c_5)=0$. Hence 
    \[
        b_6=-c_5+\Delta_yb_7+\Delta_xc_7
    \]
for some $b_7,c_7\in \ZZ[G_x\times G_y]$. Putting it all together we get 
    \[
        \begin{split}
            a & = (1-T_y)a_1+(1-T_x)a_2\,, \\
            b & = -\Gamma_x(a_2-(1-T_y)b_3)+\Delta_x(b_1+(1-T_y)b_5)+(1-T_y)^2c_5\,, \\
            c & = -\Gamma_y(a_1+(1-T_x)b_3)+\Delta_y(c_1+(1-T_x)c_6+(1-T_x)^2b_7)-(1-T_x)^2c_5\,.
        \end{split}    
    \]
After removing unnecessary columns we get a partial resolution 
    \[
        \scalemath{0.8}{
            \ZZ[G_x\times G_y]^5
            \xrightarrow{
                \begin{pmatrix}
                    1-T_y & 1-T_x & 0 & 0 & 0 \\
                    0 & -\Gamma_x & \Delta_x & 0 & (1-T_y)^2 \\
                    -\Gamma_y & 0 & 0 & \Delta_y & -(1-T_x)^2
                \end{pmatrix}
            }   
            \ZZ[G_x\times G_y]^3
            \xrightarrow{
                \begin{pmatrix}
                    p & (1-T_x)^2 & (1-T_y)^2
                \end{pmatrix}
            }    
            \ZZ[G_x\times G_y]\,.   
        } 
    \]
\section{Cohomology and duality for locally constant sheaves}\label{sec:spectral}
Let $X$ be the spectrum of the ring of integers of a number field $K$ and let $A$ be a locally constant constructible abelian sheaf on $X_{\et}$. In this section, we will show how one may compute the cohomology groups $H^i(X, D(A))$ in degrees 0,1, and 2, where $D(A)=R\HOM(A, \GG_m)$. By \cite[II.3.9]{MilneEtale}, we have an exact sequence 
    \[
        0\to \GG_{m,X} \to j_*\GG_{m,K}\to \DIV_X \to 0    
    \]
where $j\colon \Spec K\to X$ is the generic point. We denote the complex $j_*\GG_{m,K}\to \DIV_X$, concentrated in degree 0 and 1, by $\cC$. Choose a bounded resolution $\cE\to A$ by locally free, locally constant sheaves. By a locally free sheaf $E$ we mean that $E$ is, after pullback along a finite \'etale cover, a finite direct sum of copies of $\ZZ$. Since $\cE$ is a complex of locally free sheaves, we have that $\HOM(\cE, -)$ is quasi-isomorphic to $R\HOM(\cE, -)$. Hence $\HOM(\cE, \cC)$ is isomorphic to $D(\cE)$ in the derived category, and the Grothendieck spectral sequence for $R\Gamma\circ R\HOM(\cE, \cC)$ takes the form
    \[
        H^i(X, H^j(\HOM(\cE, \cC))) \Rightarrow H^{i+j}(R\Gamma(X, \HOM(\cE, \cC)))\cong H^{i+j}(X, D(A))\,.
    \]

\begin{proposition}\label{prop:edge}
    Let $X$ be the spectrum of the ring of integers of a number field and let $A$ be a locally constant constructible sheaf on $X$. The canonical map 
        \[
            H^i(\Gamma(X, \HOM(\cE, \cC)))\to H^i(X, D(A))
        \] 
    is an isomorphism for $i=0,1$, and $2$.  
\end{proposition}

\begin{proof}
    Without loss of generality, we may assume that every sheaf appearing in $\cE$ is a direct sum of sheaves of the form $\pi_*\pi^*\ZZ$ for some finite \'etale cover $\pi\colon Y\to X$, corresponding to a field extension $K\subseteq L$. Hence the $E_1$-page of the spectral sequence above will consist of groups that are direct sums of groups of the form $H^i(Y,j_*\GG_{m,L})$ or $H^i(Y, \DIV_Y)$. These are computed in \cite{MazurNotes} to be 
        \[
            H^i(Y, j_*\GG_{m,L}) = 
            \begin{cases}
                L^\times & i=0\,, \\
                0  & i = 1\,, \\
                \Br(L) & i = 2\,, \\
                0  & i \geq 3\,,
            \end{cases}   \quad
            H^i(Y, \DIV_Y) = 
            \begin{cases}
                \Div(Y) & i=0\,, \\
                0 & i = 1\,, \\
                \bigoplus_{p\in X}\QQ/\ZZ & i = 2\,, \\ 
                0  & i \geq 3\,
              
            \end{cases} 
        \]
    where $p$ ranges over all closed points of $Y$.
    This shows that the $E_1$-page is of the form 
        \[
            \begin{sseq}{0...3}{0...3}
                \ssdropbull
                \ssarrow{1}{0}
                \ssdropbull
                \ssarrow{1}{0}
                \ssdropbull
                \ssarrow{1}{0}
                \ssdropbull
                \ssmoveto{0}{2}
                \ssdropbull
                \ssarrow{1}{0}
                \ssdropbull
                \ssarrow{1}{0}
                \ssdropbull
                \ssarrow{1}{0}
                \ssdropbull
            \end{sseq}
        \]
    Without loss of generality, suppose that $\cE_0\cong \pi_*\pi^*\ZZ$ for a single finite \'etale cover $\pi\colon Y\to X$. Then the first horizontal map of the $i$th row will take the form 
        \[
            \begin{tikzcd}
                H^i(X, \HOM(\cE_0, j_*\GG_{m,K}))\ar{r}\ar[equal]{d} & H^i(X, \HOM(\cE_1, j_*\GG_{m,K})\oplus \HOM(\cE_0, \DIV_X))\ar[equal]{d} \\ 
                H^i(Y, j_*\GG_{m,L})\ar{r} & H^i(X, \HOM(\cE_1, j_*\GG_{m,K}))\oplus H^i(Y, \DIV_Y)
            \end{tikzcd}\,.
        \]
    In vertical degree $i=2$, the map $H^2(Y, j_*\GG_{m,L})\to H^2(Y, \DIV_Y)$ can be identified with the invariant map $\Br(L)\to \bigoplus_{p\in Y}\QQ/\ZZ$, which is known to be injective.    
    Hence the $E_2$-page takes the form 
        \[
            \begin{sseq}{0...3}{0...3}
                \ssdropbull
                \ssarrow{1}{0}
                \ssdropbull
                \ssarrow{1}{0}
                \ssdropbull
                \ssarrow{1}{0}
                \ssdropbull
                \ssmoveto{1}{2}
                \ssdropbull
                \ssarrow{1}{0}
                \ssdropbull
                \ssarrow{1}{0}
                \ssdropbull
            \end{sseq}
        \]
    We then immediately see that nothing except the terms on the bottom row can contribute to the target of the spectral sequence in degrees 0,1, and 2. This shows that we have an isomorphism $H^i(\Gamma(X, \HOM(\cE, \cC)))\to H^i(X, D(A))$ in degrees 0,1, and 2 as claimed. 
\end{proof}
\bibliographystyle{dary}
\bibliography{bibliography}
\end{document}